\documentclass[oneside]{amsart}  % Add [draft] to remove the pictures.
\usepackage{etex}

\usepackage{tikz}
\usetikzlibrary{calc,arrows,decorations.pathmorphing,backgrounds,fit}
\usetikzlibrary{patterns}
\usetikzlibrary{positioning}
\usepackage{ifthen}
\usepackage{amsmath}
\usepackage{amsfonts}
\usepackage{amssymb}
\usepackage{amsthm}
\usepackage{comment}
\usepackage{epsfig}
\usepackage{psfrag}
\usepackage{mathrsfs}
\usepackage{amscd}
\usepackage[all,rotate]{xy}
\usepackage{rotating}
\usepackage{lscape}
\usepackage{amsbsy}
\usepackage{verbatim}
\usepackage{moreverb}
\usepackage{xspace}
\usepackage{calc}
\usepackage{multicol}
\usepackage{manyfoot}
\usepackage{extpfeil}
\usepackage[pdftex, colorlinks=true, linkcolor=black, citecolor=black, urlcolor=black]{hyperref}
\usepackage{geometry}

\DeclareRobustCommand{\SkipTocEntry}[5]{}

\graphicspath{{../../../Archived-projects/Completed/IB-early/LQFTPix/FinalPix/miscPix/}{../../../Archived-projects/Completed/IB-early/LQFTPix/FinalPix/ci3Pix/}{../../../Archived-projects/Completed/IB-early/LQFTPix/FinalPix/cdcPix/}{../../../Archived-projects/Completed/IB-early/LQFTPix/FinalPix/cbcPix/}{../../../Archived-projects/Completed/IB-early/LQFTPix/FinalPix/c2cPix/newnewNumberScheme/}{../../../Archived-projects/Completed/IB-early/LQFTPix/FinalPix/MoritaPix/}{../../../Archived-projects/Completed/IB-early/LQFTPix/FinalPix/dualPix/}{../../../Archived-projects/Completed/IB-early/LQFTPix/FinalPix/triPix/}{../../../Archived-projects/Completed/IB-early/LQFTPix/FinalPix/dbPix/}}

\newcommand\ignore[1]{}

\newtheorem{theorem}{Theorem}[section]
\newtheorem{prop}[theorem]{Proposition}

\newtheorem{cor}[theorem]{Corollary}

\theoremstyle{definition}
\newtheorem{definition}[theorem]{Definition}

\theoremstyle{remark}
\newtheorem{remark}[theorem]{Remark}
\newtheorem{example}[theorem]{Example}

\newcommand{\nn}{\nonumber}
\newcommand{\nid}{\noindent}
\newcommand{\ra}{\rightarrow}

\newcommand{\xra}{\xrightarrow}

\newcommand{\ingeps}{\includegraphics[scale=1]}

\newcommand{\ing}{\includegraphics[scale=.3]}
\newcommand{\ingt}{\includegraphics[scale=.25]}
\newcommand{\ings}{\includegraphics[scale=.2]}

\newcommand{\xys}{\xymatrix@-15pt@M+5pt}
\newcommand{\arr}{\ar@{=>}}
\newcommand{\sct}{.4in}
\newcommand{\sca}{.5in}
\newcommand{\scb}{.6in}
\newcommand{\scc}{.7in}
\newcommand{\scd}{.8in}
\newcommand{\sce}{.9in}

\newcommand{\scg}{1.2in}

\newcommand{\hsqza}{-25pt}
\newcommand{\hsqzb}{-65pt}
\newcommand{\hscsqz}{.8}
\newcommand{\cboxht}{.6ex}
\newcommand{\hboxht}{.5ex}

\newcommand{\Hom}{\mathrm{Hom}}
\newcommand{\Bord}{\mathrm{Bord}}
\newcommand{\Vect}{\mathrm{Vect}}
\newcommand{\Hilb}{\mathrm{Hilb}}
\newcommand{\zCAT}{\mathrm{0CAT}}
\newcommand{\oCAT}{\mathrm{1CAT}}
\newcommand{\CAT}{\mathrm{CAT}}
\newcommand{\tCAT}{\mathrm{2CAT}}
\newcommand{\thCAT}{\mathrm{3CAT}}
\newcommand{\nCAT}{\mathrm{nCAT}}
\newcommand{\tCATnwk}{\mathrm{2CAT}_{\mathrm{nwk}}}
\newcommand{\tCATwk}{\mathrm{2CAT}_{\mathrm{wk}}}
\newcommand{\bBord}{\overline{\mathrm{Bord}}}
\newcommand{\SMC}{\mathrm{SMC}}
\newcommand{\Alg}{\mathrm{Alg}}

\newcommand{\vNalg}{\mathrm{vNalg}}

\newcommand{\TwVect}{\text{2-Vect}}
\newcommand{\ThVect}{\text{3-Vect}}

\newcommand{\Int}{\mathrm{Int}}
\newcommand{\CN}{\mathrm{CN}}
\newcommand{\TC}{\mathrm{TC}}

\newcommand{\dra}{\Rightarrow}
\newcommand{\Cat}{\mathrm{Cat}}

\newcommand{\obj}{\mathrm{obj}\,}
\newcommand{\mor}{\mathrm{mor}\,}
\newcommand{\fC}{\mathrm{C}}
\newcommand{\fN}{\mathrm{N}}
\newcommand{\fF}{\mathrm{F}}

\newcommand{\cb}{\raisebox{.6ex-.5\height}}
\newcommand{\db}{\raisebox{3ex-\height}}

\newcommand{\hb}{\raisebox{\hboxht}}
\newcommand{\bb}{\raisebox{2ex-.5\height}}

\newcommand{\cz}{C_0}
\newcommand{\cn}{C_1}

\newcommand{\bz}{B_0}
\newcommand{\bn}{B_1}
\newcommand{\bt}{B_2}

\newcommand{\tms}{\times}

\def\cmlg{4}
\def\cmsh{1}
\def\lbdist{1}
\def\lbgap{.2}
\def\bulletsize{.2}
\def\bull(#1,#2){
	\fill [color=black] (#1,#2) circle (\bulletsize);
}

\newcommand{\id}{\mathrm{id}}
\newcommand{\pr}{\mathrm{pr}}

\def\cA{\mathcal A}\def\cB{\mathcal B}\def\cC{\mathcal C}

\def\CC{\mathbb C}

\def\sectorh{4} % sector horizontal length
\def\dsectorh{7} % double sector h length
\def\tsectorh{10} % triple sector h length
\def\sectorv{3} % sector vertical length
\def\sectorvlow{1.65} % sector vertical length low
\def\sectorvmid{2.25} % sector vertical length mid
\def\sectorvhigh{4} % sector vertical length high
\def\dsectorvhigh{5.5} % vertical height of identity sectors of doulbe length in split situation.  Note, original was about 6.5, but because of different curve type, less is needed.
\def\tsectorvhigh{5.5} % vertical height of half identity sectors of triple length
\def\sectorhlong{5} % long sector length
\def\defecthsuplong{5} % defect super long

\def\sectorangle{60} % the angle at which the sector cookie begin.
\def\sectoranglelow{40}
\def\sectoranglemid{50}
\def\sectoranglehigh{65}
\def\dsectorangle{40}
\def\tsectorangle{28}
\def\dsectorhighangle{65}
\def\tsectorhighangle{65}
 
\def\vgap{.5} % distance between two sectors composed vertically
\def\hgap{.25} % distance associator ticks are away from picture
\def\exgap{.1} % extra gap when needed
\def\arcgap{.25} % gap at end of arc when an arrow arc is beginning there
\def\arcgapfudge{.1} % extension to cover the little end bits in the arc gap.
\def\arcgaparrowshiftv{.39} % vertical shift of caret placement in arrowarcgap
\def\arcgaparrowshifth{.24} % vertical shift of caret placement in arrowarcgap

\def\defecth{3} % defect (1-cell) horizontal length
\def\defectcgap{.5} % gap caused by caret in associator
\def\tickv{.75} % tick height
\def\tickvshort{.6} % short tick height
\def\tickvmini{.32} % mini tick height
\def\tickanglesharp{60} % angle of tick in a curving associator
\def\itickv{.5} % identity tick height
\def\itickvmini{.32} % mini id tick height
\def\tickvadd{.125} % extension of tick when near a curve
\def\iticksepv{.375} % identity tick separation for 2-cells
\def\itickseph{.36} % identity tick separation for 1-cells

\def\halfidsep{.375} % separation between lines in half identity sector
\def\halfidangle{45} % angle of half identity ticks.
\def\halfidangleshallow{35} % angle of half identity ticks when the main angle is too steep to fit.
\def\halfidanglesupshallow{18} % angle for the shallowest half id ticks
\def\hidticksangle{45} % angle for id ticks on an arc
\def\hidticksangleshallow{23} % angle for id ticks in triple sector
\def\hidshifth{1.27} % shift for id ticks on an arc % Note these are rotated shifts.
\def\hidshiftv{.18} % shift for id ticks on an arc
\def\hidshifthfact{1.21} % shift for id ticks on an arc when arc is \arcfactor stretched
\def\hidshiftvfact{.12} % shift for id ticks on an arc when arc is \arcfactor stretched
\def\hidshifttriph{1.6} % shift for id ticks on a triple arc % Note these are rotated shifts.
\def\hidshifttripv{.12} % shift for id ticks on an triple arc

\def\coordhhil{1.85}
\def\coordvhil{-.05}
\def\coordhhiltop{1.75}
\def\coordvhiltop{.5}
\def\coordhhilbotcut{2.6}
\def\coordvhilbotcut{-1.2}
\def\coordhhiltopthreecut{2}
\def\coordvhiltopthreecut{.75}
\def\coordhhilbotthreecut{2.6}
\def\coordvhilbotthreecut{-1.2}
\def\coordhhiltopthree{3}
\def\coordvhiltopthree{.5}
\def\coordhhilbot{3.5}
\def\coordvhilbot{-.75}
\def\coordhhilbotcutleft{3.35}%{5}
\def\coordvhilbotcutleft{-2.1}%{-3}

\def\ilength{2} % length of vertical identity lines
\def\ilengthd{4.5} % length of vertical identity lines in a two defect sector
\def\ilengtht{6.5} % length of vertical identity lines in a three defect sector
\def\ilengths{1.5} % length of vertical identity lines short
\def\assocv{1} % vertical height of associator

\def\caretangle{45} % caret angle
\def\caretlengthsect{.3} % caretlength in sectors
\def\caretlength{.375} % caretlength in arrows
\def\caretlengthmini{.2} % mini caretlength for tight spots

\def\arcfactor{.9} % factor by which arcs running along sectors are shrunk in vertical height.
\def\arcstretchfactor{1.1} % factor by which long arcs are stretched in vertical height.

\def\tickangle{13} % angle of tick in triple sector
\def\ttickh{1.31} % vertical position of tick and caret in triple sector, done by eye.
\def\dtickh{2.5} % vertical position of tick and caret in one exceptional half identity double width sector, namely hiltopthreecut.
\def\caretangle{11} %angle of caret in triple sector

\def\myinnersep{.15cm}
\def\myinnersepeps{.075cm}
\def\tableinnersep{.075cm}
\def\tableinnersepeps{.0375cm}
\def\linewidthvalue{.5985pt} % based on littlepictikzscale of .0315ex
\def\littlepicscale{.05ex}
\def\littlepicepsscale{.05ex}
\def\littlepictikzscale{.0315ex} % NB this is calibrated with respect to littlepicepsscale to produce the same size output.
\tikzset{littlepic/.style={scale=\littlepicscale, draw=black, line width=\linewidthvalue}}  %originally .08ex 
\tikzset{littlepiceps/.style={scale=\littlepicepsscale, draw=black, line width=\linewidthvalue}}  %originally .08ex 
\tikzset{littlepictikz/.style={scale=\littlepictikzscale, draw=black, line width=\linewidthvalue}}  %originally .08ex 
\tikzset{littlepicbull/.style={scale=.12ex, draw=black}}  %
\tikzset{littlenode/.style={fill=white,inner sep=\myinnersep}}
\tikzset{littlenode-eps/.style={fill=white,inner sep=\myinnersepeps}}
\tikzset{tablenode/.style={fill=white,inner sep=\tableinnersep}}
\tikzset{tablenode-eps/.style={fill=white,inner sep=\tableinnersepeps}}
\tikzset{displaypic/.style={scale=.15ex, draw=black}}

\def\innertikzeps{\tikz[littlepiceps]} % tricat eps
\def\innertikztikz{\tikz[littlepictikz]} % tricat tikz

\def\epspicscalex{.25ex}

\def\epspicscaledoubleplusx{.475ex}
\def\epspicscaledoublewidex{.6ex}
\def\epspicscaledoublesuperwidex{.7ex}

\def\epspicscaleyb{.20ex}
\def\epspicscaleyc{.25ex}
\def\epspicscaleycc{.27ex}
\def\epspicscaleyd{.30ex}
\def\tikztextscale{.15ex}

\def\tikztablescalex{.4ex}
\def\tikztablescaley{.05ex}
\def\tikztablescaleinner{.125ex}
\def\tikztablescaleinnerx{.35ex}
\def\tikztablescaleinnery{.125ex}

\def\sh(#1,#2){\def\x{#1}\def\y{#2}}

\def\sector{
	\coordinate (A) at (\x,\y);
	\coordinate (B) at ($ (A) + (.5*\sectorh,.5*\sectorv) $);
	\coordinate (C) at ($ (A) + (\sectorh,0) $);
	\coordinate (D) at ($ (A) + (.5*\sectorh,-.5*\sectorv) $);
	\draw (A) to [out=\sectorangle,in=180] (B) to [out=0,in=180-\sectorangle] (C);
	\draw (A) to [out=-\sectorangle,in=180] (D) to [out=0,in=180+\sectorangle] (C);
}

\def\arrowrarcu{
	\coordinate (A) at (\x,\y);
	\coordinate (B) at ($ (A) + (.5*\sectorh,\arcfactor*.5*\sectorv) $);
	\coordinate (C) at ($ (A) + (\sectorh,0) $);
	\draw (A) to [out=\sectorangle,in=180] (B) to [out=0,in=180-\sectorangle] (C);
	\caretr;
}

\def\arrowlarcu{
	\begin{scope}[cm={-1,0,0,1,(\x+\x+\sectorh,0)}]
		\arrowrarcu;
	\end{scope}
}

\def\arrowlarcd{
	\begin{scope}[cm={1,0,0,-1,(0,\y+\y)}]
		\arrowlarcu;
	\end{scope}
}

\def\arcu{
	\coordinate (A) at (\x,\y);
	\coordinate (B) at ($ (A) + (.5*\sectorh,\arcfactor*.5*\sectorv) $);
	\coordinate (C) at ($ (A) + (\sectorh,0) $);
	\draw (A) to [out=\sectorangle,in=180] (B) to [out=0,in=180-\sectorangle] (C);
}

\def\arcd{	
	\begin{scope}[cm={1,0,0,-1,(0,\y+\y)}]
		\arcu;
	\end{scope}
}

\def\arculong{
	\coordinate (A) at (\x,\y);
	\coordinate (B) at ($ (A) + (.5*\sectorh+.5,\arcstretchfactor*.5*\sectorv) $);
	\coordinate (C) at ($ (A) + (\sectorh+1,0) $);
	\draw (A) to [out=\sectorangle,in=180] (B) to [out=0,in=180-\sectorangle] (C);
}

\def\arcdlong{	
	\begin{scope}[cm={1,0,0,-1,(0,\y+\y)}]
		\arculong;
	\end{scope}
}

\def\arcugapl{
	\coordinate (D) at (\x-\arcgapfudge,\y-\arcgapfudge);
	\coordinate (E) at (\x+\arcgap,\y-\arcgapfudge);
	\coordinate (F) at (\x+\arcgap,\y+\arcgap+\arcgap+\arcgapfudge);
	\coordinate (G) at (\x-\arcgapfudge,\y+\arcgap+\arcgap+\arcgapfudge);
	\arcu;
	\draw[color=white,fill=white,line width=0pt] (D) -- (E) -- (F) -- (G) -- (D);
}

\def\arcugapr{
	\begin{scope}[cm={-1,0,0,1,(\x+\x+\sectorh,0)}]
		\arcugapl;
	\end{scope}
}

\def\arcdgapl{
	\begin{scope}[cm={1,0,0,-1,(0,\y+\y)}]
		\arcugapl;
	\end{scope}
}

\def\arcdgapr{
	\begin{scope}[cm={1,0,0,-1,(0,\y+\y)}]
		\arcugapr;
	\end{scope}
}

\def\arrowrarcugap{
	\arcugapl;
	\begin{scope}[shift={(\arcgaparrowshifth,\arcgaparrowshiftv)}]
		\caretr;
	\end{scope}
}

\def\arrowlarcugap{
	\begin{scope}[cm={-1,0,0,1,(\x+\x+\sectorh,0)}]
		\arrowrarcugap;
	\end{scope}
}

\def\arrowlarcdgap{
	\begin{scope}[cm={1,0,0,-1,(0,\y+\y)}]
		\arrowlarcugap;
	\end{scope}
}

\def\arrowrarcuturn{
	\arcugapl;
	\begin{scope}[shift={(.60,.57)}] %!% This is completely wrong, was tuned for a particular nonzero value of \y
	\begin{scope}[rotate=45]
		\caretrmini;
	\end{scope}
	\end{scope}
}

\def\arrowlarcuturn{
	\begin{scope}[cm={-1,0,0,1,(\x+\x+\sectorh,0)}]
		\arrowrarcuturn;
	\end{scope}
}

\def\arrowlarcdturn{
	\begin{scope}[cm={1,0,0,-1,(0,\y+\y)}]
		\arrowlarcuturn;
	\end{scope}
}

\def\splittopr{
	\coordinate (A) at (\x,\y);
	\coordinate (B) at ($ (A) + (.5*\sectorh,\arcfactor*.5*\sectorv) $);
	\coordinate (C) at ($ (A) + (\sectorh,0) $);
	\draw (A) to [out=\sectorangle,in=180] (B) to [out=0,in=180-\sectorangle] (C);
	\begin{scope}[shift={(.5*\sectorh,\arcfactor*.5*\sectorv)}]
		\caretrwhite;
	\end{scope}
}

\def\vid{
	\coordinate (A) at ($ (\x,\y) + (.5*\sectorh-.5*\ilength,-.5*\iticksepv) $);
	\coordinate (B) at ($ (\x,\y) + (.5*\sectorh+.5*\ilength,-.5*\iticksepv) $);
	\coordinate (C) at ($ (\x,\y) + (.5*\sectorh-.5*\ilength,.5*\iticksepv) $);
	\coordinate (D) at ($ (\x,\y) + (.5*\sectorh+.5*\ilength,.5*\iticksepv) $);
	\draw (A) to (B);
	\draw (C) to (D);
}

\def\vidd{
	\coordinate (A) at ($ (\x,\y) + (.5*\dsectorh-.5*\ilengthd,-.5*\iticksepv) $);
	\coordinate (B) at ($ (\x,\y) + (.5*\dsectorh+.5*\ilengthd,-.5*\iticksepv) $);
	\coordinate (C) at ($ (\x,\y) + (.5*\dsectorh-.5*\ilengthd,.5*\iticksepv) $);
	\coordinate (D) at ($ (\x,\y) + (.5*\dsectorh+.5*\ilengthd,.5*\iticksepv) $);
	\draw (A) to (B);
	\draw (C) to (D);
}

\def\vidt{
	\coordinate (A) at ($ (\x,\y) + (.5*\tsectorh-.5*\ilengtht,-.5*\iticksepv) $);
	\coordinate (B) at ($ (\x,\y) + (.5*\tsectorh+.5*\ilengtht,-.5*\iticksepv) $);
	\coordinate (C) at ($ (\x,\y) + (.5*\tsectorh-.5*\ilengtht,.5*\iticksepv) $);
	\coordinate (D) at ($ (\x,\y) + (.5*\tsectorh+.5*\ilengtht,.5*\iticksepv) $);
	\draw (A) to (B);
	\draw (C) to (D);
}

\def\vids{
	\coordinate (A) at ($ (\x,\y) + (.5*\sectorh-.5*\ilengths,-.5*\iticksepv) $);
	\coordinate (B) at ($ (\x,\y) + (.5*\sectorh+.5*\ilengths,-.5*\iticksepv) $);
	\coordinate (C) at ($ (\x,\y) + (.5*\sectorh-.5*\ilengths,.5*\iticksepv) $);
	\coordinate (D) at ($ (\x,\y) + (.5*\sectorh+.5*\ilengths,.5*\iticksepv) $);
	\draw (A) to (B);
	\draw (C) to (D);
}

\def\tripletopr{
	\coordinate (A) at (\x,\y);
	\coordinate (B) at ($ (A) + (.5*\tsectorh,.5*\sectorv) $);
	\coordinate (C) at ($ (A) + (\tsectorh,0) $);
	\draw (A) to [out=\tsectorangle,in=180] (B) to [out=0,in=180-\tsectorangle] (C);
	\begin{scope}[shift={(.31*\tsectorh,\ttickh)}]
		\tickcc;
	\end{scope}
	\begin{scope}[shift={(.69*\tsectorh,\ttickh)}]
		\caretrcwhite;
	\end{scope}
}

\def\tripletopl{
	\begin{scope}[cm={-1,0,0,1,(\x+\x+\tsectorh,0)}]
		\tripletopr;
	\end{scope}
}

\def\tripletoparc{
	\coordinate (A) at (\x,\y);
	\coordinate (B) at ($ (A) + (.5*\tsectorh,.5*\sectorv) $);
	\coordinate (C) at ($ (A) + (\tsectorh,0) $);
	\draw (A) to [out=\tsectorangle,in=180] (B) to [out=0,in=180-\tsectorangle] (C);
}

\def\triplebotr{
	\begin{scope}[cm={1,0,0,-1,(0,\y+\y)}]
		\tripletopr;
	\end{scope}
}

\def\triplebotl{
	\begin{scope}[cm={1,0,0,-1,(0,\y+\y)}]
		\tripletopl;
	\end{scope}
}

\def\triplebotarc{
	\begin{scope}[cm={1,0,0,-1,(0,\y+\y)}]
		\tripletoparc;
	\end{scope}
}

\def\tripletophigh{
	\coordinate (A) at (\x,\y);
	\coordinate (B) at ($ (A) + (.5*\tsectorh,.5*\tsectorvhigh) $);
	\coordinate (C) at ($ (A) + (\tsectorh,0) $);
	\draw (A) to [out=\tsectorhighangle,in=180] (B) to [out=0,in=180-\tsectorhighangle] (C);
}

\def\triplebothigh{
	\begin{scope}[cm={1,0,0,-1,(0,\y+\y)}]
		\tripletophigh;
	\end{scope}
}

\def\doubletop{
	\coordinate (A) at (\x,\y);
	\coordinate (B) at ($ (A) + (.5*\dsectorh,.5*\sectorv) $);
	\coordinate (C) at ($ (A) + (\dsectorh,0) $);
	\draw (A) to [out=\dsectorangle,in=180] (B) to [out=0,in=180-\dsectorangle] (C);
}

\def\doubletoptick{
	\doubletop;
	\begin{scope}[shift={(.5*\dsectorh,.5*\sectorv)}]
		\tick;
	\end{scope}
}

\def\doublebottick{
	\begin{scope}[cm={1,0,0,-1,(0,\y+\y)}]
		\doubletoptick;
	\end{scope}
}

\def\doubletopr{
	\doubletop;
	\begin{scope}[shift={(.5*\dsectorh,.5*\sectorv)}]
		\caretrwhite;
	\end{scope}
}

\def\doubletopl{
	\begin{scope}[cm={-1,0,0,1,(\x+\x+\dsectorh,0)}]
		\doubletopr;
	\end{scope}
}

\def\doublebotl{
	\begin{scope}[cm={1,0,0,-1,(0,\y+\y)}]
		\doubletopl;
	\end{scope}
}

\def\doubletophigh{
	\coordinate (A) at (\x,\y);
	\coordinate (B) at ($ (A) + (.5*\dsectorh,.5*\dsectorvhigh) $);
	\coordinate (C) at ($ (A) + (\dsectorh,0) $);
	\draw (A) to [out=\dsectorhighangle,in=180] (B) to [out=0,in=180-\dsectorhighangle] (C);
}

\def\doublebothigh{
	\begin{scope}[cm={1,0,0,-1,(0,\y+\y)}]
		\doubletophigh;
	\end{scope}
}

\def\caretrcwhite{
	\begin{scope}[rotate around={-\caretangle:(\x,\y)}]
		\caretrwhite;
	\end{scope}
}

\def\caretrwhite{
	\coordinate (A) at (\x,\y);
	\coordinate (B) at ($ (A) + (-\caretlength,\caretlength) $);
	\coordinate (C) at ($ (A) + (-\caretlength,-\caretlength) $);
	\coordinate (D) at ($ (A) + (-1.2*\caretlength,1.2*\caretlength) $);
	\coordinate (E) at ($ (A) + (-1.2*\caretlength,-1.2*\caretlength) $);
	\draw[color=white,fill=white,line width=0pt] (D) -- (A) -- (E) -- (D);
	\draw (B) to (A) to (C);
}

\def\caretr{
	\coordinate (A) at (\x,\y);
	\coordinate (B) at ($ (A) + (-\caretlength,\caretlength) $);
	\coordinate (C) at ($ (A) + (-\caretlength,-\caretlength) $);
	\draw (B) to (A) to (C);
}

\def\caretlwhite{
	\begin{scope}[cm={-1,0,0,1,(\x+\x+\caretlength,0)}]
		\caretrwhite;
	\end{scope}
}

\def\caretrmini{
	\coordinate (A) at (\x,\y);
	\coordinate (B) at ($ (A) + (-\caretlengthmini,\caretlengthmini) $);
	\coordinate (C) at ($ (A) + (-\caretlengthmini,-\caretlengthmini) $);
	\draw (B) to (A) to (C);
}

\def\tick{
	\coordinate (A) at ($ (\x,\y) + (0,.5*\tickv) $);
	\coordinate (B) at ($ (\x,\y) + (0,-.5*\tickv) $);
	\draw (A) to (B);
}

\def\tickaddup{
	\coordinate (A) at ($ (\x,\y) + (0,.5*\tickv+\tickvadd) $);
	\coordinate (B) at ($ (\x,\y) + (0,-.5*\tickv) $);
	\draw (A) to (B);
}

\def\tickadddown{
	\coordinate (A) at ($ (\x,\y) + (0,.5*\tickv) $);
	\coordinate (B) at ($ (\x,\y) + (0,-.5*\tickv-\tickvadd) $);
	\draw (A) to (B);
}

\def\tickshortupadddown{
	\coordinate (A) at ($ (\x,\y) + (0,.5*\tickv-\tickvadd) $);
	\coordinate (B) at ($ (\x,\y) + (0,-.5*\tickv-\tickvadd) $);
	\draw (A) to (B);
}

\def\tickshortdownaddup{
	\coordinate (A) at ($ (\x,\y) + (0,.5*\tickv+\tickvadd) $);
	\coordinate (B) at ($ (\x,\y) + (0,-.5*\tickv+\tickvadd) $);
	\draw (A) to (B);
}

\def\tickshort{
	\coordinate (A) at ($ (\x,\y) + (0,.5*\tickvshort) $);
	\coordinate (B) at ($ (\x,\y) + (0,-.5*\tickvshort) $);
	\draw (A) to (B);
}

\def\tickmini{
	\coordinate (A) at ($ (\x,\y) + (0,.5*\tickvmini) $);
	\coordinate (B) at ($ (\x,\y) + (0,-.5*\tickvmini) $);
	\draw (A) to (B);
}

\def\itickshere{
	\coordinate (A) at ($ (\x,\y) + (-.5*\itickseph,.5*\itickv) $);
	\coordinate (B) at ($ (\x,\y) + (-.5*\itickseph,-.5*\itickv) $);
	\coordinate (C) at ($ (\x,\y) + (.5*\itickseph,.5*\itickv) $);
	\coordinate (D) at ($ (\x,\y) + (.5*\itickseph,-.5*\itickv) $);	
	\draw (A) to (B);
	\draw (C) to (D);
}

\def\iticks{
	\coordinate (A) at ($ (\x,\y) + (.5*\defecth-.5*\itickseph,.5*\itickv) $);
	\coordinate (B) at ($ (\x,\y) + (.5*\defecth-.5*\itickseph,-.5*\itickv) $);
	\coordinate (C) at ($ (\x,\y) + (.5*\defecth+.5*\itickseph,.5*\itickv) $);
	\coordinate (D) at ($ (\x,\y) + (.5*\defecth+.5*\itickseph,-.5*\itickv) $);	
	\draw (A) to (B);
	\draw (C) to (D);
}

\def\iticksminihere{
	\coordinate (A) at ($ (\x,\y) + (-.5*\itickseph,.5*\itickvmini) $);
	\coordinate (B) at ($ (\x,\y) + (-.5*\itickseph,-.5*\itickvmini) $);
	\coordinate (C) at ($ (\x,\y) + (.5*\itickseph,.5*\itickvmini) $);
	\coordinate (D) at ($ (\x,\y) + (.5*\itickseph,-.5*\itickvmini) $);	
	\draw (A) to (B);
	\draw (C) to (D);
}

\def\iticksrotl{
	\begin{scope}[rotate around={\hidticksangle:(\x,\y)}]
	\begin{scope}[shift={(\hidshifth,\hidshiftv)}]
		\itickshere;
	\end{scope}
	\end{scope} %% Remark: the shifts end up being rotated; but if you try to do it the other way it doesn't work.
}

\def\iticksrotr{
	\begin{scope}[rotate around={-\hidticksangle:(\x,\y)}]
	\begin{scope}[shift={(.707*\sectorh-\hidshifth,.707*\sectorh+\hidshiftv)}]
		\itickshere;
	\end{scope}
	\end{scope} 
}

\def\iticksrotlminifact{
	\begin{scope}[rotate around={\hidticksangle:(\x,\y)}]
	\begin{scope}[shift={(\hidshifthfact,\hidshiftvfact)}]
		\iticksminihere;
	\end{scope}
	\end{scope} 
}

\def\iticksrotrminifact{
	\begin{scope}[rotate around={-\hidticksangle:(\x,\y)}]
	\begin{scope}[shift={(.707*\sectorh-\hidshifthfact,.707*\sectorh+\hidshiftvfact)}]
		\iticksminihere;
	\end{scope}
	\end{scope} 
}

\def\iticksrotlfact{
	\begin{scope}[rotate around={\hidticksangle:(\x,\y)}]
	\begin{scope}[shift={(\hidshifthfact,\hidshiftvfact)}]
		\itickshere;
	\end{scope}
	\end{scope} 
}

\def\iticksrotrfact{
	\begin{scope}[rotate around={-\hidticksangle:(\x,\y)}]
	\begin{scope}[shift={(.707*\sectorh-\hidshifthfact,.707*\sectorh+\hidshiftvfact)}]
		\itickshere;
	\end{scope}
	\end{scope} 
}

\def\itickstriprotlshallow{ %!%
	\begin{scope}[rotate around={\hidticksangleshallow:(\x,\y)}]
	\begin{scope}[shift={(\hidshifttriph,\hidshifttripv)}]
		\itickshere;
	\end{scope}
	\end{scope} 
}

\def\itickstriprotrshallow{ %!%
	\begin{scope}[rotate around={-\hidticksangleshallow:(\x,\y)}]
	\begin{scope}[shift={(-\hidshifttriph,\hidshifttripv)}]
		\itickshere;
	\end{scope}
	\end{scope} 
}

\def\tickcc{
	\begin{scope}[rotate around={\tickangle:(\x,\y)}]
		\tick;
	\end{scope}
}

\def\tickminicc{
	\begin{scope}[rotate around={\tickanglesharp:(\x,\y)}]
		\tickmini;
	\end{scope}
}

\def\defect{
	\coordinate (A) at (\x,\y);
	\coordinate (B) at ($ (A) + (\defecth,0) $);
	\draw (A) to (B);
}

\def\defectsemilong{
	\coordinate (A) at (\x,\y);
	\coordinate (B) at ($ (A) + (\defecth+.5,0) $); %!%
	\draw (A) to (B);
}

\def\defectlong{
	\coordinate (A) at (\x,\y);
	\coordinate (B) at ($ (A) + (\sectorh,0) $);
	\draw (A) to (B);
}

\def\defectsuplong{
	\coordinate (A) at (\x,\y);
	\coordinate (B) at ($ (A) + (\defecthsuplong,0) $);
	\draw (A) to (B);
}

\def\defectshortr{
	\coordinate (A) at (\x,\y);
	\coordinate (B) at ($ (A) + (\defecth-\defectcgap,0) $);
	\draw (A) to (B);
}

\def\defectshortl{
	\coordinate (A) at (\x+\defectcgap,\y);
	\coordinate (B) at ($ (A) + (\defecth-\defectcgap,0) $);
	\draw (A) to (B);
}

\def\defectshorterr{
	\coordinate (A) at (\x,\y);
	\coordinate (B) at ($ (A) + (\defecth-\defectcgap-\exgap,0) $);
	\draw (A) to (B);
}

\def\defectshorterl{
	\coordinate (A) at (\x+\defectcgap+\exgap,\y);
	\coordinate (B) at ($ (A) + (\defecth-\defectcgap-\exgap,0) $);
	\draw (A) to (B);
}

\def\strutv{
	\coordinate (A) at (\x,\y);
	\coordinate (B) at ($ (A) + (0,\assocv) $);
	\draw (A) to (B);
}

\def\strutlong{
	\coordinate (A) at (\x,\y);
	\coordinate (B) at ($ (A) + (0,\assocv+.5*\vgap) $);
	\draw (A) to (B);
}

\def\strutshort{
	\coordinate (A) at (\x,\y);
	\coordinate (B) at ($ (A) + (0,\assocv-.5*\vgap) $);
	\draw (A) to (B);
}

\def\strutshortdown{
	\coordinate (A) at (\x,\y);
	\coordinate (B) at ($ (A) - (0,\assocv-.5*\vgap) $);
	\draw (A) to (B);
}

\def\assoctickr{
	\coordinate (A) at (\x + \hgap,\y);
	\coordinate (B) at (\x + \hgap,\y + \vgap);
	\draw (A) to [out=45,in=-45] (B);
}

\def\assoctickl{
	\coordinate (A) at (\x - \hgap,\y);
	\coordinate (B) at (\x - \hgap,\y + \vgap);
	\draw (A) to [out=135,in=-135] (B);
}

\def\arrowr{
	\caretr;
	\defect;
}

\def\arrowrlong{
	\caretr;
	\defectlong;
}

\def\arrowl{
	\begin{scope}[cm={-1,0,0,1,(\x+\x+\defecth,0)}]
		\arrowr;
	\end{scope}
}

\def\arrowllong{
	\begin{scope}[cm={-1,0,0,1,(\x+\x+\sectorh,0)}]
		\arrowrlong;
	\end{scope}
}

\def\arrowrgap{
	\begin{scope}[cm={1,0,0,1,(\defectcgap,0)}]
		\caretr;
	\end{scope}
	\defectshortl;
}	

\def\arrowrgapbig{
	\begin{scope}[cm={1,0,0,1,(\defectcgap+\exgap,0)}]
		\caretr;
	\end{scope}
	\defectshorterl;
}

\def\arrowlgap{
	\begin{scope}[cm={-1,0,0,1,(\x+\x+\defecth-\defectcgap,0)}]
		\caretr;
	\end{scope}
	\defectshortr;
}

\def\arrowlgapbig{
	\begin{scope}[cm={-1,0,0,1,(\x+\x+\defecth-\defectcgap-\exgap,0)}]
		\caretr;
	\end{scope}
	\defectshorterr;
}

\def\sectorcaretr{
	\coordinate (A) at (\x,\y);
	\coordinate (B) at ($ (A) + (.5*\sectorh,.5*\sectorv) $);
	\coordinate (C) at ($ (A) + (\sectorh-\caretlengthsect,\caretlengthsect) $);
	\coordinate (CCC) at ($ (A) + (\sectorh-\caretlengthsect,-\caretlengthsect) $);
	\coordinate (CC) at ($ (A) + (\sectorh-2*\caretlengthsect,0) $);
	\coordinate (D) at ($ (A) + (.5*\sectorh,-.5*\sectorv) $);
	\draw (A) to [out=\sectorangle,in=180] (B) to [out=0,in=180-\sectorangle] (C) to (CC) to (CCC) to [out=180+\sectorangle,in=0] (D) to [out=180,in=-\sectorangle] (A);
}	

\def\sectorcaretl{
	\begin{scope}[cm={-1,0,0,1,(\x+\x+\sectorh,0)}]
		\sectorcaretr;
	\end{scope}
}

\def\sectortoparc{
	\coordinate (A) at (\x,\y);
	\coordinate (B) at ($ (A) + (.5*\sectorh,.5*\sectorv) $);
	\coordinate (C) at ($ (A) + (\sectorh,0) $);
	\draw (A) to [out=\sectorangle,in=180] (B) to [out=0,in=180-\sectorangle] (C);
}

\def\sectorbotarc{
	\coordinate (A) at (\x,\y);
	\coordinate (C) at ($ (A) + (\sectorh,0) $);
	\coordinate (D) at ($ (A) + (.5*\sectorh,-.5*\sectorv) $);
	\draw (A) to [out=-\sectorangle,in=180] (D) to [out=0,in=180+\sectorangle] (C);

}

\def\sectorbotarclong{
	\coordinate (A) at (\x,\y);
	\coordinate (C) at ($ (A) + (\sectorhlong,0) $);
	\coordinate (D) at ($ (A) + (.5*\sectorhlong,-.5*\sectorv) $);
	\draw (A) to [out=-\sectorangle,in=180] (D) to [out=0,in=180+\sectorangle] (C);

}

\def\sectortop{
	\sectortoparc;
	\draw (A) to (C);
}

\def\sectorbot{
	\sectorbotarc;
	\draw (A) to (C);
}

\def\sectortoparclow{
	\coordinate (A) at (\x,\y);
	\coordinate (B) at ($ (A) + (.5*\sectorh,.5*\sectorvlow) $);
	\coordinate (C) at ($ (A) + (\sectorh,0) $);
	\draw (A) to [out=\sectoranglelow,in=180] (B) to [out=0,in=180-\sectoranglelow] (C);
}

\def\sectortoparcmid{
	\coordinate (A) at (\x,\y);
	\coordinate (B) at ($ (A) + (.5*\sectorh,.5*\sectorvmid) $);
	\coordinate (C) at ($ (A) + (\sectorh,0) $);
	\draw (A) to [out=\sectoranglemid,in=180] (B) to [out=0,in=180-\sectoranglemid] (C);
}

\def\sectortoparchigh{
	\coordinate (A) at (\x,\y);
	\coordinate (B) at ($ (A) + (.5*\sectorh,.5*\sectorvhigh) $);
	\coordinate (C) at ($ (A) + (\sectorh,0) $);
	\draw (A) to [out=\sectoranglehigh,in=180] (B) to [out=0,in=180-\sectoranglehigh] (C);
}

\def\sectorbotarclow{
	\coordinate (A) at (\x,\y);
	\coordinate (B) at ($ (A) + (.5*\sectorh,-.5*\sectorvlow) $);
	\coordinate (C) at ($ (A) + (\sectorh,0) $);
	\draw (A) to [out=-\sectoranglelow,in=180] (B) to [out=0,in=180+\sectoranglelow] (C);
}

\def\sectorbotarcmid{
	\coordinate (A) at (\x,\y);
	\coordinate (B) at ($ (A) + (.5*\sectorh,-.5*\sectorvmid) $);
	\coordinate (C) at ($ (A) + (\sectorh,0) $);
	\draw (A) to [out=-\sectoranglemid,in=180] (B) to [out=0,in=180+\sectoranglemid] (C);
}

\def\sectorbotarchigh{
	\coordinate (A) at (\x,\y);
	\coordinate (B) at ($ (A) + (.5*\sectorh,-.5*\sectorvhigh) $);
	\coordinate (C) at ($ (A) + (\sectorh,0) $);
	\draw (A) to [out=-\sectoranglehigh,in=180] (B) to [out=0,in=180+\sectoranglehigh] (C);
}

\def\hil{ % half id left
	\coordinate (A) at (\x,\y);
	\coordinate (B) at ($ (A) + (.5*\sectorh,.5*\sectorv) $);
	\coordinate (BB) at ($ (A) + (.5*\sectorh,-.5*\sectorv) $);
	\coordinate (C) at ($ (A) + (\sectorh,0) $);
	\coordinate (D) at ($ (A) + (\coordhhil,\coordvhil) $);
	\coordinate (E) at ($ (D) + (90-\halfidangle:.5*\halfidsep) $);
	\coordinate (F) at ($ (D) + (90-\halfidangle:-.5*\halfidsep) $);
	\coordinate (T) at ($ (A) + (.5*\sectorh,.5*\sectorv+.5*\tickv) $);
	\coordinate (S) at ($ (A) + (.5*\sectorh,.5*\sectorv-.5*\tickv) $);		
	\begin{scope}
		\clip (A) to [out=\sectorangle,in=180] (B) to [out=0,in=180-\sectorangle] (C) to [out=180+\sectorangle,in=0] (BB) to [out=180,in=-\sectorangle] (A);
		\draw (E) to +(180-\halfidangle:3);
		\draw (F) to +(180-\halfidangle:3);
	\end{scope}
	\sectortoparc;
	\sectorbotarc;
	\draw (T) to (S);
}

\def\hir{
	\begin{scope}[cm={-1,0,0,1,(\x+\x+\sectorh,0)}]
		\hil;
	\end{scope}
}

\def\hiltop{
	\coordinate (A) at (\x,\y);
	\coordinate (B) at ($ (A) + (.5*\sectorh,.5*\sectorv) $);
	\coordinate (C) at ($ (A) + (\sectorh,0) $);
	\coordinate (D) at ($ (A) + (\coordhhiltop,\coordvhiltop) $);
	\coordinate (E) at ($ (D) + (90-\halfidangle:.5*\halfidsep) $);
	\coordinate (F) at ($ (D) + (90-\halfidangle:-.5*\halfidsep) $);
	\coordinate (T) at ($ (A) + (.5*\sectorh,.5*\sectorv+.5*\tickv) $);
	\coordinate (S) at ($ (A) + (.5*\sectorh,.5*\sectorv-.5*\tickv) $);		
	\begin{scope}
		\clip (A) to [out=\sectorangle,in=180] (B) to [out=0,in=180-\sectorangle] (C) to (A);
		\draw (E) to +(180-\halfidangle:3);
		\draw (F) to +(180-\halfidangle:3);
	\end{scope}
	\sectortop;
	\draw (T) to (S);
}

\def\hirtop{
	\begin{scope}[cm={-1,0,0,1,(\x+\x+\sectorh,0)}]
		\hiltop;
	\end{scope}
}

\def\hilbotcut{
	\coordinate (A) at (\x,\y);
	\coordinate (B) at ($ (A) + (\defecth,0) $);
	\coordinate (C) at ($ (A) + (\defecth,-\defecth) $);
	\coordinate (D) at ($ (A) + (0,-\defecth) $);
	\coordinate (S) at ($ (A) + (\coordhhilbotcut,\coordvhilbotcut) $);
	\coordinate (E) at ($ (S) + (90-\halfidangle:.5*\halfidsep) $);
	\coordinate (F) at ($ (S) + (90-\halfidangle:-.5*\halfidsep) $);
	\begin{scope}
		\clip (A) to (B) to (C) to (D) to (A);
		\draw (E) to +(180-\halfidangle:3);
		\draw (F) to +(180-\halfidangle:3);
	\end{scope}
	\doublebothigh;
	\defect;
	\begin{scope}[shift={(\defecth,0)}]
		\arcd;
	\end{scope}
}

\def\hilbotcuttight{
	\coordinate (A) at (\x,\y);
	\coordinate (B) at ($ (A) + (\defecth,0) $);
	\coordinate (C) at ($ (A) + (\defecth,-\defecth) $);
	\coordinate (D) at ($ (A) + (0,-\defecth) $);
	\coordinate (S) at ($ (A) + (\coordhhilbotcut,\coordvhilbotcut) $);
	\coordinate (E) at ($ (S) + (90-\halfidangle:.5*\halfidsep) $);
	\coordinate (F) at ($ (S) + (90-\halfidangle:-.5*\halfidsep) $);
	\begin{scope}
		\clip (A) to (B) to (C) to (D) to (A);
		\draw (E) to +(180-\halfidangle:3);
		\draw (F) to +(180-\halfidangle:3);
	\end{scope}
	\doublebothigh;
	\defect;
	\begin{scope}[shift={(\defecth,0)}]
	\end{scope}
}

\def\hirbotcut{
	\begin{scope}[cm={-1,0,0,1,(\x+\x+\sectorh+\defecth,0)}]
		\hilbotcut;
	\end{scope}
}

\def\hirbotcuttight{
	\begin{scope}[cm={-1,0,0,1,(\x+\x+\sectorh+\defecth,0)}]
		\hilbotcuttight;
	\end{scope}
}

\def\hiltopthreecut{
	\coordinate (A) at (\x,\y);
	\coordinate (B) at ($ (A) + (.5*\dsectorh,.5*\dsectorvhigh) $);
	\coordinate (C) at ($ (A) + (\dsectorh,0) $);
	\coordinate (S) at ($ (A) + (\coordhhiltopthreecut,\coordvhiltopthreecut) $);
	\coordinate (E) at ($ (S) + (90-\halfidangleshallow:.5*\halfidsep) $);
	\coordinate (F) at ($ (S) + (90-\halfidangleshallow:-.5*\halfidsep) $);
	\begin{scope}
		\clip (A) to [out=\dsectorhighangle,in=180] (B) to [out=0,in=180-\dsectorhighangle] (C) to (A);
		\draw (E) to +(180-\halfidangleshallow:3);
		\draw (F) to +(180-\halfidangleshallow:3);
	\end{scope}
	\doubletophigh;
	\defect;
	\begin{scope}[shift={(\defecth,0)}]
		\arcu;
	\end{scope}
	\begin{scope}[shift={(.69*\dsectorh,\dtickh)}]
	\begin{scope}[rotate around={-\tickangle-10:(\x,\y)}]
		\tick;
	\end{scope}
	\end{scope}
	\begin{scope}[shift={(.225*\dsectorh,\dtickh-.265)}] %!%!%
	\begin{scope}[rotate around={\tickangle+10:(\x,\y)}]
		\caretlwhite;
	\end{scope}
	\end{scope}	
}

\def\hirtopthreecut{
	\begin{scope}[cm={-1,0,0,1,(\x+\x+\sectorh+\defecth,0)}]
		\hiltopthreecut;
	\end{scope}
}

\def\hilbotthreecut{
	\coordinate (A) at (\x,\y);
	\coordinate (B) at ($ (A) + (1.2*\defecth,0) $);
	\coordinate (C) at ($ (A) + (1.2*\defecth,-\defecth) $);
	\coordinate (D) at ($ (A) + (0,-\defecth) $);
	\coordinate (S) at ($ (A) + (\coordhhilbotthreecut,\coordvhilbotthreecut) $);
	\coordinate (E) at ($ (S) + (90-\halfidangle:.5*\halfidsep) $);
	\coordinate (F) at ($ (S) + (90-\halfidangle:-.5*\halfidsep) $);
	\begin{scope}
		\clip (A) to (B) to (C) to (D) to (A);
		\draw (E) to +(180-\halfidangle:3);
		\draw (F) to +(180-\halfidangle:3);
	\end{scope}
	\triplebothigh;
	\arrowl;
	\begin{scope}[shift={(\defecth,0)}]
		\defectshortl;
	\end{scope}
	\begin{scope}[shift={(\defecth+\defecth,0)}]
		\arcd;
	\end{scope}
	\begin{scope}[shift={(.5*\sectorh+.5*\defecth+.5*\defecth,-.5*\tsectorvhigh)}]
		\tick;
	\end{scope}
}

\def\hirbotthreecut{
	\begin{scope}[cm={-1,0,0,1,(\x+\x+\sectorh+\defecth+\defecth,0)}]
		\hilbotthreecut;
	\end{scope}
}

\def\hilbotthreecutmid{
	\coordinate (A) at (\x,\y);
	\coordinate (B) at ($ (A) + (1.2*\defecth,0) $);
	\coordinate (C) at ($ (A) + (1.2*\defecth,-\defecth) $);
	\coordinate (D) at ($ (A) + (0,-\defecth) $);
	\coordinate (S) at ($ (A) + (\coordhhilbotthreecut,\coordvhilbotthreecut) $);
	\coordinate (E) at ($ (S) + (90-\halfidangle:.5*\halfidsep) $);
	\coordinate (F) at ($ (S) + (90-\halfidangle:-.5*\halfidsep) $);
	\begin{scope}
		\clip (A) to (B) to (C) to (D) to (A);
		\draw (E) to +(180-\halfidangle:3);
		\draw (F) to +(180-\halfidangle:3);
	\end{scope}
	\triplebothigh;
	\arrowlgap;
	\begin{scope}[shift={(\defecth,0)}]
		\arcd;
	\end{scope}
	\begin{scope}[shift={(\defecth+\sectorh,0)}]
		\defect;
		\tickadddown;
	\end{scope}
	\begin{scope}[shift={(.5*\sectorh+.5*\defecth+.5*\defecth,-.5*\tsectorvhigh)}]
		\tick;
	\end{scope}
}

\def\hilbotthreecutmidshorttick{
	\coordinate (A) at (\x,\y);
	\coordinate (B) at ($ (A) + (1.2*\defecth,0) $);
	\coordinate (C) at ($ (A) + (1.2*\defecth,-\defecth) $);
	\coordinate (D) at ($ (A) + (0,-\defecth) $);
	\coordinate (S) at ($ (A) + (\coordhhilbotthreecut,\coordvhilbotthreecut) $);
	\coordinate (E) at ($ (S) + (90-\halfidangle:.5*\halfidsep) $);
	\coordinate (F) at ($ (S) + (90-\halfidangle:-.5*\halfidsep) $);
	\begin{scope}
		\clip (A) to (B) to (C) to (D) to (A);
		\draw (E) to +(180-\halfidangle:3);
		\draw (F) to +(180-\halfidangle:3);
	\end{scope}
	\triplebothigh;
	\arrowlgap;
	\begin{scope}[shift={(\defecth,0)}]
		\arcd;
	\end{scope}
	\begin{scope}[shift={(\defecth+\sectorh,0)}]
		\defect;
		\tickshortupadddown;
	\end{scope}
	\begin{scope}[shift={(.5*\sectorh+.5*\defecth+.5*\defecth,-.5*\tsectorvhigh)}]
		\tick;
	\end{scope}
}

\def\hilbotthreecutmidnotick{
	\coordinate (A) at (\x,\y);
	\coordinate (B) at ($ (A) + (1.2*\defecth,0) $);
	\coordinate (C) at ($ (A) + (1.2*\defecth,-\defecth) $);
	\coordinate (D) at ($ (A) + (0,-\defecth) $);
	\coordinate (S) at ($ (A) + (\coordhhilbotthreecut,\coordvhilbotthreecut) $);
	\coordinate (E) at ($ (S) + (90-\halfidangle:.5*\halfidsep) $);
	\coordinate (F) at ($ (S) + (90-\halfidangle:-.5*\halfidsep) $);
	\begin{scope}
		\clip (A) to (B) to (C) to (D) to (A);
		\draw (E) to +(180-\halfidangle:3);
		\draw (F) to +(180-\halfidangle:3);
	\end{scope}
	\triplebothigh;
	\arrowlgap;
	\begin{scope}[shift={(\defecth,0)}]
		\arcd;
	\end{scope}
	\begin{scope}[shift={(\defecth+\sectorh,0)}]
		\defect;
	\end{scope}
	\begin{scope}[shift={(.5*\sectorh+.5*\defecth+.5*\defecth,-.5*\tsectorvhigh)}]
		\tick;
	\end{scope}
}

\def\hirbotthreecutmidnotick{
	\begin{scope}[cm={-1,0,0,1,(\x+\x+\tsectorh,0)}]
		\hilbotthreecutmidnotick;
	\end{scope}
}

\def\hiltopthree{
	\coordinate (A) at (\x,\y);
	\coordinate (B) at ($ (A) + (.5*\tsectorh,.5*\sectorv) $);
	\coordinate (C) at ($ (A) + (\tsectorh,0) $);
	\coordinate (S) at ($ (A) + (\coordhhiltopthree,\coordvhiltopthree) $);
	\coordinate (E) at ($ (S) + (90-\halfidanglesupshallow:.5*\halfidsep) $);
	\coordinate (F) at ($ (S) + (90-\halfidanglesupshallow:-.5*\halfidsep) $);
	\begin{scope}
		\clip (A) to [out=\tsectorangle,in=180] (B) to [out=0,in=180-\tsectorangle] (C) to (A);
		\draw (E) to +(180-\halfidanglesupshallow:3);
		\draw (F) to +(180-\halfidanglesupshallow:3);
	\end{scope}
	\draw (C) to (A);
	\tripletopl;
}

\def\hirtopthree{
	\begin{scope}[cm={-1,0,0,1,(\x+\x+\tsectorh,0)}]
		\hiltopthree;
	\end{scope}
}

\def\hilbot{
	\coordinate (A) at (\x,\y);
	\coordinate (B) at ($ (A) + (\defecthsuplong,0) $);
	\coordinate (C) at ($ (A) + (\defecthsuplong,-\defecth) $);
	\coordinate (D) at ($ (A) + (0,-\defecth) $);
	\coordinate (H) at ($ (A) + (\tsectorh,0) $);
	\coordinate (S) at ($ (A) + (\coordhhilbot,\coordvhilbot) $);
	\coordinate (E) at ($ (S) + (90-\halfidangleshallow:.5*\halfidsep) $);
	\coordinate (F) at ($ (S) + (90-\halfidangleshallow:-.5*\halfidsep) $);
	\begin{scope}
		\clip (A) to (B) to (C) to (D) to (A);
		\draw (E) to +(180-\halfidangleshallow:3);
		\draw (F) to +(180-\halfidangleshallow:3);
	\end{scope}
	\triplebotarc;
	\draw (A) to (H);
	\begin{scope}[shift={(\defecthsuplong,0)}]
		\tick;
	\end{scope}
}

\def\hilbottight{
	\coordinate (A) at (\x,\y);
	\coordinate (B) at ($ (A) + (\defecthsuplong,0) $);
	\coordinate (C) at ($ (A) + (\defecthsuplong,-\defecth) $);
	\coordinate (D) at ($ (A) + (0,-\defecth) $);
	\coordinate (H) at ($ (A) + (\tsectorh,0) $);
	\coordinate (S) at ($ (A) + (\coordhhilbot,\coordvhilbot) $);
	\coordinate (E) at ($ (S) + (90-\halfidangleshallow:.5*\halfidsep) $);
	\coordinate (F) at ($ (S) + (90-\halfidangleshallow:-.5*\halfidsep) $);
	\begin{scope}
		\clip (A) to (B) to (C) to (D) to (A);
		\draw (E) to +(180-\halfidangleshallow:3);
		\draw (F) to +(180-\halfidangleshallow:3);
	\end{scope}
	\triplebotarc;
	\begin{scope}[shift={(\defecthsuplong,0)}]
	\end{scope}

}

\def\hirbot{
	\begin{scope}[cm={-1,0,0,1,(\x+\x+\tsectorh,0)}]
		\hilbot;
	\end{scope}
}

\def\hirbottight{
	\begin{scope}[cm={-1,0,0,1,(\x+\x+\tsectorh,0)}]
		\hilbottight;
	\end{scope}
}

\def\hilbotcutleft{
	\coordinate (A) at (\x,\y);
	\coordinate (BT) at ($ (A) + (.5*\sectorh,-.5*\sectorv) $); %(.5*\sectorh,-.5*\sectorv)
	\coordinate (CT) at ($ (A) + (\sectorh,0) $);
	\coordinate (DT) at ($ (A) + (\sectorh+\defecth,0) $);
	\coordinate (DDT) at ($ (A) + (.5*\sectorh+.5*\defecth,-.5*\dsectorvhigh) $);
	\coordinate (ST) at ($ (A) + (\coordhhilbotcutleft,\coordvhilbotcutleft) $);
	\coordinate (ET) at ($ (ST) + (90-\halfidangleshallow:.5*\halfidsep) $);
	\coordinate (FT) at ($ (ST) + (90-\halfidangleshallow:-.5*\halfidsep) $);
	\begin{scope}
		\clip (A) to [out=-\sectorangle,in=180] (BT) to [out=0,in=180+\sectorangle] (CT) to (DT) to [out=180+\sectorangle,in=0] (DDT) to [out=180,in=-\sectorangle] (A);
		\draw (ET) to +(180-\halfidangleshallow:3);
		\draw (FT) to +(180-\halfidangleshallow:3);
	\end{scope}
	\doublebothigh;
	\begin{scope}[shift={(\sectorh,0)}]
		\defect;
	\end{scope}
}

\def\ax#1-#2{
\ifthenelse{\equal{#1}{1}}{
			\axiombA{#2}}{}
\ifthenelse{\equal{#1}{2}}{
			\axiombB{#2}}{}
\ifthenelse{\equal{#1}{3}}{
			\axiombC{#2}}{}
\ifthenelse{\equal{#1}{4}}{
			\axiombD{#2}}{}
\ifthenelse{\equal{#1}{5}}{
			\axiombE{#2}}{}
\ifthenelse{\equal{#1}{6}}{
			\axiombF{#2}}{}
\ifthenelse{\equal{#1}{7}}{
			\axiombG{#2}}{}
\ifthenelse{\equal{#1}{8}}{
			\axiombH{#2}}{}
\ifthenelse{\equal{#1}{9}}{
			\axiombI{#2}}{}
\ifthenelse{\equal{#1}{10}}{
			\axiombJ{#2}}{}			
\ifthenelse{\equal{#1}{11}}{
			\axiombK{#2}}{}
\ifthenelse{\equal{#1}{12}}{
			\axiombL{#2}}{}			
\ifthenelse{\equal{#1}{13}}{
			\axiombM{#2}}{}
\ifthenelse{\equal{#1}{14}}{
			\axiombN{#2}}{}			
\ifthenelse{\equal{#1}{15}}{
			\axiombO{#2}}{}
\ifthenelse{\equal{#1}{16}}{
			\axiombP{#2}}{}
\ifthenelse{\equal{#1}{17}}{
			\axiombQ{#2}}{}
\ifthenelse{\equal{#1}{18}}{
			\axiombR{#2}}{}
\ifthenelse{\equal{#1}{19}}{
			\axiombS{#2}}{}
\ifthenelse{\equal{#1}{20}}{
			\axiombT{#2}}{}			
\ifthenelse{\equal{#1}{21}}{
			\axiombU{#2}}{}
\ifthenelse{\equal{#1}{22}}{
			\axiombV{#2}}{}			
\ifthenelse{\equal{#1}{23}}{
			\axiombW{#2}}{}
\ifthenelse{\equal{#1}{24}}{
			\axiombX{#2}}{}			
\ifthenelse{\equal{#1}{25}}{
			\axiombY{#2}}{}
\ifthenelse{\equal{#1}{26}}{
			\axiombZ{#2}}{}			
\ifthenelse{\equal{#1}{27}}{
			\axiombRefl{#2}}{}			
}

\def\axgrid#1{
\ifthenelse{\equal{#1}{1}}{
			\axgridA}{}
\ifthenelse{\equal{#1}{2}}{
			\axgridB}{}
\ifthenelse{\equal{#1}{3}}{
			\axgridC}{}
\ifthenelse{\equal{#1}{4}}{
			\axgridD}{}
\ifthenelse{\equal{#1}{5}}{
			\axgridE}{}
\ifthenelse{\equal{#1}{6}}{
			\axgridF}{}
\ifthenelse{\equal{#1}{7}}{
			\axgridG}{}
\ifthenelse{\equal{#1}{8}}{
			\axgridH}{}
\ifthenelse{\equal{#1}{9}}{
			\axgridI}{}
\ifthenelse{\equal{#1}{10}}{
			\axgridJ}{}			
\ifthenelse{\equal{#1}{11}}{
			\axgridK}{}
\ifthenelse{\equal{#1}{12}}{
			\axgridL}{}			
\ifthenelse{\equal{#1}{13}}{
			\axgridM}{}
\ifthenelse{\equal{#1}{14}}{
			\axgridN}{}			
\ifthenelse{\equal{#1}{15}}{
			\axgridO}{}
\ifthenelse{\equal{#1}{16}}{
			\axgridP}{}
\ifthenelse{\equal{#1}{17}}{
			\axgridQ}{}
\ifthenelse{\equal{#1}{18}}{
			\axgridR}{}
\ifthenelse{\equal{#1}{19}}{
			\axgridS}{}
\ifthenelse{\equal{#1}{20}}{
			\axgridT}{}			
\ifthenelse{\equal{#1}{21}}{
			\axgridU}{}
\ifthenelse{\equal{#1}{22}}{
			\axgridV}{}			
\ifthenelse{\equal{#1}{23}}{
			\axgridW}{}
\ifthenelse{\equal{#1}{24}}{
			\axgridX}{}			
\ifthenelse{\equal{#1}{25}}{
			\axgridY}{}
\ifthenelse{\equal{#1}{26}}{
			\axgridZ}{}			
\ifthenelse{\equal{#1}{27}}{
			\axgridRefl}{}			
}

\def\table#1-#2{
\ifthenelse{\equal{#1}{1}}{
			\tableAA{#2}}{}
\ifthenelse{\equal{#1}{2}}{
			\tableBB{#2}}{}
\ifthenelse{\equal{#1}{3}}{
			\tableCC{#2}}{}
\ifthenelse{\equal{#1}{4}}{
			\tableDD{#2}}{}
}

\def\tablegrid#1{
\ifthenelse{\equal{#1}{1}}{
			\tablegridAA}{}
\ifthenelse{\equal{#1}{2}}{
			\tablegridBB}{}
\ifthenelse{\equal{#1}{3}}{
			\tablegridCC}{}
\ifthenelse{\equal{#1}{4}}{
			\tablegridDD}{}
}

\def\axiombI#1{
\ifnum#1=1
	\sh(0,0);
	\sector;
	\vid;
	\sh(0,-1);
	\strutv;
	\arrowlarcd;
	\sh(4,0);
	\defectshortr;
	\sh(7,0);
	\arrowr;
	\sh(4,-1);
	\defectshortl;
	\sh(7,-1);
	\tick;
	\defect;
	\sh(10,-1);
	\strutv;
\else \ifnum#1=2
	\sh(0,0);
	\sectorcaretr;
	\sh(-.5*\caretlength,0);
	\vid;
	\sh(4,0);
	\defect;
	\sh(7,0);
	\tick;
	\defect;
	\sh(10,0);
	\strutv;
	\sh(0,0);
	\strutv;
	\sh(0,1);
	\arcu;
	\sh(4,1);
	\tickaddup;
	\defectshortr;
	\sh(7,1);
	\arrowr;
\else \ifnum#1=3
	\sh(0,0);
	\tripletopl;
	\triplebotl;
	\vidt;
	\strutv;
	\sh(0,1);
	\tripletopr;
	\sh(10,0);
	\strutv;
\else \ifnum#1=4
	\sh(0,0);
	\arrowl;
	\strutv;
	\sh(3,0);
	\defectshortl;
	\sh(6,0);
	\defect;
	\tick;
	\sh(9,0);
	\strutv;
	\sh(0,1);
	\defect;
	\sh(3,1);
	\tick;
	\defectshortr;
	\sh(6,1);
	\arrowr;
		
\else \ifnum#1=5
	\sh(0,0);
	\tripletopr;
	\triplebotr;
	\vidt;
	\sh(0,-1);
	\triplebotl;
	\strutv;
	\sh(10,-1);
	\strutv;

\else \ifnum#1=6
	\sh(0,0);
	\vidd;
	\doubletoptick;
	\doublebottick;
	\sh(7,0);
	\defect;
	\sh(0,-1);
	\strutv;
	\doublebotl;
	\sh(7,-1);
	\tick;
	\defect;
	\sh(10,-1);
	\strutv;
	
\else \ifnum#1=0
	\begin{scope}[cm={-1,0,0,-1,(0,0)}]
		\axiombI{2}
	\end{scope}

\fi\fi\fi\fi\fi\fi\fi}

\def\axiombJ#1{
\ifnum#1=1
	\sh(0,0);
	\defect;
	\sh(3,0);
	\sector;
	\vid;
	\sh(7,0);
	\arrowrgap;
	\sh(10,-1);
	\strutv;
	\sh(7,-1);
	\tickadddown;
	\defect;
	\sh(3,-1);
	\arcd;
	\sh(0,-1);
	\strutv;
	\arrowlgap;

\else \ifnum#1=2
	\sh(0,0);
	\defect;
	\sh(3,0);
	\tickaddup;
	\arcu;
	\sh(7,0);
	\arrowrgap;
	\sh(10,-1);
	\strutv;
	\sh(7,-1);
	\defect;
	\sh(3,-1);
	\sector;
	\vid;
	\sh(0,-1);
	\strutv;
	\arrowlgap;

\else \ifnum#1=3
	\sh(0,0);
	\defect;
	\sh(3,0);
	\tick;
	\doubletopr;
	\sh(0,-1);
	\strutv;
	\defect;
	\sh(3,-1);
	\vidd;
	\doubletoptick;
	\doublebottick;
	\sh(10,-1);
	\strutv;
\else \ifnum#1=4
	\sh(0,0);
	\tripletopl;
	\triplebotl;
	\vidt;
	\strutv;
	\sh(0,1);
	\tripletopr;
	\sh(10,0);
	\strutv;
\else \ifnum#1=5
	\sh(0,0);
	\arrowl;
	\strutv;
	\sh(3,0);
	\defectshortl;
	\sh(6,0);
	\defect;
	\tick;
	\sh(9,0);
	\strutv;
	\sh(0,1);
	\defect;
	\sh(3,1);
	\tick;
	\defectshortr;
	\sh(6,1);
	\arrowr;
		
\else \ifnum#1=6
	\sh(0,0);
	\tripletopr;
	\triplebotr;
	\vidt;
	\sh(0,-1);
	\triplebotl;
	\strutv;
	\sh(10,-1);
	\strutv;

\else \ifnum#1=7
	\sh(0,0);
	\vidd;
	\doubletoptick;
	\doublebottick;
	\sh(7,0);
	\defect;
	\sh(0,-1);
	\strutv;
	\doublebotl;
	\sh(7,-1);
	\tick;
	\defect;
	\sh(10,-1);
	\strutv;

\fi\fi\fi\fi\fi\fi\fi}

\def\axiombK#1{
\ifnum#1=1
	\sh(0,0);
	\sectortop;
	\sectorbotarc;
	\sh(4,0);
	\arrowrgap;
	\sh(7,0);
	\arrowrgap;
	\sh(10,-1);
	\strutv;
	\sh(7,-1);
	\defect;
	\tick;
	\sh(4,-1);
	\defectshortl;
	\sh(0,-1);
	\arrowlarcd;
	\strutv;
	
\else \ifnum#1=2
	\sh(0,0);
	\sectortop;
	\sh(4,0);
	\defectshortr;
	\sh(0,-.5);
	\sectorbot;
	\sh(4,-.5);
	\defectshortr;
	\sh(7,-.25);
	\arrowr;
	\sh(0,-1.5);
	\strutv;
	\arrowlarcd;
	\sh(4,-1.5);
	\defectshortl;
	\sh(7,-1.5);
	\tick;
	\defect;
	\sh(10,-1.5);
	\strutlong;
	
\else \ifnum#1=3
	\sh(0,0);
	\strutshort;
	\arrowlarcd;
	\sh(4,0);
	\defectshortl;
	\sh(7,0);
	\tick;
	\defect;
	\sh(10,0);
	\strutv;
	\sh(0,1.5);
	\sectorbot;
	\sh(4,1.5);
	\defectshortr;
	\sh(0,2);
	\sectortop;
	\sh(4,2);
	\defectshortr;	
	\sh(7,1.5);
	\arrowr;
	\sh(7,2);
	\arrowr;
	\sh(0,1.5);
	\assoctickl;
	\sh(10,1.5);
	\assoctickr;

\else \ifnum#1=4
	\sh(0,0);
	\strutv;
	\arrowlarcd;
	\sh(4,0);
	\defectshortl;
	\sh(7,0);
	\tick;
	\defect;
	\sh(10,0);
	\strutv;
	\sh(0,1);
	\sectorbot;
	\sh(4,1);
	\defectshortr;
	\sh(0,1.5);
	\sectortop;
	\sh(4,1.5);
	\defectshortr;	
	\sh(7,1);
	\arrowr;
	\sh(7,1.5);
	\arrowr;

\else \ifnum#1=5
	\sh(0,0);
	\sectorbot;
	\strutv;
	\sh(4,0);
	\defectshortl;
	\sh(7,0);
	\tick;
	\defect;
	\sh(10,0);
	\strutv;
	\sh(0,1);
	\defectlong;
	\sh(4,1);
	\tickshort;
	\defectshortr;
	\sh(7,1);
	\arrowr;
	\sh(0,1.5);
	\sectortop;
	\sh(4,1.5);
	\defectshortr;
	\sh(7,1.5);
	\arrowr;
\else \ifnum#1=6
	\sh(0,0);
	\sectorbot;
	\sh(4,0);
	\defectshortl;
	\sh(7,0);
	\tick;
	\defect;
	\sh(0,.5);
	\strutv;
	\arrowllong;
	\sh(4,.5);
	\defectshortl;
	\sh(7,.5);
	\tick;
	\defect;
	\sh(10,.5);
	\strutv;
	\sh(0,1.5);
	\sectortop;
	\sh(4,1.5);
	\defectshortr;
	\sh(7,1.5);
	\arrowr;

\else \ifnum#1=7
	\sh(0,0);
	\sectorbot;
	\sh(4,0);
	\defectshortl;
	\sh(7,0);
	\tick;
	\defect;
	\sh(0,.5);
	\sectortop;
	\strutv;
	\sh(4,.5);
	\defectshortl;
	\sh(7,.5);
	\tick;
	\defect;
	\sh(10,.5);
	\strutv;
	\sh(0,1.5);
	\arcu;
	\sh(4,1.5);
	\tickaddup;
	\defectshortr;
	\sh(7,1.5);
	\arrowr;
	
\else \ifnum#1=8
	\sh(0,0);
	\sectorbot;
	\sh(4,0);
	\defectshortl;
	\sh(7,0);
	\tick;
	\defect;
	\sh(0,.5);
	\sectortop;
	\sh(4,.5);
	\defectshortl;
	\sh(7,.5);
	\tick;
	\defect;
	\sh(0,2);
	\strutshortdown;
	\sh(10,1);
	\strutv;
	\sh(0,2);
	\arcu;
	\sh(4,2);
	\tickaddup;
	\defectshortr;
	\sh(7,2);
	\arrowr;
	\sh(0,0);
	\assoctickl;
	\sh(10,0);
	\assoctickr;	

\else \ifnum#1=9
	\sh(0,0);
	\sectorbot;
	\sectortoparc;
	\strutv;
	\sh(4,0);
	\defectshortl;
	\sh(7,0);
	\tick;
	\defect;
	\sh(10,0);
	\strutv;
	\sh(0,1);
	\arcu;
	\sh(4,1);
	\tickaddup;
	\defectshortr;
	\sh(7,1);
	\arrowr;
	
\else \ifnum#1=0
	\begin{scope}[cm={-1,0,0,-1,(0,0)}]
		\axiombK{9}
	\end{scope}

\fi\fi\fi\fi\fi\fi\fi\fi\fi\fi}

\def\axiombL#1{
\ifnum#1=1
	\sh(0,0);
	\arrowlgap;
	\strutv;
	\sh(3,0);
	\arcd;
	\sh(7,0);
	\tickadddown;
	\defect;
	\sh(10,0);
	\strutv;
	\sh(0,1);
	\arrowlgap;
	\sh(3,1);
	\sectorbot;
	\sectortoparc;
	\sh(7,1);
	\arrowrgap;
	\sh(7.9,1);
	\caretr;
	
\else \ifnum#1=2
	\begin{scope}[cm={-1,0,0,-1,(0,0)}]
		\axiombL{1}
	\end{scope}
\else \ifnum#1=3
	\sh(0,0);
	\arrowlgapbig; 
	\strutlong;
	\sh(3,-.25);
	\sectorbot;
	\sh(3,.25);
	\sectortop;
	\sh(7,-.25);
	\defect;
	\sh(7,.25);
	\defect;
	\sh(10,.25);
	\strutv;
	\sh(0,1.25);
	\defect;
	\sh(3,1.25);
	\tickaddup;
	\arcu;
	\sh(7,1.25);
	\arrowrgap;
	
\else \ifnum#1=4
	\sh(0,0);
	\arrowlgap;
	\assoctickl;
	\sh(3,0);
	\sectorbot;
	\sh(7,0);
	\defect;
	\sh(10,0);
	\assoctickr;
	\sh(0,.5);
	\arrowlgap;
	\sh(3,.5);
	\sectortop;
	\sh(7,.5);
	\defect;
	\sh(0,1);
	\strutv;
	\sh(10,1);
	\strutv;
	\sh(0,2);
	\defect;
	\sh(3,2);
	\tickaddup;
	\arcu;
	\sh(7,2);
	\arrowrgap;
	
\else \ifnum#1=5
	\sh(0,0);
	\arrowlgap;
	\sh(3,0);
	\sectorbot;
	\sh(7,0);
	\defect;
	\sh(0,.5);
	\arrowlgap;
	\sh(3,.5);
	\sectortop;
	\sh(7,.5);
	\defect;
	\sh(0,.5);
	\strutv;
	\sh(10,.5);
	\strutv;
	\sh(0,1.5);
	\defect;
	\sh(3,1.5);
	\tickaddup;
	\arcu;
	\sh(7,1.5);
	\arrowrgap;

\else \ifnum#1=6
	\sh(0,0);
	\arrowlgap;
	\sh(3,0);
	\sectorbot;
	\sh(7,0);
	\defect;
	\sh(0,.5);
	\arrowlgap;
	\strutv;
	\sh(3,.5);
	\defectlong;
	\sh(7,.5);
	\defect;
	\tickshort;
	\sh(10,.5);
	\strutv;
	\sh(0,1.5);
	\defect;
	\sh(3,1.5);
	\sectortop;
	\sh(7,1.5);
	\arrowrgap;
	
\else \ifnum#1=7
	\sh(0,0);
	\arrowlgap;
	\strutv;
	\sh(3,0);
	\sectorbot;
	\sh(7,0);
	\defect;
	\sh(10,0);
	\strutv;
	\sh(0,1);
	\defect;
	\sh(3,1);
	\tickshort;
	\defectlong;
	\sh(7,1);
	\arrowrgap;
	\sh(0,1.5);
	\defect;
	\sh(3,1.5);
	\sectortop;
	\sh(7,1.5);
	\arrowrgap;
	
\else \ifnum#1=8
	\sh(0,0);
	\strutv;
	\arrowlgap;
	\sh(3,0);
	\arcd;
	\sh(7,0);
	\tickadddown;
	\defect;
	\sh(10,0);
	\strutv;
	\sh(0,1);
	\defect;
	\sh(3,1);
	\sectorbot;
	\sh(7,1);
	\arrowrgap;
	\sh(0,1.5);
	\defect;
	\sh(3,1.5);
	\sectortop;
	\sh(7,1.5);
	\arrowrgap;
	
\else \ifnum#1=9
	\sh(0,0);
	\strutv;
	\arrowlgap;
	\sh(3,0);
	\arcd;
	\sh(7,0);
	\tickadddown;
	\defect;
	\sh(10,0);
	\strutv;
	\sh(0,1.5);
	\defect;
	\assoctickl;
	\sh(3,1.5);
	\sectorbot;
	\sh(7,1.5);
	\arrowrgap;
	\sh(10,1.5);
	\assoctickr;
	\sh(0,2);
	\defect;
	\sh(3,2);
	\sectortop;
	\sh(7,2);
	\arrowrgap;
	
\else \ifnum#1=10
	\sh(0,0);
	\strutv;
	\arrowlgap;
	\sh(3,0);
	\arcd;
	\sh(7,0);
	\tickadddown;
	\defect;
	\sh(10,0);
	\strutlong;
	\sh(0,1);
	\defect;
	\sh(3,1);
	\sectorbot;
	\sh(7,1.25);
	\arrowrgapbig;
	\sh(0,1.5);
	\defect;
	\sh(3,1.5);
	\sectortop;

\fi\fi\fi\fi\fi\fi\fi\fi\fi\fi}

\def\axiombM#1{
\ifnum#1=1
	\sh(4,0);
	\arcdgapl;
	\sh(0,0);
	\strutv;
	\arrowlarcdgap;
	\sh(8,0);
	\tickadddown;
	\defect;
	\sh(11,0);
	\strutlong;
	\sh(0,1);
	\arcd;
	\sh(4,1);
	\sector;
	\sh(8,1.25);
	\arrowrgapbig;
	\sh(0,1.5);
	\sector;
	\sh(4,1.5);
	\arcu;
	
\else \ifnum#1=2
	\sh(4,0);
	\arcdgapl;
	\sh(0,0);
	\strutv;
	\arrowlarcdgap;
	\sh(8,0);
	\tickadddown;
	\defect;
	\sh(11,0);
	\strutlong;
	\sh(0,1);
	\sector;
	\sh(4,1);
	\arcd;
	\sh(8,1.25);
	\arrowrgapbig;
	\sh(0,1.5);
	\arcu;
	\sh(4,1.5);
	\sector;

\else \ifnum#1=3
	\sh(4,0);
	\arcdgapl;
	\sh(0,0);
	\strutshort;
	\arrowlarcdgap;
	\sh(8,0);
	\tickadddown;
	\defect;
	\sh(11,0);
	\strutv;
	\sh(0,1.5);
	\sector;
	\sh(4,1.5);
	\arcd;
	\sh(8,1.5);
	\arrowrgap;
	\sh(8,2);
	\arrowrgap;
	\sh(11,1.5);
	\assoctickr;
	\sh(0,1.5);
	\assoctickl;
	\sh(0,2);
	\arcu;
	\sh(4,2);
	\sector;

\else \ifnum#1=4
	\sh(4,0);
	\arcdgapl;
	\sh(0,0);
	\strutv;
	\arrowlarcdgap;
	\sh(8,0);
	\tickadddown;
	\defect;
	\sh(11,0);
	\strutv;
	\sh(0,1);
	\sector;
	\sh(4,1);
	\arcd;
	\sh(8,1);
	\arrowrgap;
	\sh(8,1.5);
	\arrowrgap;
	\sh(0,1.5);
	\arcu;
	\sh(4,1.5);
	\sector;
	
\else \ifnum#1=5
	\sh(0,0);
	\sectorcaretr;
	\strutv;
	\sh(4,0);
	\arcd;
	\sh(8,0);
	\tickadddown;
	\defect;
	\sh(11,0);
	\strutv;
	\sh(0,1);
	\arcu;
	\sh(4,1);
	\arcd;
	\tickshort;
	\sh(8,1);
	\arrowrgap;
	\sh(8,1.5);
	\arrowrgap;
	\sh(0,1.5);
	\arcu;
	\sh(4,1.5);
	\sector;

\else \ifnum#1=6
	\sh(0,0);
	\sectorcaretr;
	\sh(4,0);
	\arcd;
	\sh(8,0);
	\tickadddown;
	\defect;
	\sh(0,.5);
	\arrowlarcuturn;
	\strutv;
	\sh(4,.5);
	\arcd;
	\sh(8,.5);
	\tickadddown;
	\defect;
	\sh(11,.5);
	\strutv;
	\sh(8,1.5);
	\arrowrgap;
	\sh(0,1.5);
	\arcu;
	\sh(4,1.5);
	\sector;

\else \ifnum#1=7
	\sh(0,0);
	\sectorcaretr;
	\sh(4,0);
	\arcdlong;
	\sh(9,0);
	\tickshortupadddown;
	\defect;
	\sh(0,.5);
	\arrowlarcu;
	\strutv;
	\sh(4.5,.5);
	\sector;
	\sh(8.5,.5);
	\defectsemilong;
	\sh(12,.5);
	\strutv;
	\sh(0,1.5);
	\arcu;
	\sh(4,1.5);
	\tickaddup;
	\arculong;
	\sh(9,1.5);
	\arrowrgap;

\else \ifnum#1=8
	\sh(0,0);
	\sectorcaretr;
	\assoctickl;
	\sh(4,0);
	\arcdlong;
	\sh(9,0);
	\tickshortupadddown;
	\defect;
	\sh(12,0);
	\assoctickr;
	\sh(0,.5);
	\arrowlarcu;
	\sh(4.5,.5);
	\sector;
	\sh(8.5,.5);
	\defectsemilong;
	\sh(12,.5);
	\sh(0,2);
	\strutshortdown;
	\sh(12,1);
	\strutv;
	\sh(0,2);
	\arcu;
	\sh(4,2);
	\tickaddup;
	\arculong;
	\sh(9,2);
	\arrowrgap;

\else \ifnum#1=9
	\sh(0,.5);
	\sectorcaretr;
	\sh(4,.5);
	\arculong;
	\sh(9,.5);
	\tickshortdownaddup;
	\defect;
	\sh(12,0);
	\assoctickr;
	\sh(0,0);
	\arrowlarcd;
	\assoctickl;
	\sh(4.5,0);
	\sector;
	\sh(8.5,0);
	\defectsemilong;
	\sh(12,.5);
	\sh(0,2);
	\strutshortdown;
	\sh(12,1);
	\strutv;
	\sh(0,2);
	\arcu;
	\sh(4,2);
	\tickaddup;
	\arculong;
	\sh(9,2);
	\arrowrgap;
	
\else \ifnum#1=10
	\sh(0,.5);
	\sectorcaretr;
	\sh(4,.5);
	\arculong;
	\sh(9,.5);
	\tickshortdownaddup;
	\defect;
	\sh(0,0);
	\arrowlarcd;
	\sh(4.5,0);
	\sector;
	\sh(8.5,0);
	\defectsemilong;
	\sh(12,.5);
	\sh(0,.5);
	\strutv;
	\sh(12,.5);
	\strutv;
	\sh(0,1.5);
	\arcu;
	\sh(4,1.5);
	\tickaddup;
	\arculong;
	\sh(9,1.5);
	\arrowrgap;
	
\else \ifnum#1=11
	\sh(0,.5);
	\strutv;
	\arrowlarcdturn;
	\sh(4,.5);
	\arcu;
	\sh(8,.5);
	\tickaddup;
	\defect;
	\sh(11,.5);
	\strutv;
	\sh(0,1.5);
	\sector;
	\sh(4,1.5);
	\arcu;
	\sh(8,1.5);
	\arrowrgap;
	\sh(0,-.25);
	\arcdgapr;
	\sh(4,-.25);
	\begin{scope}[rotate=-135,shift={(-6.17,3.38)}] %!% insane hack
		\caretrmini;
	\end{scope}
	\sh(4,-.25);
	\sector;
	\sh(8,-.25);
	\defect;

\else \ifnum#1=12
	\sh(0,.65);
	\arcd;
	\sh(4,.65);
	\arcu;
	\tickshort;
	\sh(8,.65);
	\arrowrgap;
	\sh(0,1.15);
	\sector;
	\sh(4,1.15);
	\arcu;
	\sh(8,1.15);
	\arrowrgap;
	\sh(0,-.35);
	\arcdgapr;
	\strutv;
	\sh(4,-.35);
	\begin{scope}[rotate=-135,shift={(-6.10,3.55)}] %!% insane hack
		\caretrmini;
	\end{scope}
	\sh(4,-.35);
	\sector;
	\sh(8,-.35);
	\defect;
	\sh(11,-.35);
	\strutv;
	
\else \ifnum#1=13
	\sh(4,0);
	\arcdgapl;
	\sh(0,0);
	\strutv;
	\arrowlarcdgap;
	\sh(8,0);
	\tickadddown;
	\defect;
	\sh(11,0);
	\strutv;
	\sh(0,1);
	\arcd;
	\sh(4,1);
	\sector;
	\sh(8,1);
	\arrowrgap;
	\sh(8,1.5);
	\arrowrgap;
	\sh(0,1.5);
	\sector;
	\sh(4,1.5);
	\arcu;

\else \ifnum#1=14
	\sh(4,0);
	\arcdgapl;
	\sh(0,0);
	\strutshort;
	\arrowlarcdgap;
	\sh(8,0);
	\tickadddown;
	\defect;
	\sh(11,0);
	\strutv;
	\sh(0,1.5);
	\arcd;
	\sh(4,1.5);
	\sector;
	\sh(8,1.5);
	\arrowrgap;
	\sh(8,2);
	\arrowrgap;
	\sh(11,1.5);
	\assoctickr;
	\sh(0,1.5);
	\assoctickl;
	\sh(0,2);
	\sector;
	\sh(4,2);
	\arcu;
	
\else \ifnum#1=0
	\sh(0,0);
	\defect;
	\sh(3,0);
	\sector;
	\sh(7,0);
	\sectorcaretl;
	\sh(0,-1);
	\strutv;
	\arrowlgap;
	\sh(3,-1);
	\arcd;
	\sh(7,-1);
	\tickshortupadddown;
	\arcd;
	\sh(11,-1);
	\strutv;

\fi\fi\fi\fi\fi\fi\fi\fi\fi\fi\fi\fi\fi\fi\fi}

\def\axiombN#1{
\ifnum#1=1
	\sh(0,0);
	\strutshort;
	\arrowlarcd;
	\sh(4,0);
	\defectshortl;
	\sh(7,0);
	\tickadddown;
	\arcd;
	\sh(11,0);
	\strutshort;
	\sh(0,1.5);
	\assoctickl;
	\arcd;
	\sh(4,1.5);
	\tickshortupadddown;
	\defect;
	\sh(7,1.5);
	\sectorcaretl;
	\sh(11,1.5);
	\assoctickr;
	\sh(0,2);
	\sector;
	\sh(4,2);
	\defectshortr;
	\sh(7,2);
	\arcu;
	\caretrmini;

\else \ifnum#1=2
	\sh(0,0);
	\strutshort;
	\arrowlarcd;
	\sh(4,0);
	\defectshortl;
	\sh(7,0);
	\tickadddown;
	\arcd;
	\sh(11,0);
	\strutshort;
	\sh(0,1.5);
	\assoctickl;
	\sector;
	\sh(4,1.5);
	\defectshortr;
	\sh(7,1.5);
	\arcd;
	\caretrmini;
	\sh(11,1.5);
	\assoctickr;
	\sh(0,2);
	\arcu;
	\sh(4,2);
	\tickshortdownaddup;
	\defect;
	\sh(7,2);
	\sectorcaretl;

\else \ifnum#1=3
	\sh(0,0);
	\strutv;
	\arrowlarcd;
	\sh(4,0);
	\defectshortl;
	\sh(7,0);
	\tickadddown;
	\arcd;
	\sh(11,0);
	\strutv;
	\sh(0,1);
	\sector;
	\sh(4,1);
	\defectshortr;
	\sh(7,1);
	\arcd;
	\caretrmini;
	\sh(0,1.5);
	\arcu;
	\sh(4,1.5);
	\tickshortdownaddup;
	\defect;
	\sh(7,1.5);
	\sectorcaretl;

\else \ifnum#1=4
	\sh(0,0);
	\sectorcaretr;
	\strutv;
	\sh(4,0);
	\defect;
	\sh(7,0);
	\tickadddown;
	\arcd;
	\sh(11,0);
	\strutv;
	\sh(0,1);
	\arcu;
	\sh(4,1);
	\tickaddup;
	\defectshortr;
	\sh(7,1);
	\arcd;
	\caretrmini;
	\sh(0,1.5);
	\arcu;
	\sh(4,1.5);
	\tickaddup;
	\defect;
	\sh(7,1.5);
	\sectorcaretl;
	
\else \ifnum#1=5
	\begin{scope}[cm={-1,0,0,-1,(0,0)}]
		\axiombN{4}
	\end{scope}
\else \ifnum#1=6
	\begin{scope}[cm={-1,0,0,-1,(0,0)}]
		\axiombN{3}
	\end{scope}
\else \ifnum#1=7
	\begin{scope}[cm={-1,0,0,-1,(0,0)}]
		\axiombN{2}
	\end{scope}
\else \ifnum#1=8
	\begin{scope}[cm={-1,0,0,-1,(0,0)}]
		\axiombN{1}
	\end{scope}
\else \ifnum#1=9
	\begin{scope}[cm={-1,0,0,-1,(0,0)}]
		\axiombN{12}
	\end{scope}
\else \ifnum#1=10
	\begin{scope}[cm={-1,0,0,-1,(0,0)}]
		\axiombN{11}
	\end{scope}
\else \ifnum#1=11
	\sh(0,0);
	\strutv;
	\arrowlarcd;
	\sh(4,0);
	\defectshortl;
	\sh(7,0);
	\sector
	\sh(11,0);
	\strutv;
	\sh(0,1);
	\arcd;
	\sh(4,1);
	\tickshortupadddown;
	\defectshortr;
	\sh(7,1);
	\arrowrarcu;
	\sh(0,1.5);
	\sector;
	\sh(4,1.5);
	\defectshortr;
	\sh(7,1.5);
	\arrowrarcu;

\else \ifnum#1=12
	\sh(0,0);
	\strutv;
	\arrowlarcd;
	\sh(4,0);
	\defectshortl;
	\sh(7,0);
	\tickadddown;
	\arcd;
	\sh(11,0);
	\strutv;
	\sh(0,1);
	\arcd;
	\sh(4,1);
	\tickshortupadddown;
	\defect;
	\sh(7,1);
	\sectorcaretl;
	\sh(0,1.5);
	\sector;
	\sh(4,1.5);
	\defectshortr;
	\sh(7,1.5);
	\arcu;
	\caretrmini;

\fi\fi\fi\fi\fi\fi\fi\fi\fi\fi\fi\fi}

\def\axiombT#1{
\ifnum#1=1
	\sh(0,0);
	\hirbotcut;
	\sectortoparc;
	\sh(7,0);
	\arrowrgap;
\else \ifnum#1=2
	\sh(0,0);
	\hirbotcut;
	\sh(7,0);
	\arrowrgap;
	\sh(0,.5);
	\sector;
	\sh(4,.5);
	\defect;
	\iticks;
	\sh(7,.5);
	\arrowrgap;
\else \ifnum#1=3
	\sh(0,0);
	\arcd;
	\strutv;
	\sh(4,0);
	\hil;
	\sh(8,0);
	\strutv;
	\sh(0,1);
	\arcd;
	\sh(4,1);
	\tickminicc;
	\splittopr;
	\iticksrotlminifact;
	\sh(0,1.5);
	\sector;
	\sh(4,1.5);
	\splittopr;
	\iticksrotlminifact;
		
\else \ifnum#1=4
	\sh(0,.5);
	\strutv;
	\arrowlarcdturn;
	\sh(4,.5);
	\arcu;
	\iticksrotlminifact;	
	\sh(6,.5+\arcfactor*.5*\sectorv);
	\tick;
	\sh(8,.5);
	\strutv;
	\sh(0,1.5);
	\sector;
	\sh(4,1.5);
	\splittopr;
	\iticksrotlminifact;
	\sh(0,0);
	\arcd;
	\sh(4,0);
	\hil;
	
\else \ifnum#1=5
	\sh(0,0);
	\arcd;
	\sh(4,0);
	\hil;
	\sh(0,.5);
	\sectorcaretr;
	\strutv;
	\sh(4,.5);
	\arcu;
	\iticksrotlminifact;
	\sh(6,.5+\arcfactor*.5*\sectorv);
	\tick;
	\sh(8,.5);
	\strutv;
	\sh(0,1.5);
	\arcu;
	\sh(4,1.5);
	\tickshortdownaddup;
	\splittopr;
	\iticksrotlminifact;
	
\else \ifnum#1=6
	\sh(0,0);
	\arcd;
	\assoctickl;
	\sh(4,0);
	\hil;
	\sh(8,0);
	\assoctickr;
	\sh(0,.5);
	\sectorcaretr;
	\sh(4,.5);
	\arcu;
	\iticksrotlminifact;
	\sh(6,.5+\arcfactor*.5*\sectorv);
	\tick;
	\sh(0,2);
	\strutshortdown;
	\sh(8,2);
	\strutshortdown;
	\sh(0,2);
	\arcu;
	\sh(4,2);
	\tickshortdownaddup;
	\splittopr;
	\iticksrotlminifact;	

\else \ifnum#1=7
	\sh(0,0);
	\sector;
	\assoctickl;
	\sh(4,0);
	\arcd;
	\sh(8,0);
	\assoctickr;
	\sh(0,.5);
	\arcu;
	\sh(4,.5);
	\hil;
	\sh(0,2);
	\strutshortdown;
	\sh(8,2);
	\strutshortdown;
	\sh(0,2);
	\arcu;
	\sh(4,2);
	\tickshortdownaddup;
	\splittopr;
	\iticksrotlminifact;	

\else \ifnum#1=8
	\sh(0,0);
	\sector;
	\sh(4,0);
	\arcd;
	\sh(0,.5);
	\arcu;
	\sh(4,.5);
	\hil;
	\sh(0,.5);
	\strutv;
	\sh(8,.5);
	\strutv;
	\sh(0,1.5);
	\arcu;
	\sh(4,1.5);
	\tickshortdownaddup;
	\splittopr;
	\iticksrotlminifact;	

\else \ifnum#1=9
	\sh(0,0);
	\sectorbot;
	\sh(4,0);
	\defect;
	\sh(0,.5);
	\hirtop;
	\sh(4,.5);
	\defect;
	
\else \ifnum#1=10
	\sh(0,0);
	\hirtop;
	\sectorbotarc;
	\sh(4,0);
	\arrowrgap;
	
\else \ifnum#1=0
	\sh(0,0);
	\hir;
	\sh(4,0);
	\sector;
	
\fi\fi\fi\fi\fi\fi\fi\fi\fi\fi\fi}

\def\axiombU#1{
\ifnum#1=1
	\sh(0,0);
	\arcd;
	\assoctickl;
	\sh(4,0);
	\hir;
	\sh(8,0);
	\assoctickr;
	\sh(0,.5);
	\sectorcaretr;
	\sh(4,.5);
	\arcu;
	\iticksrotrminifact;
	\sh(6,.5+\arcfactor*.5*\sectorv);
	\tick;
	\sh(0,2);
	\strutshortdown;
	\sh(8,2);
	\strutshortdown;
	\sh(0,2);
	\arcu;
	\sh(4,2);
	\tickshortdownaddup;
	\splittopr;
	\iticksrotrminifact;

\else \ifnum#1=2
	\sh(0,0);
	\sector;
	\assoctickl;
	\sh(4,0);
	\arcd;
	\sh(8,0);
	\assoctickr;
	\sh(0,.5);
	\arcu;
	\sh(4,.5);
	\hir;
	\sh(0,2);
	\strutshortdown;
	\sh(8,2);
	\strutshortdown;
	\sh(0,2);
	\arcu;
	\sh(4,2);
	\tickshortdownaddup;
	\splittopr;
	\iticksrotrminifact;	

\else \ifnum#1=3
	\sh(0,0);
	\sector;
	\sh(4,0);
	\arcd;
	\sh(0,.5);
	\arcu;
	\sh(4,.5);
	\hir;
	\sh(0,.5);
	\strutv;
	\sh(8,.5);
	\strutv;
	\sh(0,1.5);
	\arcu;
	\sh(4,1.5);
	\tickshortdownaddup;
	\splittopr;
	\iticksrotrminifact;	

\else \ifnum#1=4
	\sh(0,0);
	\hirtopthreecut;
	\sectorbotarc;
	
\else \ifnum#1=5
	\sh(0,0);
	\hirbotthreecut;
	\sectortoparc;
	
\else \ifnum#1=6
	\sh(0,0);
	\arcd;
	\strutv;
	\sh(4,0);
	\hir;
	\sh(8,0);
	\strutv;
	\sh(0,1);
	\arcd;
	\sh(4,1);
	\tickminicc;
	\splittopr;
	\iticksrotrminifact;
	\sh(0,1.5);
	\sector;
	\sh(4,1.5);
	\splittopr;
	\iticksrotrminifact;

\else \ifnum#1=7
	\sh(0,.5);
	\arrowlarcdturn;
	\strutv;
	\sh(0,0);
	\arcd;
	\sh(4,0);
	\hir;
	\sh(4,.5);
	\arcu;
	\iticksrotrminifact;
	\sh(6,.5+\arcfactor*.5*\sectorv);
	\tick;
	\sh(8,.5);
	\strutv;
	\sh(0,1.5);
	\sector;
	\sh(4,1.5);
	\splittopr;
	\iticksrotrminifact;	

\else \ifnum#1=8
	\sh(0,0);
	\arcd;
	\sh(4,0);
	\hir;
	\sh(0,.5);
	\sectorcaretr;
	\sh(4,.5);
	\arcu;
	\iticksrotrminifact;
	\sh(6,.5+\arcfactor*.5*\sectorv);
	\tick;
	\sh(0,.5);
	\strutv;
	\sh(8,.5);
	\strutv;
	\sh(0,1.5);
	\arcu;
	\sh(4,1.5);
	\tickshortdownaddup;
	\splittopr;
	\iticksrotrminifact;
	
\else \ifnum#1=0
	\sh(0,0);
	\hil;
	\sh(4,0);
	\sector;		

\fi\fi\fi\fi\fi\fi\fi\fi\fi}

\def\axiombV#1{
\ifnum#1=1
	\sh(0,0);
	\hilbotcut;
	\sh(3,0);
	\sectortoparc;
	\sh(7,0);
	\arrowrgap;
	
\else \ifnum#1=2
	\sh(0,0);
	\hilbotcut;
	\sh(7,0);
	\arrowrgap;
	\sh(0,.5);
	\defect;
	\iticks;
	\sh(3,.5);
	\sector;
	\sh(7,.5);
	\arrowrgap;
	
\else \ifnum#1=3
	\sh(0,0);
	\hilbotthreecutmid;
	\strutv;
	\sh(10,0);
	\strutv;
	\sh(0,1);
	\defect;
	\iticks;
	\sh(3,1);
	\tickshortupadddown;
	\arcd;
	\sh(7,1);
	\arrowrgap;
	\sh(0,1.5);
	\defect;
	\iticks;
	\sh(3,1.5);
	\sector;
	\sh(7,1.5);
	\arrowrgap;
	
\else \ifnum#1=4
	\sh(0,0);
	\hilbotthreecutmid;
	\sh(0,.5);
	\strutv;
	\arrowlgap;
	\iticks;
	\sh(3,.5);
	\arcd;
	\sh(7,.5);
	\defect;
	\tickadddown;
	\sh(10,.5);
	\strutv;
	\sh(0,1.5);
	\defect;
	\iticks;
	\sh(3,1.5);
	\sector;
	\sh(7,1.5);
	\arrowrgap;

\else \ifnum#1=5
	\sh(0,0);
	\hilbotthreecutmidshorttick;
	\sh(0,.5);
	\strutv;
	\arrowlgap;
	\iticks;
	\sh(3,.5);
	\sector;
	\sh(7,.5);
	\defect;
	\sh(10,.5);
	\strutv;
	\sh(0,1.5);
	\defect;
	\iticks;
	\sh(3,1.5);
	\arcu;
	\tickaddup;
	\sh(7,1.5);
	\arrowrgap;

\else \ifnum#1=6
	\sh(0,0);
	\hilbotthreecutmidnotick;
	\sh(3,0);
	\sectortoparc;
	\sh(0,.5);
	\strutv;
	\sh(10,.5);
	\strutv;
	\sh(0,1.5);
	\defect;
	\iticks;
	\sh(3,1.5);
	\tickaddup;
	\arcu;
	\sh(7,1.5);
	\arrowrgap;
	
\else \ifnum#1=7
	\sh(0,0);
	\hiltopthree;
	\sectorbotarclong;
	\sh(0,.5);
	\strutv;
	\sh(10,.5);
	\strutv;
	\sh(0,1.5);
	\tripletopr;
	\itickstriprotlshallow;
	
\else \ifnum#1=8
	\sh(0,0);
	\sectorbotarclong;
	\defectsuplong;
	\sh(5,0);
	\defectsuplong;
	\sh(0,.5);
	\hiltopthree;
	\strutv;
	\sh(10,.5);
	\strutv;
	\sh(0,1.5);
	\tripletopr;
	\itickstriprotlshallow;

\else \ifnum#1=9
	\sh(0,0);
	\sectorbot;
	\sh(4,0);
	\defect;
	\sh(0,.5);
	\hiltop;
	\sh(4,.5);
	\defect;
	
\else \ifnum#1=10
	\sh(0,0);
	\hiltop;
	\sectorbotarc;
	\sh(4,0);
	\arrowrgap;
	
\else \ifnum#1=0
	\sh(0,0);
	\hirbotthreecutmidnotick;
	\sh(3,0);
	\sectortoparc;	

\fi\fi\fi\fi\fi\fi\fi\fi\fi\fi\fi}

\def\axiombW#1{
\ifnum#1=1
	\sh(0,0);
	\hilbotcuttight;
	\sh(3,0);
	\hil;
	\sh(0,.5);
	\strutv;
	\sh(7,1.5);
	\strutshortdown;
	\sh(0,1.5);
	\defect;
	\sh(3,1.5);
	\tickaddup;
	\splittopr;
	\sh(0,1.5);
	\iticks;
	\sh(3,1.5);
	\iticksrotlfact;
	
\else \ifnum#1=2
	\sh(0,0);
	\hilbottight;
	\hiltopthree;
	\sh(5,0);
	\tick;
	\sh(5,1.5);
	\itickshere;
	\sh(0,.5);
	\strutv;
	\sh(10,.5);
	\strutv;
	\sh(0,1.5);
	\tripletopr;
	\itickstriprotlshallow;
	\sh(5,3);
	\itickshere;
	
\else \ifnum#1=3
	\sh(0,0);
	\hilbot;
	\sh(0,.5);
	\hiltopthree;
	\sh(5,.5);
	\tick;
	\sh(5,2);
	\itickshere;
	\sh(0,.5);
	\strutv;
	\sh(10,.5);
	\strutv;
	\sh(0,1.5);
	\tripletopr;
	\itickstriprotlshallow;
	\sh(5,3);
	\itickshere;

\else \ifnum#1=4
	\sh(0,0);
	\hilbotcutleft;
	\hil;
	\iticksrotr;
	
\else \ifnum#1=5
	\sh(0,0);
	\hilbotcutleft;
	\hir;
	\iticksrotl;

\else \ifnum#1=6
	\sh(0,0);
	\hilbotcut;
	\sh(0,.5);
	\strutv;
	\defect;
	\iticks;
	\sh(3,.5);
	\hil;
	\sh(7,.5);
	\strutv;
	\sh(0,1.5);
	\defect;
	\iticks;
	\sh(3,1.5);
	\tickaddup;
	\splittopr;
	\iticksrotlfact;
	
\else \ifnum#1=0
	\sh(0,0);
	\hirbotcuttight;
	\hir;
	
\fi\fi\fi\fi\fi\fi\fi}

\def\axiombX#1{
\ifnum#1=1
	\sh(0,0);
	\hirbotcuttight;
	\hil;

\else \ifnum#1=2
	\sh(0,0);
	\hirbot;
	\sh(0,.5);
	\hiltopthree;
	\sh(5,.5);
	\tick;
	\sh(10,.5);
	\itickstriprotrshallow;
	\sh(0,.5);
	\strutv;
	\sh(10,.5);
	\strutv;
	\sh(0,1.5);
	\tripletopr;
	\itickstriprotlshallow;
	\sh(10,1.5);
	\itickstriprotrshallow;

\else \ifnum#1=3
	\sh(0,0);
	\hirbottight;
	\hiltopthree;
	\sh(5,0);
	\tick;
	\sh(0,.5);
	\strutv;
	\sh(10,.5);
	\strutv;
	\sh(0,1.5);
	\tripletopr;
	\itickstriprotlshallow;
	\sh(10,0);
	\itickstriprotrshallow;
	\sh(10,1.5);
	\itickstriprotrshallow;

\else \ifnum#1=4
	\sh(0,0);
	\hilbotcuttight;
	\sh(3,0);
	\hir;
	\sh(0,.5);
	\strutv;
	\sh(7,1.5);
	\strutshortdown;
	\sh(0,1.5);
	\defect;
	\sh(3,1.5);
	\tickaddup;
	\splittopr;
	\sh(0,1.5);
	\iticks;
	\sh(3,1.5);
	\iticksrotrfact;

\else \ifnum#1=5
	\sh(0,0);
	\hilbotcut;
	\sh(0,.5);
	\strutv;
	\defect;
	\iticks;
	\sh(3,.5);
	\hir;
	\sh(7,.5);
	\strutv;
	\sh(0,1.5);
	\defect;
	\iticks;
	\sh(3,1.5);
	\tickaddup;
	\splittopr;
	\iticksrotrfact;

\else \ifnum#1=6
	\sh(0,0);
	\hilbottight;
	\hirtopthree;
	\sh(5,0);
	\tick;
	\sh(0,0);
	\itickstriprotlshallow;

\fi\fi\fi\fi\fi\fi}

\def\axiombA#1{
\ifnum#1=1
\pgftext{\ingeps{3stuff-item1-1.eps}}
\else \ifnum#1=2
\pgftext{\ingeps{3stuff-item1-2.eps}} 
\else \ifnum#1=3 % This is the null node, though in fact it is not called
\fi\fi\fi}

\def\axiombB#1{
\ifnum#1=1
\pgftext{\ingeps{3stuff-item2-1.eps}}
\else \ifnum#1=2
\pgftext{\ingeps{3stuff-item2-2.eps}} 
\else \ifnum#1=3
\pgftext{\ingeps{3stuff-item2-3.eps}}
\else \ifnum#1=4
\else \ifnum#1=0
\pgftext{\ingeps{3stuff-item2a.eps}}
\fi\fi\fi\fi\fi}

\def\axiombC#1{
\ifnum#1=1
\pgftext{\ingeps{3stuff-item3-1.eps}}
\else \ifnum#1=2
\pgftext{\ingeps{3stuff-item3-2.eps}} 
\else \ifnum#1=3
\pgftext{\ingeps{3stuff-item3-3.eps}}
\else \ifnum#1=4
\fi\fi\fi\fi}

\def\axiombD#1{
\ifnum#1=1
\pgftext{\ingeps{3stuff-item4-1.eps}}
\else \ifnum#1=2
\pgftext{\ingeps{3stuff-item4-2.eps}} 
\else \ifnum#1=3
\pgftext{\ingeps{3stuff-item4-3.eps}}
\else \ifnum#1=4
\pgftext{\ingeps{3stuff-item4-4.eps}} 
\else \ifnum#1=5
\pgftext{\ingeps{3stuff-item4-5.eps}}
\else \ifnum#1=6
\fi\fi\fi\fi\fi\fi}

\def\axiombE#1{
\ifnum#1=1
\pgftext{\ingeps{3stuff-item5-1.eps}}
\else \ifnum#1=2
\pgftext{\ingeps{3stuff-item5-2.eps}} 
\else \ifnum#1=3
\pgftext{\ingeps{3stuff-item5-3.eps}}
\else \ifnum#1=4
\pgftext{\ingeps{3stuff-item5-4.eps}} 
\else \ifnum#1=5
\else \ifnum#1=0
\pgftext{\ingeps{3stuff-item5a.eps}}
\else \ifnum#1=-1
\pgftext{\ingeps{3stuff-item5b.eps}}
\else \ifnum#1=-2
\pgftext{\ingeps{3stuff-item5c.eps}}

\fi\fi\fi\fi\fi\fi\fi\fi}

\def\axiombF#1{
\ifnum#1=1
\pgftext{\ingeps{3stuff-item6-1.eps}}
\else \ifnum#1=2
\pgftext{\ingeps{3stuff-item6-2.eps}} 
\else \ifnum#1=3
\pgftext{\ingeps{3stuff-item6-3.eps}}
\else \ifnum#1=4
\pgftext{\ingeps{3stuff-item6-4.eps}} 
\else \ifnum#1=5
\pgftext{\ingeps{3stuff-item6-5.eps}}
\else \ifnum#1=6
\else \ifnum#1=0
\pgftext{\ingeps{3stuff-item6a.eps}}
\fi\fi\fi\fi\fi\fi\fi}

\def\axiombG#1{
\ifnum#1=1
\pgftext{\ingeps{3stuff-item7-1.eps}}
\else \ifnum#1=2
\pgftext{\ingeps{3stuff-item7-2.eps}} 
\else \ifnum#1=3
\pgftext{\ingeps{3stuff-item7-3.eps}}
\else \ifnum#1=4
\pgftext{\ingeps{3stuff-item7-4.eps}} 
\else \ifnum#1=5
\pgftext{\ingeps{3stuff-item7-5.eps}}
\else \ifnum#1=6
\pgftext{\ingeps{3stuff-item7-6.eps}}
\else \ifnum#1=7
\else \ifnum#1=0
\pgftext{\ingeps{3stuff-item7a.eps}}
\fi\fi\fi\fi\fi\fi\fi\fi}

\def\axiombH#1{
\ifnum#1=1
\pgftext{\ingeps{3stuff-item8-1.eps}}
\else \ifnum#1=2
\pgftext{\ingeps{3stuff-item8-2.eps}} 
\else \ifnum#1=3
\pgftext{\ingeps{3stuff-item8-3.eps}}
\else \ifnum#1=4
\pgftext{\ingeps{3stuff-item8-4.eps}} 
\else \ifnum#1=5
\pgftext{\ingeps{3stuff-item8-5.eps}}
\else \ifnum#1=6
\pgftext{\ingeps{3stuff-item8-6.eps}} 
\else \ifnum#1=7
\pgftext{\ingeps{3stuff-item8-7.eps}}
\else \ifnum#1=8
\pgftext{\ingeps{3stuff-item8-8.eps}}
\else \ifnum#1=9
\else \ifnum#1=0
\pgftext{\ingeps{3stuff-item8a.eps}}
\fi\fi\fi\fi\fi\fi\fi\fi\fi\fi}

\def\axiombO#1{
\ifnum#1=1
\pgftext{\ingeps{b3stuff-item15-1A.eps}}
\else \ifnum#1=2
\pgftext{\ingeps{b3stuff-item15-1B.eps}} 
\else \ifnum#1=3
\pgftext{\ingeps{b3stuff-item15-2A.eps}}
\else \ifnum#1=4
\pgftext{\ingeps{b3stuff-item15-2B.eps}} 
\else \ifnum#1=5
\pgftext{\ingeps{b3stuff-item15-3A.eps}}
\else \ifnum#1=6
\pgftext{\ingeps{b3stuff-item15-4A.eps}} 
\else \ifnum#1=7
\pgftext{\ingeps{b3stuff-item15-5A.eps}}
\else \ifnum#1=8
\pgftext{\ingeps{b3stuff-item15-5B.eps}}
\else \ifnum#1=9
\pgftext{\ingeps{b3stuff-item15-6A.eps}} 
\else \ifnum#1=10
\pgftext{\ingeps{b3stuff-item15-6B.eps}}
\else \ifnum#1=11
\pgftext{\ingeps{b3stuff-item15-7A.eps}} 
\else \ifnum#1=12
\pgftext{\ingeps{b3stuff-item15-7B.eps}}
\else \ifnum#1=13
\pgftext{\ingeps{b3stuff-item15-8A.eps}} 
\else \ifnum#1=14
\pgftext{\ingeps{b3stuff-item15-8B.eps}}
\else \ifnum#1=15
\pgftext{\ingeps{b3stuff-item15-9A.eps}}
\else \ifnum#1=16
\pgftext{\ingeps{b3stuff-item15-9B.eps}}
\else \ifnum#1=0
\pgftext{\ingeps{b3stuff-item15a.eps}}
\fi\fi\fi\fi\fi\fi\fi\fi\fi\fi\fi\fi\fi\fi\fi\fi\fi}

\def\axiombP#1{
\ifnum#1=1
\pgftext{\ingeps{b3stuff-item16-1A.eps}}
\else \ifnum#1=2
\pgftext{\ingeps{b3stuff-item16-1B.eps}} 
\else \ifnum#1=3
\pgftext{\ingeps{b3stuff-item16-2A.eps}}
\else \ifnum#1=4
\pgftext{\ingeps{b3stuff-item16-2B.eps}} 
\else \ifnum#1=5
\pgftext{\ingeps{b3stuff-item16-3A.eps}}
\else \ifnum#1=6
\pgftext{\ingeps{b3stuff-item16-4A.eps}} 
\else \ifnum#1=7
\pgftext{\ingeps{b3stuff-item16-5A.eps}}
\else \ifnum#1=8
\pgftext{\ingeps{b3stuff-item16-5B.eps}}
\else \ifnum#1=9
\pgftext{\ingeps{b3stuff-item16-6A.eps}} 
\else \ifnum#1=10
\pgftext{\ingeps{b3stuff-item16-6B.eps}}
\else \ifnum#1=11
\pgftext{\ingeps{b3stuff-item16-7A.eps}} 
\else \ifnum#1=12
\pgftext{\ingeps{b3stuff-item16-8A.eps}}
\else \ifnum#1=13
\pgftext{\ingeps{b3stuff-item16-9A.eps}} 
\else \ifnum#1=14
\pgftext{\ingeps{b3stuff-item16-9B.eps}}
\else \ifnum#1=15
\pgftext{\ingeps{b3stuff-item16-10A.eps}}
\else \ifnum#1=16
\pgftext{\ingeps{b3stuff-item16-10B.eps}}
\else \ifnum#1=17
\pgftext{\ingeps{b3stuff-item16-11A.eps}}
\else \ifnum#1=18
\pgftext{\ingeps{b3stuff-item16-11B.eps}}
\else \ifnum#1=0
\pgftext{\ingeps{b3stuff-item16a.eps}}
\fi\fi\fi\fi\fi\fi\fi\fi\fi\fi\fi\fi\fi\fi\fi\fi\fi\fi\fi}

\def\axiombQ#1{
\ifnum#1=1
\pgftext{\ingeps{b3stuff-item17-1A.eps}}
\else \ifnum#1=2
\pgftext{\ingeps{b3stuff-item17-13B.eps}} 
\else \ifnum#1=3
\pgftext{\ingeps{b3stuff-item17-13A.eps}}
\else \ifnum#1=4
\pgftext{\ingeps{b3stuff-item17-12A.eps}} 
\else \ifnum#1=5
\pgftext{\ingeps{b3stuff-item17-11A.eps}}
\else \ifnum#1=6
\pgftext{\ingeps{b3stuff-item17-10B.eps}} 
\else \ifnum#1=7
\pgftext{\ingeps{b3stuff-item17-10A.eps}}
\else \ifnum#1=8
\pgftext{\ingeps{b3stuff-item17-9B.eps}}
\else \ifnum#1=9
\pgftext{\ingeps{b3stuff-item17-9A.eps}} 
\else \ifnum#1=10
\pgftext{\ingeps{b3stuff-item17-8B.eps}}
\else \ifnum#1=11
\pgftext{\ingeps{b3stuff-item17-8A.eps}} 
\else \ifnum#1=12
\pgftext{\ingeps{b3stuff-item17-7B.eps}}
\else \ifnum#1=13
\pgftext{\ingeps{b3stuff-item17-7A.eps}} 
\else \ifnum#1=14
\pgftext{\ingeps{b3stuff-item17-6B.eps}}
\else \ifnum#1=15
\pgftext{\ingeps{b3stuff-item17-6A.eps}}
\else \ifnum#1=16
\pgftext{\ingeps{b3stuff-item17-5B.eps}}
\else \ifnum#1=17
\pgftext{\ingeps{b3stuff-item17-5A.eps}}
\else \ifnum#1=18
\pgftext{\ingeps{b3stuff-item17-4B.eps}}
\else \ifnum#1=19
\pgftext{\ingeps{b3stuff-item17-4A.eps}}
\else \ifnum#1=20
\pgftext{\ingeps{b3stuff-item17-3A.eps}}
\else \ifnum#1=21
\pgftext{\ingeps{b3stuff-item17-2A.eps}}
\else \ifnum#1=22
\pgftext{\ingeps{b3stuff-item17-1B.eps}}
\fi\fi\fi\fi\fi\fi\fi\fi\fi\fi\fi\fi\fi\fi\fi\fi\fi\fi\fi\fi\fi\fi}

\def\axiombR#1{
\ifnum#1=1
\pgftext{\ingeps{d3stuff-item18-1.eps}}
\else \ifnum#1=2
\pgftext{\ingeps{d3stuff-item18-2.eps}} 
\else \ifnum#1=3
\pgftext{\ingeps{d3stuff-item18-3.eps}}
\else \ifnum#1=4
\pgftext{\ingeps{d3stuff-item18-4.eps}} 
\else \ifnum#1=5
\else \ifnum#1=0
\pgftext{\ingeps{d3stuff-item18a.eps}} 
\fi\fi\fi\fi\fi\fi}

\def\axiombS#1{
\ifnum#1=1
\pgftext{\ingeps{d3stuff-item19-1.eps}}
\else \ifnum#1=2
\pgftext{\ingeps{d3stuff-item19-2.eps}} 
\else \ifnum#1=3
\pgftext{\ingeps{d3stuff-item19-3.eps}}
\else \ifnum#1=4
\pgftext{\ingeps{d3stuff-item19-4.eps}} 
\else \ifnum#1=5
\pgftext{\ingeps{d3stuff-item19-5.eps}} 
\else \ifnum#1=6
\pgftext{\ingeps{d3stuff-item19-6.eps}}
\else \ifnum#1=7
\pgftext{\ingeps{d3stuff-item19-7.eps}} 
\else \ifnum#1=8
\else \ifnum#1=0
\pgftext{\ingeps{d3stuff-item19a.eps}} 
\fi\fi\fi\fi\fi\fi\fi\fi\fi}

\def\axiombY#1{
\ifnum#1=1
\pgftext{\ingeps{b3stuff-item25-1.eps}}
\else \ifnum#1=2
\pgftext{\ingeps{b3stuff-item25-2.eps}} 
\else \ifnum#1=3
\pgftext{\ingeps{b3stuff-item25-3A.eps}}
\else \ifnum#1=4
\pgftext{\ingeps{b3stuff-item25-3B.eps}} 
\else \ifnum#1=5
\pgftext{\ingeps{b3stuff-item25-4A.eps}}
\else \ifnum#1=6
\pgftext{\ingeps{b3stuff-item25-4B.eps}} 
\else \ifnum#1=7
\pgftext{\ingeps{b3stuff-item25-5A.eps}}
\else \ifnum#1=8
\pgftext{\ingeps{b3stuff-item25-5B.eps}}
\else \ifnum#1=9
\pgftext{\ingeps{b3stuff-item25-6A.eps}}
\else \ifnum#1=10
\pgftext{\ingeps{b3stuff-item25-7.eps}}
\else \ifnum#1=11
\else \ifnum#1=0
\pgftext{\ingeps{b3stuff-item25a.eps}}
\fi\fi\fi\fi\fi\fi\fi\fi\fi\fi\fi\fi}

\def\axiombZ#1{
\ifnum#1=1
\pgftext{\ingeps{b3stuff-item26-1.eps}}
\else \ifnum#1=2
\pgftext{\ingeps{b3stuff-item26-8.eps}} 
\else \ifnum#1=3
\pgftext{\ingeps{b3stuff-item26-7B.eps}}
\else \ifnum#1=4
\pgftext{\ingeps{b3stuff-item26-7A.eps}} 
\else \ifnum#1=5
\pgftext{\ingeps{b3stuff-item26-6B.eps}}
\else \ifnum#1=6
\pgftext{\ingeps{b3stuff-item26-6A.eps}} 
\else \ifnum#1=7
\pgftext{\ingeps{b3stuff-item26-5A.eps}}
\else \ifnum#1=8
\pgftext{\ingeps{b3stuff-item26-4A.eps}}
\else \ifnum#1=9
\pgftext{\ingeps{b3stuff-item26-3B.eps}} 
\else \ifnum#1=10
\pgftext{\ingeps{b3stuff-item26-3A.eps}}
\else \ifnum#1=11
\pgftext{\ingeps{b3stuff-item26-2.eps}}
\else \ifnum#1=12
\else \ifnum#1=0
\pgftext{\ingeps{b3stuff-item26a.eps}}
\fi\fi\fi\fi\fi\fi\fi\fi\fi\fi\fi\fi\fi}

\def\axiombRefl#1{
\ifnum#1=2
\axiombB{0}
\else \ifnum#1=51
\axiombE{0}
\else \ifnum#1=52
\axiombE{-1}
\else \ifnum#1=53
\axiombE{-2}
\else \ifnum#1=6
\axiombF{0}
\else \ifnum#1=7
\axiombG{0}
\else \ifnum#1=8
\axiombH{0}
\else \ifnum#1=9
\axiombI{0}
\else \ifnum#1=11
\axiombK{0}
\else \ifnum#1=13
\axiombM{0}
\else \ifnum#1=15
\axiombO{0}
\else \ifnum#1=16
\axiombP{0}
\else \ifnum#1=18
\axiombR{0}
\else \ifnum#1=19
\axiombS{0}
\else \ifnum#1=20
\axiombT{0}
\else \ifnum#1=21
\axiombU{0}
\else \ifnum#1=22
\axiombV{0}
\else \ifnum#1=23
\axiombW{0}
\else \ifnum#1=25
\axiombY{0}
\else \ifnum#1=26
\axiombZ{0}
\fi\fi\fi\fi\fi\fi\fi\fi\fi\fi\fi\fi\fi\fi\fi\fi\fi\fi\fi\fi}

\def\tableAA#1{
\ifnum#1=1
	\sh(0,0);
	\draw[fill=black] (0,0) circle (.3ex);
\else \ifnum#1=2
	\sh(0,0);
	\defect;
\else \ifnum#1=3
	\sh(0,0);
	\sector;
\fi\fi\fi}

\def\tableBB#1{
\ifnum#1=1
	\sh(0,0);
	\defect;
	\iticks;
\else \ifnum#1=2
	\sh(0,0);
	\defect;
	\sh(3,0);
	\tick;
	\defect;
\else \ifnum#1=3
	\sh(0,0);
	\sector;
	\vid;
\else \ifnum#1=4
	\sh(0,0);
	\sectortop;
	\sectorbotarc;
\else \ifnum#1=5
	\sh(0,0);
	\sector;
	\sh(4,0);
	\defect;
\else \ifnum#1=6
	\sh(0,0);
	\defect;
	\sh(3,0);
	\sector;
\else \ifnum#1=7
	\sh(0,0);
	\defect;
	\iticks;
	\sh(3,0);
	\tick;
	\defect;
\else \ifnum#1=8
	\sh(0,0);
	\defect;
	\sh(3,0);
	\tick;
	\defect;
	\iticks;
\else \ifnum#1=9
	\sh(0,0);
	\defect;
	\sh(3,0);
	\tick;
	\defect;
	\sh(6,0);
	\tick;
	\defect;
\fi\fi\fi\fi\fi\fi\fi\fi\fi}

\def\tableCC#1{
\ifnum#1=1
\pgftext{\ingeps{2stuff-item1-1.eps}}
\else \ifnum#1=2
\pgftext{\ingeps{2stuff-item2-1.eps}}
\else \ifnum#1=3
	\sh(0,0);
	\sectortoparclow;
	\sectortoparchigh;
	\sectorbotarclow;
	\sectorbotarchigh;
\else \ifnum#1=4
	\sh(0,0);
	\sector;
	\vid;
	\sh(4,0);
	\defect;
\else \ifnum#1=5
	\sh(0,0);
	\defect;
	\sh(3,0);
	\sector;
	\vid;
\else \ifnum#1=6
	\sh(0,0);
	\sectortop;
	\sectorbotarc;
	\sh(4,0);
	\defect;
\else \ifnum#1=7
	\sh(0,0);
	\defect;
	\sh(3,0);
	\sectortop;
	\sectorbotarc;
\else \ifnum#1=8
	\sh(0,0);
	\sector;
	\sh(4,0);
	\sector;
\else \ifnum#1=9
	\sh(0,0);
	\sector;
	\sh(4,0);
	\defect;
	\sh(7,0);
	\tick;
	\defect;
\else \ifnum#1=10
	\sh(0,0);
	\defect;
	\sh(3,0);
	\tick;
	\defect;
	\sh(6,0);
	\sector;
\else \ifnum#1=11
	\sh(0,0);
	\defect;
	\sh(3,0);
	\sector;
	\sh(7,0);
	\defect;
\else \ifnum#1=12
	\sh(0,0);
	\defect;
	\sh(3,0);
	\tick;
	\defect;
	\sh(6,0);
	\tick;
	\defect;
	\sh(9,0);
	\tick;
	\defect;
\else \ifnum#1=13
	\sh(0,0);
	\defect;
	\iticks;
	\sh(3,0);
	\sector;
\else \ifnum#1=14
	\sh(0,0);
	\sector;
	\sh(4,0);
	\defect;
	\iticks;
\else \ifnum#1=15
	\sh(0,0);
	\defect;
	\iticks;
	\sh(3,0);
	\tick;
	\defect;
	\iticks;
\else \ifnum#1=16
	\sh(0,0);
	\defect;
	\iticks;
	\sh(3,0);
	\tick;
	\defect;
	\sh(6,0);
	\tick;
	\defect;
\else \ifnum#1=17
	\sh(0,0);
	\defect;
	\sh(3,0);
	\tick;
	\defect;
	\sh(6,0);
	\tick;
	\defect;
	\iticks;
\else \ifnum#1=18
	\sh(0,0);
	\defect;
	\sh(3,0);
	\tick;
	\defect;
	\iticks;
	\sh(6,0);
	\tick;
	\defect;
\fi\fi\fi\fi\fi\fi\fi\fi\fi\fi\fi\fi\fi\fi\fi\fi\fi\fi}

\def\tableDD#1{
\ifnum#1=1
	\pgftext{\ingeps{3stuff-item1-1.eps}}
\else \ifnum#1=2
	\begin{scope}[cm={1,0,0,-1,(0,0)}]
		\pgftext{\ingeps{3stuff-item2a.eps}}
	\end{scope}
\else \ifnum#1=3
	\pgftext{\ingeps{3stuff-item2a.eps}}
\else \ifnum#1=4
	\sh(0,0);
	\sectortoparclow;
	\sectortoparchigh;
	\sectorbotarclow;
	\sectorbotarchigh;
	\vids;
\else \ifnum#1=5
	\sh(0,0);
	\sectortoparcmid;
	\sectortoparchigh;
	\sectorbotarcmid;
	\sectorbotarchigh;
	\draw (A) to (C);
\else \ifnum#1=6
	\begin{scope}[cm={1,0,0,-1,(0,0)}]
		\pgftext{\ingeps{3stuff-item5a.eps}}
	\end{scope}
\else \ifnum#1=7
	\pgftext{\ingeps{3stuff-item5a.eps}}
\else \ifnum#1=8
	\pgftext{\ingeps{3stuff-item5b.eps}}
\else \ifnum#1=9
	\pgftext{\ingeps{3stuff-item5c.eps}}
\else \ifnum#1=10
	\sh(0,0);
	\sector;
	\sh(4,0);
	\sector;
	\vid;
\else \ifnum#1=11
	\sh(0,0);
	\sector;
	\vid;
	\sh(4,0);
	\sector;
\else \ifnum#1=12
	\begin{scope}[cm={-1,0,0,1,(0,0)}]
		\pgftext{\ingeps{3stuff-item7a.eps}}
	\end{scope}
\else \ifnum#1=13
	\pgftext{\ingeps{3stuff-item7a.eps}}
\else \ifnum#1=14
	\sh(0,0);
	\sectortop;
	\sectorbotarc;
	\sh(4,0);
	\sector;
\else \ifnum#1=15
	\sh(0,0);
	\sector;
	\sh(4,0);
	\sectortop;
	\sectorbotarc;
\else \ifnum#1=16
	\sh(0,0);
	\sector;
	\vid;
	\sh(4,0);
	\defect;
	\sh(7,0);
	\tick;
	\defect;
\else \ifnum#1=17
	\sh(0,0);
	\defect;
	\sh(3,0);
	\tick;
	\defect;
	\sh(6,0);
	\sector;
	\vid;
\else \ifnum#1=18
	\sh(0,0);
	\defect;
	\sh(3,0);
	\sector;
	\vid;
	\sh(7,0);
	\defect;
\else \ifnum#1=19
	\sh(0,0);
	\sectortop;
	\sectorbotarc;
	\sh(4,0);
	\defect;
	\sh(7,0);
	\tick;
	\defect;
\else \ifnum#1=20
	\sh(0,0);
	\defect;
	\sh(3,0);
	\tick;
	\defect;
	\sh(6,0);
	\sectortop;
	\sectorbotarc;
\else \ifnum#1=21
	\sh(0,0);
	\defect;
	\sh(3,0);
	\sectortop;
	\sectorbotarc;
	\sh(7,0);
	\defect;
\else \ifnum#1=22
	\sh(0,0);
	\sector;
	\sh(4,0);
	\sector;
	\sh(8,0);
	\defect;
\else \ifnum#1=23
	\sh(0,0);
	\defect;
	\sh(3,0);
	\sector;
	\sh(7,0);
	\sector;
\else \ifnum#1=24
	\sh(0,0);
	\sector;
	\sh(4,0);
	\defect;
	\sh(7,0);
	\sector;
\else \ifnum#1=25
	\sh(0,0);
	\sector;
	\sh(4,0);
	\defect;
	\sh(7,0);
	\tick;
	\defect;
	\sh(10,0);
	\tick;
	\defect;
\else \ifnum#1=26
	\sh(0,0);
	\defect;
	\sh(3,0);
	\tick;
	\defect;
	\sh(6,0);
	\tick;
	\defect;
	\sh(9,0);
	\sector;
\else \ifnum#1=27
	\sh(0,0);
	\defect;
	\sh(3,0);
	\sector;
	\sh(7,0);
	\defect;
	\sh(10,0);
	\tick;
	\defect;
\else \ifnum#1=28
	\sh(0,0);
	\defect;
	\sh(3,0);
	\tick;
	\defect;
	\sh(6,0);
	\sector;
	\sh(10,0);
	\defect;
\else \ifnum#1=29
	\sh(0,0);
	\defect;
	\sh(3,0);
	\tick;
	\defect;
	\sh(6,0);
	\tick;
	\defect;
	\sh(9,0);
	\tick;
	\defect;
	\sh(12,0);
	\tick;
	\defect;
\else \ifnum#1=30
	\sh(0,0);
	\sector;
	\vid;
	\sh(4,0);
	\defect;
	\iticks;
\else \ifnum#1=31
	\sh(0,0);
	\defect;
	\iticks;
	\sh(3,0);
	\sector;
	\vid;
\else \ifnum#1=32
	\sh(0,0);
	\sectortop;
	\sectorbotarc;
	\sh(4,0);
	\defect;
	\iticks;
\else \ifnum#1=33
	\sh(0,0);
	\defect;
	\iticks;
	\sh(3,0);
	\sectortop;
	\sectorbotarc;
\else \ifnum#1=34
	\sh(0,0);
	\sector;
	\sh(4,0);
	\defect;
	\iticks;
	\sh(7,0);
	\tick;
	\defect;
\else \ifnum#1=35
	\sh(0,0);
	\defect;
	\sh(3,0);
	\tick;
	\defect;
	\iticks;
	\sh(6,0);
	\sector;
\else \ifnum#1=36
	\sh(0,0);
	\sector;
	\sh(4,0);
	\defect;
	\sh(7,0);
	\tick;
	\defect;
	\iticks;
\else \ifnum#1=37
	\sh(0,0);
	\defect;
	\iticks;
	\sh(3,0);
	\tick;
	\defect;
	\sh(6,0);
	\sector;
\else \ifnum#1=38
	\sh(0,0);
	\defect;
	\iticks;
	\sh(3,0);
	\sector;
	\sh(7,0);
	\defect;
\else \ifnum#1=39
	\sh(0,0);
	\defect;
	\sh(3,0);
	\sector;
	\sh(7,0);
	\defect;
	\iticks;
\else \ifnum#1=40
	\sh(0,0);
	\defect;
	\iticks;
	\sh(3,0);
	\tick;
	\defect;
	\iticks;
	\sh(6,0);
	\tick;
	\defect;
\else \ifnum#1=41
	\sh(0,0);
	\defect;
	\sh(3,0);
	\tick;
	\defect;
	\iticks;
	\sh(6,0);
	\tick;
	\defect;
	\iticks;
\else \ifnum#1=42
	\sh(0,0);
	\defect;
	\iticks;
	\sh(3,0);
	\tick;
	\defect;
	\sh(6,0);
	\tick;
	\defect;
	\iticks;
\else \ifnum#1=43
	\sh(0,0);
	\defect;
	\iticks;
	\sh(3,0);
	\tick;
	\defect;
	\sh(6,0);
	\tick;
	\defect;
	\sh(9,0);
	\tick;
	\defect;
\else \ifnum#1=44
	\sh(0,0);
	\defect;
	\sh(3,0);
	\tick;
	\defect;
	\sh(6,0);
	\tick;
	\defect;
	\sh(9,0);
	\tick;
	\defect;
	\iticks;
\else \ifnum#1=45
	\sh(0,0);
	\defect;
	\sh(3,0);
	\tick;
	\defect;
	\iticks;
	\sh(6,0);
	\tick;
	\defect;
	\sh(9,0);
	\tick;
	\defect;
\else \ifnum#1=46
	\sh(0,0);
	\defect;
	\sh(3,0);
	\tick;
	\defect;
	\sh(6,0);
	\tick;
	\defect;
	\iticks;
	\sh(9,0);
	\tick;
	\defect;
\fi\fi\fi\fi\fi\fi\fi\fi\fi\fi\fi\fi\fi\fi\fi\fi\fi\fi\fi\fi\fi\fi\fi\fi\fi\fi\fi\fi\fi\fi\fi\fi\fi\fi\fi\fi\fi\fi\fi\fi\fi\fi\fi\fi\fi\fi}

\def\axgridA{
\path
	(0,0) coordinate (1)
	(0,-1) coordinate (2)
	(0,-1.75) coordinate (n); % This is a null node, may need to change to a number if need to put content in the node
\draw (1) -- (2) -- (n);
}

\def\axgridB{
\path
	(0,0) coordinate (1)
	(0,-1) coordinate (2)
	(0,-2) coordinate (3)
	(0,-2.75) coordinate (n)
	(0,-3.5) coordinate (0); % This is a reflected node
\draw (1) -- (2) -- (3) -- (n);
}

\def\axgridC{
\path
	(0,0) coordinate (1)
	(0,-1) coordinate (2)
	(0,-2) coordinate (3)
	(0,-2.75) coordinate (n);
\draw (1) -- (2) -- (3) -- (n);
}

\def\axgridD{
\path
	(0,0) coordinate (1)
	(0,-1) coordinate (2)
	(0,-2) coordinate (3)
	(0,-3) coordinate (4)
	(0,-4) coordinate (5)
	(0,-4.75) coordinate (n);
\draw (1) -- (2) -- (3) -- (4) -- (5) -- (n);
}

\def\axgridE{
\path
	(0,0) coordinate (1)
	(0,-1) coordinate (2)
	(0,-2) coordinate (3)
	(0,-3) coordinate (4)
	(0,-3.75) coordinate (n)
	(0,-4.5) coordinate (0)
	(0,-5.5) coordinate (-1)
	(0,-6.5) coordinate (-2);
\draw (1) -- (2) -- (3) -- (4) -- (n);
}

\def\axgridF{
\path
	(0,0) coordinate (1)
	(0,-1) coordinate (2)
	(0,-2) coordinate (3)
	(0,-3) coordinate (4)
	(0,-4) coordinate (5)
	(0,-4.75) coordinate (n)
	(0,-5.5) coordinate (0);
\draw (1) -- (2) -- (3) -- (4) -- (5) -- (n);
}

\def\axgridG{
\path
	(0,0) coordinate (1)
	(0,-1) coordinate (2)
	(0,-2) coordinate (3)
	(0,-3) coordinate (4)
	(0,-4) coordinate (5)
	(0,-5) coordinate (6)
	(0,-5.75) coordinate (n)
	(0,-6.5) coordinate (0);
\draw (1) -- (2) -- (3) -- (4) -- (5) -- (6) -- (n);
}

\def\axgridH{
\path
	(0,0) coordinate (1)
	(0,-1) coordinate (2)
	(0,-2) coordinate (3)
	(0,-3) coordinate (4)
	(0,-4) coordinate (5)
	(0,-5) coordinate (6)
	(0,-6) coordinate (7)
	(0,-7) coordinate (8)
	(0,-7.75) coordinate (n)
	(0,-8.5) coordinate (0);
\draw (1) -- (2) -- (3) -- (4) -- (5) -- (6) -- (7) -- (8) -- (n);
}

\def\axgridI{
\path
	(0,0) coordinate (1)
	(0,-1) coordinate (2)
	(0,-2) coordinate (3)
	(0,-3) coordinate (4)
	(0,-4) coordinate (5)
	(0,-5) coordinate (6)
	(0,-5.75) coordinate (n)
	(0,-6.5) coordinate (0);
\draw (1) -- (2) -- (3) -- (4) -- (5) -- (6) -- (n);
}

\def\axgridJ{
\path
	(0,0) coordinate (1)
	(0,-1) coordinate (2)
	(0,-2) coordinate (3)
	(0,-3) coordinate (4)
	(0,-4) coordinate (5)
	(0,-5) coordinate (6)
	(0,-6) coordinate (7)
	(0,-6.75) coordinate (n);
\draw (1) -- (2) -- (3) -- (4) -- (5) -- (6) -- (7) -- (n);
}

\def\axgridK{
\path
	(0,0) coordinate (1)
	(0,-1) coordinate (2)
	(0,-2) coordinate (3)
	(0,-3) coordinate (4)
	(0,-4) coordinate (5)
	(0,-5) coordinate (6)
	(0,-6) coordinate (7)
	(0,-7) coordinate (8)
	(0,-8) coordinate (9)
	(0,-8.75) coordinate (n)
	(0,-9.5) coordinate (0);
\draw (1) -- (2) -- (3) -- (4) -- (5) -- (6) -- (7) -- (8) -- (9) -- (n);
}

\def\axgridL{
\path
	(0,0) coordinate (1)
	(0,-1) coordinate (2)
	(0,-2) coordinate (3)
	(0,-3) coordinate (4)
	(0,-4) coordinate (5)
	(0,-5) coordinate (6)
	(0,-6) coordinate (7)
	(0,-7) coordinate (8)
	(0,-8) coordinate (9)
	(0,-9) coordinate (10)
	(0,-9.75) coordinate (n)
	(0,-10.5) coordinate (0);
\draw (1) -- (2) -- (3) -- (4) -- (5) -- (6) -- (7) -- (8) -- (9) -- (10) -- (n);
}

\def\axgridM{
\path
	(0,0) coordinate (1)
	(1,0) coordinate (2)
	(1,-1) coordinate (3)
	(1,-2) coordinate (4)
	(1,-3) coordinate (5)
	(1,-4) coordinate (6)
	(1,-5) coordinate (7)
	(1,-6) coordinate (8)	
	(0,-6) coordinate (9)
	(0,-5) coordinate (10)
	(0,-4) coordinate (11)
	(0,-3) coordinate (12)
	(0,-2) coordinate (13)
	(0,-1) coordinate (14);
\draw (1) -- (2) -- (3) -- (4) -- (5) -- (6) -- (7) -- (8) -- (9) -- (10) -- (11) -- (12) -- (13) -- (14) -- (1);
}

\def\axgridN{
\path
	(0,0) coordinate (1)
	(1,0) coordinate (2)
	(1,-1) coordinate (3)
	(1,-2) coordinate (4)
	(1,-3) coordinate (5)
	(1,-4) coordinate (6)
	(1,-5) coordinate (7)
	(0,-5) coordinate (8)	
	(0,-4) coordinate (9)
	(0,-3) coordinate (10)
	(0,-2) coordinate (11)
	(0,-1) coordinate (12);
\draw (1) -- (2) -- (3) -- (4) -- (5) -- (6) -- (7) -- (8) -- (9) -- (10) -- (11) -- (12) -- (1);
}

\def\axgridO{
\path
	(0,0) coordinate (1)
	(1,0) coordinate (2)
	(1,-1) coordinate (3)
	(1,-2) coordinate (4)
	(1,-3) coordinate (5)
	(1,-4) coordinate (6)
	(1,-5) coordinate (7)
	(1,-6) coordinate (8)	
	(1,-7) coordinate (9)	
	(0,-7) coordinate (10)
	(0,-6) coordinate (11)
	(0,-5) coordinate (12)
	(0,-4) coordinate (13)
	(0,-3) coordinate (14)
	(0,-2) coordinate (15)
	(0,-1) coordinate (16);
\draw (1) -- (2) -- (3) -- (4) -- (5) -- (6) -- (7) -- (8) -- (9) -- (10) -- (11) -- (12) -- (13) -- (14) -- (15) -- (16) -- (1);
}

\def\axgridP{
\path
	(0,0) coordinate (1)
	(1,0) coordinate (2)
	(1,-1) coordinate (3)
	(1,-2) coordinate (4)
	(1,-3) coordinate (5)
	(1,-4) coordinate (6)
	(1,-5) coordinate (7)
	(1,-6) coordinate (8)	
	(1,-7) coordinate (9)	
	(1,-8) coordinate (10)	
	(0,-8) coordinate (11)
	(0,-7) coordinate (12)
	(0,-6) coordinate (13)
	(0,-5) coordinate (14)
	(0,-4) coordinate (15)
	(0,-3) coordinate (16)
	(0,-2) coordinate (17)
	(0,-1) coordinate (18);
\draw (1) -- (2) -- (3) -- (4) -- (5) -- (6) -- (7) -- (8) -- (9) -- (10) -- (11) -- (12) -- (13) -- (14) -- (15) -- (16) -- (17) -- (18) -- (1);
}

\def\axgridQ{
\path
	(0,0) coordinate (1)
	(1,0) coordinate (2)
	(1,-1) coordinate (3)
	(1,-2) coordinate (4)
	(1,-3) coordinate (5)
	(1,-4) coordinate (6)
	(1,-5) coordinate (7)
	(1,-6) coordinate (8)	
	(1,-7) coordinate (9)	
	(1,-8) coordinate (10)	
	(1,-9) coordinate (11)	
	(1,-10) coordinate (12)	
	(0,-10) coordinate (13)
	(0,-9) coordinate (14)
	(0,-8) coordinate (15)
	(0,-7) coordinate (16)
	(0,-6) coordinate (17)
	(0,-5) coordinate (18)
	(0,-4) coordinate (19)
	(0,-3) coordinate (20)
	(0,-2) coordinate (21)
	(0,-1) coordinate (22);
\draw (1) -- (2) -- (3) -- (4) -- (5) -- (6) -- (7) -- (8) -- (9) -- (10) -- (11) -- (12) -- (13) -- (14) -- (15) -- (16) -- (17) -- (18) -- (19) -- (20) -- (21) -- (22) -- (1);
}

\def\axgridR{
\path
	(0,0) coordinate (1)
	(0,-1) coordinate (2)
	(0,-2) coordinate (3)
	(0,-3) coordinate (4)
	(0,-3.75) coordinate (n);
\draw (1) -- (2) -- (3) -- (4) -- (n);
}

\def\axgridS{
\path
	(0,0) coordinate (1)
	(0,-1) coordinate (2)
	(0,-2) coordinate (3)
	(0,-3) coordinate (4)
	(0,-4) coordinate (5)
	(0,-5) coordinate (6)
	(0,-6) coordinate (7)
	(0,-6.75) coordinate (n);
\draw (1) -- (2) -- (3) -- (4) -- (5) -- (6) -- (7) -- (n);
}

\def\axgridT{
\path
	(0,0) coordinate (1)
	(0,-1) coordinate (2)
	(0,-2) coordinate (3)
	(0,-3) coordinate (4)
	(0,-4) coordinate (5)
	(0,-5) coordinate (6)
	(0,-6) coordinate (7)
	(0,-7) coordinate (8)
	(0,-8) coordinate (9)
	(0,-9) coordinate (10)
	(0,-9.75) coordinate (n);
\draw (1) -- (2) -- (3) -- (4) -- (5) -- (6) -- (7) -- (8) -- (9) -- (10) -- (n);
}

\def\axgridU{
\path
	(0,0) coordinate (1)
	(0,-1) coordinate (2)
	(0,-2) coordinate (3)
	(0,-3) coordinate (4)
	(0,-4) coordinate (5)
	(0,-5) coordinate (6)
	(0,-6) coordinate (7)
	(0,-7) coordinate (8)
	(0,-7.75) coordinate (n);
\draw (1) -- (2) -- (3) -- (4) -- (5) -- (6) -- (7) -- (8) -- (n);
}

\def\axgridV{
\path
	(0,0) coordinate (1)
	(0,-1) coordinate (2)
	(0,-2) coordinate (3)
	(0,-3) coordinate (4)
	(0,-4) coordinate (5)
	(0,-5) coordinate (6)
	(0,-6) coordinate (7)
	(0,-7) coordinate (8)
	(0,-8) coordinate (9)
	(0,-9) coordinate (10)
	(0,-9.75) coordinate (n);
\draw (1) -- (2) -- (3) -- (4) -- (5) -- (6) -- (7) -- (8) -- (9) -- (10) -- (n);
}

\def\axgridW{
\path
	(0,0) coordinate (1)
	(0,-1) coordinate (2)
	(0,-2) coordinate (3)
	(0,-3) coordinate (4)
	(0,-4) coordinate (5)
	(0,-5) coordinate (6)
	(0,-5.75) coordinate (n);
\draw (1) -- (2) -- (3) -- (4) -- (5) -- (6) -- (n);
}

\def\axgridX{
\path
	(0,0) coordinate (1)
	(0,-1) coordinate (2)
	(0,-2) coordinate (3)
	(0,-3) coordinate (4)
	(0,-4) coordinate (5)
	(0,-5) coordinate (6)
	(0,-5.75) coordinate (n);
\draw (1) -- (2) -- (3) -- (4) -- (5) -- (6) -- (n);
}

\def\axgridY{
\path
	(0,0) coordinate (1)
	(0,-1) coordinate (2)
	(0,-2) coordinate (3)
	(0,-3) coordinate (4)
	(0,-4) coordinate (5)
	(0,-5) coordinate (6)
	(0,-6) coordinate (7)
	(0,-7) coordinate (8)
	(0,-8) coordinate (9)
	(0,-9) coordinate (10)
	(0,-9.75) coordinate (n);
\draw (1) -- (2) -- (3) -- (4) -- (5) -- (6) -- (7) -- (8) -- (9) -- (10) -- (n);
}

\def\axgridZ{
\path
	(0,0) coordinate (1)
	(0,-1) coordinate (2)
	(0,-2) coordinate (3)
	(0,-3) coordinate (4)
	(0,-4) coordinate (5)
	(0,-5) coordinate (6)
	(0,-6) coordinate (7)
	(0,-7) coordinate (8)
	(0,-8) coordinate (9)
	(0,-9) coordinate (10)
	(0,-10) coordinate (11)
	(0,-10.75) coordinate (n);
\draw (1) -- (2) -- (3) -- (4) -- (5) -- (6) -- (7) -- (8) -- (9) -- (10) -- (11) -- (n);
}

\def\axgridRefl{
\path
	(0,0) coordinate (2)
	(0,-1) coordinate (51)
	(0,-2) coordinate (52)
	(0,-3) coordinate (53)
	(0,-4) coordinate (6)
	(0,-5) coordinate (7)
	(0,-6) coordinate (8)
	(0,-7) coordinate (9)
	(0,-8) coordinate (11)
	(0,-9) coordinate (13)
	(0,-10) coordinate (15)
	(0,-11) coordinate (16)
	(0,-12) coordinate (18)
	(0,-13) coordinate (19)
	(0,-14) coordinate (20)
	(0,-15) coordinate (21)
	(0,-16) coordinate (22)
	(0,-17) coordinate (23)
	(0,-18) coordinate (25)
	(0,-19) coordinate (26);
	
}

\def\tablegridAA{
\path
	(0,0) coordinate (1)
	(0,-.75) coordinate (2)
	(0,-1.75) coordinate (3);
}

\def\tablegridBB{
\path
	(0,0) coordinate (1)
	(0,-.75) coordinate (2)
	(0,-1.75) coordinate (3)
	(0,-2.75) coordinate (4)
	(0,-3.75) coordinate (5)
	(0,-4.75) coordinate (6)
	(0,-5.75) coordinate (7)
	(0,-6.5) coordinate (8)
	(0,-7.25) coordinate (9);
}

\def\tablegridCC{
\path
	(0.02,0) coordinate (1)
	(0.02,-1) coordinate (2)
	(0,-2.25) coordinate (3)
	(0,-3.5) coordinate (4)
	(0,-4.5) coordinate (5)
	(0,-5.5) coordinate (6)
	(0,-6.5) coordinate (7)
	(0,-7.5) coordinate (8)
	(1,0) coordinate (9)
	(1,-1) coordinate (10)
	(1,-2) coordinate (11)
	(1,-3) coordinate (12)
	(1,-4) coordinate (13)
	(1,-5) coordinate (14)
	(1,-6) coordinate (15)
	(1,-6.75) coordinate (16)
	(1,-7.5) coordinate (17)
	(1,-8.25) coordinate (18);
}

\def\tablegridDD{
\path
	(0.01,0) coordinate (1)
	(0.01,-1.25) coordinate (2)
	(0.01,-2.625) coordinate (3)
	(0,-4) coordinate (4)
	(0,-5.25) coordinate (5)
	(0.01,-6.5) coordinate (6)
	(0.01,-7.5) coordinate (7)
	(0.01,-8.5) coordinate (8)
	(0.01,-9.5) coordinate (9)
	(0,-10.5) coordinate (10)
	(0,-11.5) coordinate (11)
	(1.06,0) coordinate (12)
	(1.06,-1.375) coordinate (13)
	(1.05,-2.75) coordinate (14)
	(1.05,-3.75) coordinate (15)
	(1.05,-4.75) coordinate (16)
	(1.05,-5.75) coordinate (17)
	(1.05,-6.75) coordinate (18)
	(1.05,-7.75) coordinate (19)
	(1.05,-8.75) coordinate (20)
	(1.05,-9.75) coordinate (21)
	(1.05,-10.75) coordinate (22)
	(1.05,-11.75) coordinate (23)
	(1.05,-12.75) coordinate (24)
	(2.25,0) coordinate (25)
	(2.25,-1) coordinate (26)
	(2.25,-2) coordinate (27)
	(2.25,-3) coordinate (28)
	(2.25,-4) coordinate (29)
	(2.25,-5) coordinate (30)
	(2.25,-6) coordinate (31)
	(2.25,-7) coordinate (32)
	(2.25,-8) coordinate (33)
	(2.25,-9) coordinate (34)
	(2.25,-10) coordinate (35)
	(2.25,-11) coordinate (36)
	(2.25,-12) coordinate (37)
	(2.25,-13) coordinate (38)
	(2.25,-14) coordinate (39)
	(3.75,0) coordinate (40)
	(3.75,-1) coordinate (41)
	(3.75,-2) coordinate (42)
	(3.75,-3) coordinate (43)
	(3.75,-4) coordinate (44)
	(3.75,-5) coordinate (45)
	(3.75,-6) coordinate (46);
}

\begin{document}

\author{Christopher L. Douglas}
\address{Department of Mathematics, University of Oxford, Oxford OX1 3LB, UK}
\email{cdouglas@maths.ox.ac.uk}

\author{Andr\'e G. Henriques}
\address{Mathematisch Instituut, Universiteit Utrecht, 3508 TA Utrecht, NL}
\email{a.g.henriques@uu.nl}

\title{Internal bicategories}

\begin{abstract}
We define bicategories internal to 2-categories.  When the ambient 2-category is that of symmetric monoidal categories, we regard this as a framework for encoding the structure of a symmetric monoidal 3-category.  This framework is well suited to examples arising in geometry and algebra, such as the 3-category of bordisms or the 3-category of conformal nets.
\end{abstract}

\maketitle

%%%%%%%%%%

\xyoption{poly}
\xyoption{2cell}
\objectmargin{5pt}

\setcounter{tocdepth}{2}

\vspace*{-20pt}
\tableofcontents

\vspace*{-20pt}
Our purpose in this paper is to introduce a collection of notions that conveniently encode 3-categorical structures arising in geometry and algebra.  We begin, in Section~\ref{sec-lft}, by discussing the examples of primary interest: the bordism 3-category of $0$-, $1$-, $2$-, and $3$-dimensional manifolds, the 3-category of tensor categories and bimodule categories,
and the 3-category of conformal nets.  This last example (studied extensively in the papers~\cite{BDH(nets),BDH(modularity),BDH(1*1),BDH(3-category)}) was the principle motivation for this work: in considering the question, ``What kind of 3-categorical structure do conformal nets form?", we were led ineluctably to the categorical notions in the present paper.

Each structure we describe is, for some $n$ and $k$, a notion of $n$-category defined as a weak $k$-category internal to the $(n-k+1)$-category of strict $(n-k)$-categories.  We recall, in Section~\ref{sec-internalcats}, the concept of categories in $2$-categories, and then define categories in $3$-categories.  Next, in Section~\ref{sec-internalbicats}, we define bicategories in $2$-categories and the accompanying, stricter, notions of $2$-categories in 2-categories and dicategories in $2$-categories. (Recall $2$-categories are bicategories where the associator and identity transformations are strict; we use the term `dicategory' for a bicategory where only the associator transformation is strict.)  We relate these various structures by showing that an internal $2$-category has an associated internal dicategory, which in turn has an associated internal bicategory; we also show that a category in the 3-category of $2$-categories has an associated bicategory in the $2$-category of categories.  Finally, we connect the given notions to more traditional categorical structures by observing that bicategories in the $2$-category of categories have an underlying tricategory.

\section{Motivating examples} \label{sec-lft}

We describe the primary examples that motivated our investigation of internal higher categories: bordism higher categories, the 2-categorical structure of algebras and bimodules, the 3-categorical structure of tensor categories and bimodule categories, and the 3-categorical structure of conformal nets.  By examining the natural structures that are present in these examples, we motivate the notions of internal categories and internal bicategories, the study of which will occupy the bulk of the paper.  

We then describe a framework that organizes the various categorical structures encountered in the examples: weak $k$-categories internal to the $(n-k+1)$-category of strict $(n-k)$-categories.  If one replaces the $(n-k+1)$-category of $(n-k)$-categories by the $(n-k+1)$-category of symmetric monoidal $(n-k)$-categories, these internal weak higher categories provide notions of symmetric monoidal $n$-categories.  In the case of our primary concern, for conformal nets, $n=3$ and $k=2$: the notion of bicategory in the 2-category of symmetric monoidal categories furnishes a type of symmetric monoidal 3-category.

\subsection{Geometric examples} \label{sec-localgeo} 
\nopagebreak

\nopagebreak
\subsubsection*{Dimension $1+\epsilon$}

The usual bordism category $\Bord_0^1$ has objects $0$-manifolds and morphisms $1$-manifolds (bordisms) up to diffeomorphism rel boundary.  Instead of quotienting out by the diffeomorphisms of $1$-manifolds to form $\Bord_0^1$, we can take these diffeomorphisms into account in our definition of the bordism category.  Even the $0$-manifolds in the bordism category have diffeomorphisms, and we can incorporate these at the same time.  The resulting structure is as follows.  Zero-manifolds and their diffeomorphisms form a category $B_0$.  There is a second category $B_1$ whose objects are, roughly speaking, $1$-manifolds equipped with a partition of the boundary into two pieces, the source and target, and whose morphisms are diffeomorphisms.  The source and target are functors $s,t: B_1 \ra B_0$, and there is a composition functor $B_1 \times_{B_0} B_1 \ra B_1$.  Altogether, the pair $(B_0, B_1)$ forms a \emph{category object in $\CAT$}, also called a \emph{category internal to $\CAT$}, or simply a \emph{category in $\CAT$}---this notion is defined precisely in Section~\ref{sec-catin2cat}.  This category object is abbreviated $\bBord_0^1$, to indicate that it is an enrichment of the usual bordism category of $0$- and $1$-manifolds.  In fact, both the category $B_0$ and the category $B_1$ are symmetric monoidal, under disjoint union, and $\bBord_0^1 = (B_0, B_1)$ forms a category object in $\SMC$, the $2$-category of symmetric monoidal categories. 

\subsubsection*{Dimension $2$}

There is a bicategory whose objects are $0$-manifolds, whose 1-cells are $1$-manifolds (bordisms), and whose $2$-cells are $2$-manifolds (bordisms between bordisms) up to diffeomorphism.  This bicategory and its symmetric monoidal structure can be conveniently encoded as a category object in $\SMC$, as follows.  The symmetric monoidal category $B_0$ is again 0-manifolds and their diffeomorphisms.  There is a second symmetric monoidal category $B_1^2$ whose objects are $1$-manifolds with a partition of their boundary and whose morphisms are $2$-manifolds with boundary written as a union of a source $1$-manifold and target $1$-manifold, up to diffeomorphism.  The pair $(B_0, B_1^2)$ forms a category object in $\SMC$, denoted $\Bord_0^2$.  Note that there are various distinct ways to give precise models for $B_1^2$, the choices occurring primarily in the treatment of corners and in the definition of identity bordisms.

\subsubsection*{Dimension $2+\epsilon$}

In the category object $\Bord_0^2$, diffeomorphic $2$-manifolds are considered equivalent.  We can refine this bordism category to incorporate diffeomorphisms of $2$-manifolds.  Consider the category $B_0$ of 0-manifolds and diffeomorphisms, the category $B_1$ of 1-manifolds (with source and target $0$-manifolds) and diffeomorphisms, and the category $B_2$ of 2-manifolds (with source and target $1$-manifolds) and diffeomorphisms.  As before there are source and target functors $s,t: B_1 \ra B_0$, and a composition functor $B_1 \times_{B_0} B_1 \ra B_1$; now in addition there are source and target functors $s,t: B_2 \ra B_1$ and a composition functor $B_2 \times_{B_1} B_2 \ra B_2$.  Altogether the triple $(B_0, B_1, B_2)$ forms a \emph{bicategory object in $\CAT$}, also called a \emph{bicategory internal to $\CAT$}, or simply a \emph{bicategory in $\CAT$}---bicategory objects are defined in Section~\ref{sec-internalbicats}.  We denote this bordism bicategory object by $\bBord_0^2$.  By making careful choices in the formulation of the categories $B_1$ and $B_2$, it is possible to construct $\bBord_0^2$ as a \emph{2-category object in $\CAT$}.  In a 2-category object, the composition and identity functors are stricter than in a bicategory object---see the discussion at the beginning of Section~\ref{sec-internalbicats}.  The categories $B_0$, $B_1$, and $B_2$ are symmetric monoidal, and we can easily incorporate this structure, observing that $\bBord_0^2$ is in fact a bicategory in $\SMC$.

\subsubsection*{Dimension $3$}

There is a tricategory whose objects, $1$-cells, $2$-cells, and $3$-cells are $0$-, $1$-, $2$-, and $3$-manifolds, respectively.  We can reformulate this bordism tricategory in the framework of bicategory objects in symmetric monoidal categories.  As before $B_0$, respectively $B_1$, is the symmetric monoidal category of $0$-manifolds and diffeomorphisms, respectively $1$-manifolds and diffeomorphisms.  There is a third symmetric monoidal category $B_2^3$ whose objects are $2$-manifolds (with as before source and target $1$-manifolds) and whose morphisms are 3-dimensional bordisms, up to diffeomorphism.  Altogether the triple $(B_0, B_1, B_2^3)$ forms a bicategory object in $\SMC$, denoted $\Bord_0^3$.

\subsection{Algebraic examples} \label{sec-localalg} 
\nopagebreak

\nopagebreak
\subsubsection*{Dimension $2$}
We consider symmetric monoidal bicategories $C$ with the property that $\Hom_C(1,1)$ is equivalent to the category of vector spaces.  We refer to this property by saying that ``$C$ deloops the category of vector spaces", and write $\Omega C \simeq \Vect$.  A category does not have a unique deloop, but rather a few natural deloops, that serve different purposes.

\emph{2-vector spaces.}
One deloop of vector spaces is the 2-category $\TwVect$ of 2-vector spaces.  Recall that $\Vect$ is a symmetric bimonoidal (that is ``commutative ring") category, and a 2-vector space is, roughly speaking, a symmetric monoidal category equipped with the structure of a module over $\Vect$; often one restricts attention to the free finitely generated $\Vect$-modules, or even for simplicity to the specific modules $\Vect^n$~\cite{kapvoev,elgueta}.  More generally and more precisely, one may define a 2-vector space as a semisimple $\mathbb C$-linear abelian category.  The 1-morphisms of $\TwVect$ are linear functors, and the 2-morphisms of $\TwVect$ are linear natural transformations.  The unit object of $\TwVect$ is the category $\Vect$ itself.
A $\Vect$-module functor from $\Vect$ to itself is determined by the image of the object $\CC$, thus by a choice of an object of $\Vect$.  Therefore $\Hom_{\TwVect}(1,1)$ is equivalent to the category of vector spaces, i.e.\ $\TwVect$ is a deloop of $\Vect$.  Note that $\TwVect$ can be viewed as a category object in symmetric monoidal categories, where the first category is 2-vector spaces with module functors, and the second category is module functors with module natural transformations. 

\emph{Algebras.}
The other prominent deloop of $\Vect$, besides 2-vector spaces, is the bicategory of algebras, bimodules, and maps of bimodules.  The unit object in this bicategory is the trivial algebra $\CC$.  Note that indeed the category of $\CC$-$\CC$-bimodules is the category of vector spaces.  Algebras naturally form a symmetric monoidal category $\Alg_0$ whose morphisms are algebra isomorphisms.  Bimodules and their maps also form a symmetric monoidal category $\Alg_1$.  We see that the bicategory of algebras, bimodules, and maps is a shadow of the category object $\Alg=(\Alg_0,\Alg_1)$ in symmetric monoidal categories.  Note that there are variations on this example---for instance von Neumann algebras and their bimodules provide a category object delooping, not vector spaces per se, but Hilbert spaces.

\subsubsection*{Dimension $3$}

There were two natural deloops of the category $\Vect$, namely the 2-category $\TwVect$ and the category object $\Alg$.
Similarly, there are multiple interesting double deloops of $\Vect$---we now describe two, the 3-category of 3-vector spaces and the category object of tensor categories, and we also describe the bicategory object of conformal nets, which deloops von Neumann algebras and therefore double deloops Hilbert spaces.

\emph{3-vector spaces.}
The most direct way to double-deloop the category of vector spaces is to consider 3-vector spaces.  The 2-category $\TwVect$ of 2-vector spaces is symmetric monoidal, and a 3-vector space is, roughly speaking, a 2-category that is a module over $\TwVect$; as with 2-vector spaces, it is usually best to restrict attention to the 3-vector spaces that are free finitely-generated $\TwVect$-modules.  We have not endeavored to make precise the symmetric monoidal 3-category of 3-vector spaces.

\emph{Tensor categories.}
A notion of 3-vector space came from considering modules over the monoidal 2-category $\TwVect$.  Instead we can consider algebra objects in $\TwVect$---these are called tensor categories.  Together with functors and natural transformations, tensor categories form a 2-category $\TC_0$.  Bimodules between tensor categories, with their functors and natural transformations, form another 2-category $\TC_1$.  We expect that the pair $\TC:=(\TC_0,\TC_1)$ forms a category object in the 3-category of 2-categories---this notion of category object in a 3-category is described in Section~\ref{sec-catin3cat}.  As we will see in Sections~\ref{sec-catin2catarebicatincat} and~\ref{app-bicatincataretricat}, any category object in the 3-category of 2-categories has an underlying tricategory.  Indeed there is a tricategory of tensor categories, built as a symmetric monoidal $(\infty,3)$-category by a construction of Johnson-Freyd--Scheimbauer~\cite{jfs}, or (restricting attention to fusion categories) as a Gordon--Powers--Street-style tricategory by Schaumann~\cite{schaumann}.  Etingof, Nikshych, and Ostrik have also investigated a tricategory of tensor categories, and have used this structure to understand extensions of tensor categories and even to construct novel tensor categories~\cite{eno-fcht}.

\emph{Conformal nets.}
A 3-vector space is in particular a 2-category, and a tensor category is in particular a category equipped with a multiplication operation.  We might wonder if there is a double deloop of $\Vect$ whose objects are yet less categorical, for instance whose objects are vector spaces equipped with not one but two multiplication operations.  Indeed, conformal nets provide such a double deloop, not exactly of vector spaces but of Hilbert spaces.  Conformal nets are a mathematical formalization of the notion of a conformal field theory~\cite{Gabbiani-Froehlich(OperatorAlg-CFT),Longo(Lectures-on-Nets),Wassermann(Operator-algebras-and-conformal-field-theory),cklw}.  We sketch the notion of conformal net, and then briefly describe the symmetric monoidal 3-category of conformal nets as a bicategory object in symmetric monoidal categories---the definition of the precise relevant notion of conformal nets (finite-index coordinate-free conformal nets), the construction of the necessary associated structures of defects, sectors, and fusion operations, and the proof that these structures indeed form a bicategory object in symmetric monoidal categories all appear in a series of papers with Arthur Bartels~\cite{BDH(nets),BDH(modularity),BDH(1*1),BDH(3-category)}.  This result is a formulation of the idea that conformal field theories form a symmetric monoidal 3-category.

A conformal net is a functor $\cA:\Int\ra\vNalg$ from the category of intervals to the category of von Neumann algebras.  Here, an interval is a compact, contractible 1-manifold, and a morphism of intervals is an embedding.  The functor must satisfy various conditions, including: (1) the subalgebras $\cA([0,1])$ and $\cA([1,2])$ commute inside $\cA([0,2])$ and they generate a dense subalgebra, (2) the closed subalgebra generated by $\cA([0,1])$ and $\cA([2,3])$ inside $\cA[0,3]$ is isomorphic to the spatial tensor product $\cA([0,1])\,\bar\otimes\,\cA([2,3])$, and (3) if a diffeomorphism $\varphi:[0,1]\to [0,1]$ satisfies $\varphi'|_{[0,\epsilon]\cup[1-\epsilon,1]}=1$ then the automorphism $\cA(\varphi):\cA([0,1])\to \cA([0,1])$ is inner.  The algebra $\cA([0,1])$ can be equipped with a second multiplication that sends $a,b\in \cA([0,1])$ to $\cA(\varphi)(a)\,\cA(\varphi\!+\!\frac12)(b)$, where $\varphi:[0,1]\to [0,\frac12]$ is a diffeomorphism satisfying $\varphi'|_{[0,\epsilon]\cup[1-\epsilon,1]}=1$. That second multiplication is associative up to conjugation with respect to the usual algebra multiplication.

Conformal nets are the objects of a tricategory.  The 1-cells, called \emph{defects}, can be thought of, roughly, as bimodules between two nets, each considered with respect to the second multiplication described above.  More precisely, defects are defined as functors $D:\Int_{\mathrm{bicol}} \ra \vNalg$ from a cerain category of bicolored intervals to the category of von Neumann algebras.  Here, a bicolored interval is an interval equipped with a partition into two (possibly empty) intervals, called the white and black subintervals. The restriction of a defect $D$ to the purely white (respectively purely black) intervals in $\Int_{\mathrm{bicol}}$ is the source (respectively target) conformal net of that defect.  Given two defects $D$ and $E$ from the net $\cA$ to the net $\cB$, a 2-cell, called a \emph{sector}, from $D$ to $E$ is a Hilbert space equipped with compatible representations of the following collection of von Neumann algebras associated to intervals $I$ contained in the unit circle $S^1 \subset \CC$: the algebras that act on the sector are $\cA(I)$ for $I\subset \CC_{\mathrm{Re}<0}$, $\cB(I)$ for $I\subset \CC_{\mathrm{Re}>0}$, $D(I)$ for $i\in I$ and $-i\not \in I$, and $E(I)$ for $-i\in I$ and $i\not \in I$.  Sectors must also have the property that the actions of $D(I)$ and $E(J)$ commute if $I$ and $J$ are disjoint.

Conformal nets form a symmetric monoidal category $\CN_0$ whose morphisms are natural isomorphisms of functors from intervals to von Neumann algebras. Similarly, defects form a symmetric monoidal category $\CN_1$.  Sectors form a symmetric monoidal category $\CN_2$ whose morphisms are maps of Hilbert spaces intertwining all the von Neumann algebra actions.  There are source and target functors $s,t: \CN_1 \ra \CN_0$ and source and target functors $s,t: \CN_2 \ra \CN_1$.  Moreover, (provided one restricts to finite-index conformal nets) there are composition functors $\CN_1 \times_{\CN_0} \CN_1 \ra \CN_1$ and $\CN_2 \times_{\CN_1} \CN_2 \ra \CN_2$, and action functors $\CN_2 \times_{\CN_0} \CN_1 \ra \CN_2$ and $\CN_1 \times_{\CN_0} \CN_2 \ra \CN_2$.
(See \cite[Appendix~A]{BDH(1*1)} for more operations relating $\CN_0$, $\CN_1$, $\CN_2$, and various fiber products thereof.)
Altogether, the triple $\CN:=(\CN_0, \CN_1, \CN_2)$ is a bicategory object, actually a dicategory object, in symmetric monoidal categories:
\begin{theorem}[\cite{BDH(nets),BDH(modularity),BDH(1*1),BDH(3-category)}]\label{thm-cn}
Finite-index conformal nets, defects, sectors, and intertwiners form the 0-, 1-, 2-, and 3-cells of a dicategory object in the $2$-category of symmetric monoidal categories.
\end{theorem} \vspace{-1pt}
\nid 
By forgetting from symmetric monoidal to ordinary categories and applying Proposition~\ref{prop-dicatisbicat} and Theorem~\ref{thm-bicatincataretricat} (see Figure~\ref{fig-results} below), we have an underlying tricategory:
\begin{cor}\label{cor: CN is a tricategory}
Finite-index conformal nets, defects, sectors, and intertwiners form the 0-, 1-, 2-, and 3-cells of a tricategory.
\end{cor}

\begin{remark}
The two deloops of $\Vect$, namely $\TwVect$ and $\Alg$ are both frameworks for 2-dimensional algebra, but they have different characters.  The first deloop, $\TwVect$, is a categorification of $\Vect$, in that the objects are no longer built out of sets but out of categories; by contrast, the second deloop, $\Alg$, is an algebraification of $\Vect$, in that the objects still have underlying sets, but have in addition a multiplication operation.  Similarly, the three double deloops of $\Vect$ or $\Hilb$, namely $\ThVect$, $\TC$, and $\CN$, are all frameworks for 3-dimensional algebra.  The first, $\ThVect$, is a categorification of $\TwVect$, in that the objects are no longer based on categories, but on 2-categories.  The second, $\TC$, is an algebraification of $\TwVect$, in that the objects still have underlying categories but in addition have a multiplication operation; alternately, $\TC$ may be viewed as a categorification of $\Alg$, in that the objects have the structure of algebras, but are built on categories rather than sets.  Finally, the third double deloop, $\CN$, is an algebraification of $\Alg$ (actually of von Neumann algebras), in the sense that the objects still have underlying sets, but have not one, but two distinct multiplication operations.
\end{remark}

\subsection{Internal higher categories} \label{sec-semiweak}

We now discuss an organizational framework, namely weak higher categories internal to a higher category of strict higher categories, that accommodates all the aforementioned examples.

A strict $n$-category has \emph{morphisms}, which, like functions, compose in a strictly associative manner.  A weak $n$-category has \emph{cells}, the composition of which is only weakly associative.  We observe that by mixing cells and morphisms, one finds a family of notions of $n$-category intermediate between the strict and weak extremes.  These types of $n$-category are conveniently described as weak $k$-categories internal to the $(n-k+1)$-category of strict $(n-k)$-categories.

We begin with $1$-categories and work our way up.   A $1$-categorical structure has only objects and $1$-arrows, and in principle we might ask whether the $1$-arrows compose strictly.  However, in this case there are no $2$-arrows to give meaning to weakly associative composition, and the usual notion of 1-category is both maximally strict and maximally weak.

A $2$-categorical structure has objects, $1$-arrows, and $2$-arrows.  If the $1$-arrows compose strictly, and the $1$-identities are strict, the structure is a $2$-category in the usual sense.  If by contrast composition of $1$-arrows is only associative up to invertible $2$-arrows and similarly $1$-identities are only identities up to invertible $2$-arrows, the structure is a bicategory.  There is an intermediate notion between $2$-categories and bicategories, namely a category object in categories.  Such a category object has a category $C_0$ of objects and a category $C_1$ of 1-cells, and therefore two distinct kinds of $1$-arrows: the ``1-morphisms", that is the 1-arrows of $C_0$, compose strictly, while the ``1-cells", that is the objects of $C_1$, compose weakly. (The idea of a 2-categorical structure in which there are two distinct types of 1-cells, one strict and one weak, is familiar and well-studied in the literature; see for instance~\cite{ehresmann,grandispare-limits}.)  The three kinds of $2$-categorical structures are indicated in the ``staircase" pictured in Figure~\ref{fig-staircase}.  In that figure, to emphasize the relationship between the various notions, we describe certain familiar notions in unfamiliar terms: a ``$0$-category in $\nCAT$" is simply a strict n-category, while a ``weak $n$-category in $\zCAT$" is a weak $n$-category.
(In particular, $0$-categories in $\CAT$ and categories in $\zCAT$ both refer to the usual notion of a $1$-category).

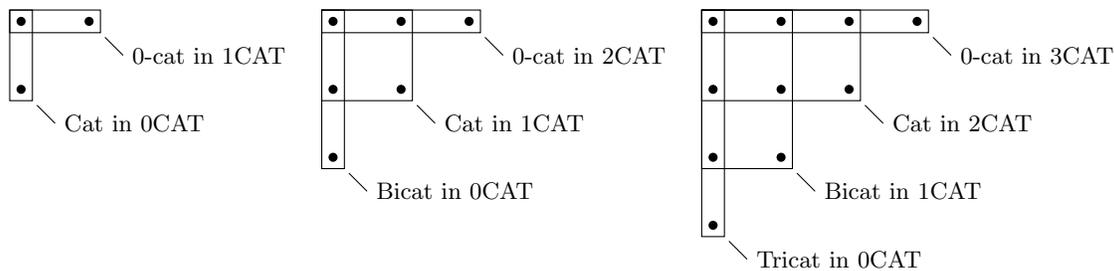
\begin{figure}[ht]
\db{
\begin{tikzpicture}[scale=.07ex] 
	\draw (0,0) rectangle (\cmlg,-\cmsh);
	\draw (0,0) rectangle (\cmsh,-\cmlg);
	\bull(.5*\cmsh,-.5*\cmsh);
	\bull(\cmlg-.5*\cmsh,-.5*\cmsh);
	\bull(.5*\cmsh,.5*\cmsh-\cmlg);
	\draw (\cmlg+\lbgap,-\cmsh-\lbgap) -- (\cmlg+\lbdist,-\cmsh-\lbdist) node[anchor=west] {\small 0-cat in $\oCAT$};
	\draw (\cmsh+\lbgap,-\cmlg-\lbgap) -- (\cmsh+\lbdist,-\cmlg-\lbdist) node[anchor=west] {\small $\Cat$ in $\zCAT$};
\end{tikzpicture}
}
\db{
\begin{tikzpicture}[scale=.07ex] 
	\draw (0,0) rectangle (2*\cmlg-\cmsh,-\cmsh);
	\draw (0,0) rectangle (\cmlg,-\cmlg);
	\draw (0,0) rectangle (\cmsh,-2*\cmlg+\cmsh);
	\bull(.5*\cmsh,-.5*\cmsh);
	\bull(\cmlg-.5*\cmsh,-.5*\cmsh);
	\bull(.5*\cmsh,.5*\cmsh-\cmlg);
	\bull(2*\cmlg-\cmsh-.5*\cmsh,-.5*\cmsh);
	\bull(\cmlg-.5*\cmsh,-\cmlg+.5*\cmsh);
	\bull(.5*\cmsh,-2*\cmlg+\cmsh+.5*\cmsh);
	\draw (2*\cmlg-\cmsh+\lbgap,-\cmsh-\lbgap) -- (2*\cmlg-\cmsh+\lbdist,-\cmsh-\lbdist) node[anchor=west] {\small 0-cat in $\tCAT$};
	\draw (\cmlg+\lbgap,-\cmlg-\lbgap) -- (\cmlg+\lbdist,-\cmlg-\lbdist) node[anchor=west] {\small $\Cat$ in $\oCAT$};
	\draw (\cmsh+\lbgap,-2*\cmlg+\cmsh-\lbgap) -- (\cmsh+\lbdist,-2*\cmlg+\cmsh-\lbdist) node[anchor=west] {\small Bicat in $\zCAT$};
\end{tikzpicture}
}
\db{
\begin{tikzpicture}[scale=.07ex] 
	\draw (0,0) rectangle (3*\cmlg-2*\cmsh,-\cmsh);
	\draw (0,0) rectangle (2*\cmlg-\cmsh,-\cmlg);
	\draw (0,0) rectangle (\cmlg,-2*\cmlg+\cmsh);
	\draw (0,0) rectangle (\cmsh,-3*\cmlg+2*\cmsh);
	\bull(.5*\cmsh,-.5*\cmsh);
	\bull(\cmlg-.5*\cmsh,-.5*\cmsh);
	\bull(.5*\cmsh,.5*\cmsh-\cmlg);
	\bull(2*\cmlg-\cmsh-.5*\cmsh,-.5*\cmsh);
	\bull(\cmlg-.5*\cmsh,-\cmlg+.5*\cmsh);
	\bull(.5*\cmsh,-2*\cmlg+\cmsh+.5*\cmsh);
	\bull(3*\cmlg-2*\cmsh-.5*\cmsh,-.5*\cmsh);
	\bull(2*\cmlg-\cmsh-.5*\cmsh,-\cmlg+.5*\cmsh);
	\bull(\cmlg-.5*\cmsh,-2*\cmlg+\cmsh+.5*\cmsh);
	\bull(.5*\cmsh,-3*\cmlg+2*\cmsh+.5*\cmsh);
	\draw (3*\cmlg-2*\cmsh+\lbgap,-\cmsh-\lbgap) -- (3*\cmlg-2*\cmsh+\lbdist,-\cmsh-\lbdist) node[anchor=west] {\small 0-cat in $\thCAT$};
	\draw (2*\cmlg-\cmsh+\lbgap,-\cmlg-\lbgap) -- (2*\cmlg-\cmsh+\lbdist,-\cmlg-\lbdist) node[anchor=west] {\small $\Cat$ in $\tCAT$};
	\draw (\cmlg+\lbgap,-2*\cmlg+\cmsh-\lbgap) -- (\cmlg+\lbdist,-2*\cmlg+\cmsh-\lbdist) node[anchor=west] {\small Bicat in $\oCAT$};	
	\draw (\cmsh+\lbgap,-3*\cmlg+2*\cmsh-\lbgap) -- (\cmsh+\lbdist,-3*\cmlg+2*\cmsh-\lbdist) node[anchor=west] {\small Tricat in $\zCAT$};	
\end{tikzpicture}
}

\caption{Interpolating between strict and weak higher categories} \label{fig-staircase}
\end{figure}

Now consider the $3$-categorical structures pictured in Figure~\ref{fig-staircase}.  At the two extremes, we have 3-categories and tricategories. (See Definition~\ref{def-tricat} for our notion of tricategory.)  In between we have two intermediate notions, bicategories in $\oCAT$, and categories in $\tCAT$.  A bicategory in $\oCAT$ has a category of objects, a category of 1-cells, and a category of 2-cells, and therefore has two distinct kinds of 1-arrows (namely 1-cells and 1-morphisms) and two distinct kinds of 2-arrows (namely 2-cells and 1-morphisms of 1-cells).  A category in $\tCAT$ has a 2-category of objects and a 2-category of 1-cells, and therefore also has two kinds of 1-arrows (the 1-cells and the 1-morphisms) and two kinds of 2-arrows (the 1-morphisms of 1-cells and the 2-morphisms).  Our main examples of $3$-dimensional categorical structures naturally separate into these various cases: 2-categories form a 3-category, tensor categories form a category in 2-categories, and conformal nets form a bicategory (actually a dicategory) in categories.  Note that there are a number of examples in the literature of notions of 3-categorical structures where 1-cells or 2-cells have multiple types.  For instance, Shulman~\cite{shulman-constructing} and Garner--Gurski~\cite{garnergurski} study respectively monoidal double categories and multi-object monoidal double categories, which (interpreted as 3-categorical structures) have one type of 1-arrows and two types of 2-arrows; Grandis--Par\'e~\cite{grandispare-cubicalglobular,grandispare-intercategories} investigate 3-cubical categories and `intercategories', which can have three types of 1-arrows and three types of 2-arrows.

One technical convenience of the internal higher category framework is that it easily accommodates symmetric monoidal structures.  For instance, in the 3-by-3 staircase of Figure~\ref{fig-staircase}, we can replace $\thCAT$, $\tCAT$, $\oCAT$, and $\zCAT$ by respectively any 4-, 3-, 2-, and 1-category.  In particular, we can consider bicategory objects in the 2-category SMC of symmetric monoidal categories---a bicategory object in SMC is a type of symmetric monoidal tricategory.

\subsection{Main results} \label{sec-main-result}

Figure~\ref{fig-staircase} lists four types of 3-categorical structures: 0-categories in $\thCAT$, 1-categories in $\tCAT$, bicategories in $\CAT$, and tricategories in $\zCAT$.  There are, however, a number of variations on these notions.  In Section~\ref{sec-catin3cat}, we will describe not only categories in $\tCAT$ but categories in a weaker 3-category ``$\tCATnwk$" in which the natural transformations are allowed to be weak rather than strict; later we also discuss ``strict" categories in $\tCAT$, in which the identity and associator transformations are strictly invertible, and a notion of ``strictly associatively strict" (sas) categories in $\tCATnwk$, in which the associator transformation is a strictly invertible strict natural transformation.  In Section~\ref{sec-internalbicats}, we will describe not only bicategories in $\CAT$, but also dicategories in $\CAT$, in which horizontal associators are somewhat strict, and 2-categories in $\CAT$, in which both horizontal associators and horizontal identities are somewhat strict.

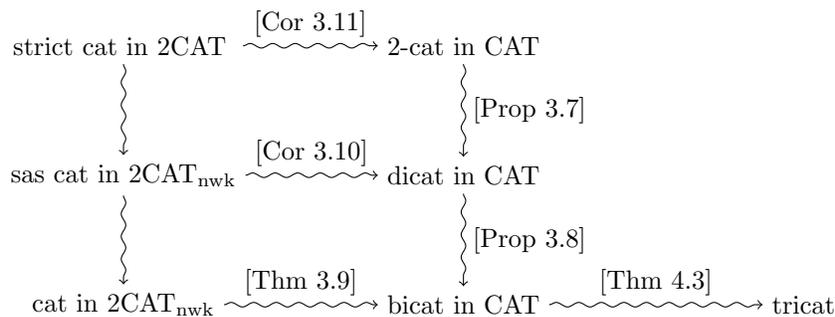
\begin{figure}[h!]
\begin{tikzpicture}[xscale=1.05ex,yscale=.4ex,auto,ssty/.style={->,decorate,decoration={snake,amplitude=.3mm,segment length=3mm,post length=.5mm}}] 
\node (cin2c) at (0,0) {strict cat in $\tCAT\phantom{.}$};
\node (2cinc) at (1,0) {2-cat in $\CAT$};
\node (ascin2c) at (0,-1) {sas cat in $\tCATnwk$};
%{\begin{minipage}{2.28cm}\begin{center}assoc strict cat\\ in $\tCATnwk$\end{center}\end{minipage}};
\node (dcinc) at (1,-1) {dicat in $\CAT$};
\node (cin2cnwk) at (0,-2) {cat in $\tCATnwk$};
\node (bcinc) at (1,-2) {bicat in $\CAT$};
\node (tri) at (2,-2) {tricat};
\draw [ssty] (cin2c) -- (ascin2c);
\draw [ssty] (ascin2c) -- (cin2cnwk); 
\draw [ssty] (cin2cnwk) -- node {[Thm~\ref{thm-catin2catarebicatincat}]} (bcinc); 
\draw [ssty] (bcinc) -- node {[Thm~\ref{thm-bicatincataretricat}]} (tri); 
\draw [ssty] (cin2c) -- node {[Cor~\ref{cor-catin2catare2catincat}]} (2cinc); 
\draw [ssty] (2cinc) -- node {[Prop~\ref{prop-2catisdicat}]} (dcinc); 
\draw [ssty] (dcinc) -- node {[Prop~\ref{prop-dicatisbicat}]} (bcinc);
\draw [ssty] (ascin2c) -- node {[Cor~\ref{cor-ascatin2cataredicatincat}]} (dcinc);
\end{tikzpicture}
\caption{Interrelations among 3-categorical structures} \label{fig-results}
\end{figure}

The main results of this paper connect these various notions by means of explicit constructions; these are displayed in Figure~\ref{fig-results}.  An unlabeled arrow indicates that the first notion is a special case of the second; a labelled arrow indicates that there is a canonical construction of a structure of the second type from one of the first type.

\section{Internal categories} \label{sec-internalcats}

We begin our investigation of internal weak higher categories by discussing internal categories in 2- and 3-categories.  Section~\ref{sec-catin2cat} concentrates on categories internal to 2-categories; that notion naturally encodes, for instance, the structure of 2-dimensional local bordism, and the structure of algebras, bimodules, and intertwiners.  In Section~\ref{sec-catin3cat} we introduce the rather more involved definition of categories internal to 3-categories; this is the appropriate framework for tensor categories, bimodule categories, functors between bimodule categories, and natural transformations between functors between bimodule categories. 

\subsection{Categories in 2-categories} \label{sec-catin2cat} ($\vcenter{\xymatrix@R=3pt@C=3pt@M=0pt{
\scriptscriptstyle\bullet & \scriptscriptstyle\bullet \\ \scriptscriptstyle\bullet & \scriptscriptstyle\bullet
}}$) 
\nopagebreak

\nopagebreak
The basic data of a category object $C$ in the 2-category of categories will be a ``category of objects" $C_0$ and a ``category of 1-cells" $C_1$.  We would like the 1-morphisms of $C$, that is the 1-arrows of $C_0$, to be closely related to the 1-cells of $C$, that is to the objects of $C_1$.  We can accomplish this by imposing a fibrancy condition on our category object, and for that we will need a notion of fibration in the 2-category of categories.  We will be restricting attention to the case when the category of objects $C_0$ is a groupoid; under this assumption various potential notions of fibration coincide.

\begin{definition}
For $A$ a category and $B$ a groupoid, a functor $F: A \ra B$ is a fibration if for all arrows $f: b \ra b'$ of $B$ and objects $a'$ of $A$ with $F(a') = b'$, there exists an invertible arrow $\tilde{f}: a \ra a'$ of $A$ with $F(\tilde{f})=f$.
\end{definition}

\nid This condition can be thought of as a categorical version of a homotopy lifting property.  

More generally, to define category objects in a 2-category, we will need to provide a notion of fibration in an arbitrary 2-category; this is accomplished via the Yoneda embedding.  An object $B$ of a 2-category $\cC$ is called a groupoid object if for all objects $T$ of $\cC$ the category $\Hom(T,B)$ is a groupoid.

\begin{definition}
For an object $A$ and a groupoid object $B$ in a 2-category $\cC$, an arrow $F: A \ra B$ is a fibration if for all objects $T$ of $\cC$ the functor $\Hom(T,A) \ra \Hom(T,B)$ is a fibration.  Explicitly, this holds if for all 2-arrows $f: b \dra b'$, with $b, b' \in \Hom(T,B)$, and $a' \in \Hom(T,A)$ with $F a'=b'$, there exists an invertible 2-arrow $\tilde{f}: a \dra a'$ with $F \tilde{f}=f$.
\end{definition}

We will be interested in the case where $\cC$ is the 2-category $\CAT$ of categories and the case where $\cC$ is the 2-category $\SMC$ of symmetric monoidal categories.  Notice that if an arrow in $\SMC$ is a fibration, then the underlying functor of categories is a fibration; this will imply that a category object in $\SMC$ is in particular a category object in $\CAT$.  We now proceed to make these notions precise.

The notion of a category object in a 2-category is a categorification of the notion of a category object in a 1-category.  To highlight this transition we list the two definitions in sequence.

\begin{definition} \label{def-catobj1cat}
A category object $C$ in the 1-category $\cC$ consists of the following two collections of data, satisfying the subsequent axioms.
\begin{description}
\item[0-data]
There is an object $C_0$ of $\cC$, denoted \cb{\ing{0stuff-item0.eps}}, and an object $C_1$ of $\cC$, denoted \cb{\ing{0stuff-item1.eps}}, together with morphisms $s,t: C_1 \ra C_0$.
\item[1-data]
There is a morphism $i: C_0 \ra C_1$, denoted \cb{\ing{1stuff-item1.eps}}, and a morphism $m: C_1 \times_{C_0} C_1 \ra C_1$, denoted \cb{\ing{1stuff-item2.eps}}.  These morphisms are compatible with the source and target maps in the sense that $s \circ i = t \circ i = \id$, $s \circ m = s \circ \pr_1$, and $t \circ m = t \circ \pr_2$.
\item[2-axioms] \hspace*{\fill}\\
%\begin{itemize}
%\item[] 
The morphism $m((i \circ s) \times \id): C_1 \ra C_1$ is the identity; pictorially 
\raisebox{\cboxht}{\xymatrix@C-12pt{\ing{2stuff-item15-1.eps}   \ar@{=}[r] &   \hb{\ing{2stuff-item15-2.eps}}}}.
%\item[] 
Similarly the morphism $m(\id \times (i \circ t)): C_1 \ra C_1$ is the identity; pictorially 
\raisebox{\cboxht}{\xymatrix@C-12pt{\ing{2stuff-item16-1.eps}   \ar@{=}[r] &   \hb{\ing{2stuff-item16-2.eps}}}}.
%\item[] 
Lastly, the morphisms $m(m \times \id): C_1 \times_{C_0} C_1 \times_{C_0} C_1 \ra C_1$ and $m(\id \times m): C_1 \times_{C_0} C_1 \times_{C_0} C_1 \ra C_1$ are equal; pictorially
\raisebox{\cboxht}{\xymatrix@C-12pt{\ing{2stuff-item12-1.eps}   \ar@{=}[r] &   \ing{2stuff-item12-2.eps}}}.
%\end{itemize}
\end{description}
\end{definition}

\begin{definition} \label{def-catobj2cat}
A category object $C$ in the 2-category $\cC$ consists of the following three collections of data, subject to the listed axioms.

\begin{description}

\item[0-data] 
There is a groupoid object $C_0$ of $\cC$, called the object of ``0-cells" and denoted by the picture \cb{\ing{0stuff-item0.eps}}, and an object $C_1$ of $\cC$, called the object of ``1-cells" and denoted by the picture \cb{\ing{0stuff-item1.eps}}.  There are 1-morphisms $s,t: C_1 \ra C_0$ of $\cC$ called the source and target, such that the 1-morphism $s \times t: C_1 \ra C_0 \times C_0$ is a fibration.  

\item[1-data] There is a 1-morphism $i: C_0 \ra C_1$, called the ``identity" and denoted \cb{\ing{1stuff-item1.eps}}, and a 1-morphism $m: C_1 \times_{C_0} C_1 \ra C_1$, called ``horizontal composition" and denoted \cb{\ing{1stuff-item2.eps}}.  These morphisms are compatible with the source and target maps.

\item[2-data] There is a 2-isomorphism of $\cC$ from $m((i \circ s) \times \id): C_1 \ra C_1$ to $\id: C_1 \ra C_1$, called the ``left identity transformation" and denoted 
\raisebox{\cboxht}{\xymatrix@C-12pt{\ing{2stuff-item15-1.eps}   \ar@{=>}[r] &   \hb{\ing{2stuff-item15-2.eps}}}}, and a 2-isomorphism from $m(\id \times (i \circ t)): C_1 \ra C_1$ to $\id: C_1 \ra C_1$, called the ``right identity transformation" and denoted 
\raisebox{\cboxht}{\xymatrix@C-12pt{\ing{2stuff-item16-1.eps}   \ar@{=>}[r] &   \hb{\ing{2stuff-item16-2.eps}}}}.  
There is finally a 2-isomorphism from $m(m \times \id): C_1 \times_{C_0} C_1 \times_{C_0} C_1 \ra C_1$ to $m(\id \times m): C_1 \times_{C_0} C_1 \times_{C_0} C_1 \ra C_1$, called the ``associator" and denoted 
\raisebox{\cboxht}{\xymatrix@C-12pt{\ing{2stuff-item12-1.eps}   \ar@{=>}[r] &   \ing{2stuff-item12-2.eps}}}.  These 2-isomorphisms are compatible with source and target maps in the sense that applying the source or target map to the left or right identity transformation or to the associator yields an identity 2-isomorphism.

\item[3-axioms] These data are required to satisfy the following axioms.  
\begin{enumerate}
\item[1.] The two 2-morphisms from $m(i \times i): C_0 \ra C_1$ to $i: C_0 \ra C_1$ induced by the left and right identity transformations are equal.  This axiom is denoted by the commutative diagram
\raisebox{\cboxht}{\xymatrix@C-12pt{\ing{3stuff-item23-1.eps} \ar@/^9pt/ @{-}[r] \ar@/_9pt/@{-}[r] & \ing{3stuff-item23-2.eps}}}.
\item[2.] The composition of the associator with the left identity transformation is equal to the horizontal composition of the left identity transformation with the identity.  These are two 2-morphisms from $m(m((i \circ s) \times \id) \times \id): C_1 \times_{C_0} C_1 \ra C_1$ to $m: C_1 \times_{C_0} C_1 \ra C_1$.  The axiom is denoted 
\def\alphanum{\ifcase\xypolynode\or \ing{3stuff-item24-2.eps} \or \ing{3stuff-item24-1.eps} \or 
\ing{3stuff-item24-3.eps} \fi}
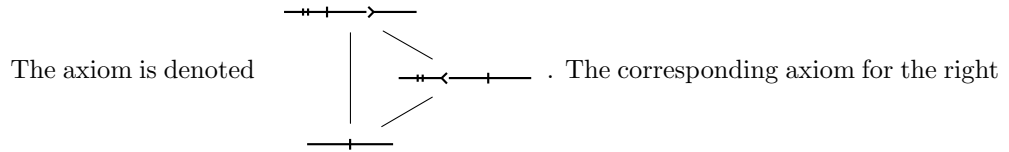
\xy/r\sct/: \xypolygon3{~={0}~*{\alphanum}}
\endxy.
The corresponding axiom for the right identity is also satisfied---this axiom is abbreviated \raisebox{\cboxht}{\xys{\ing{3stuff-item24a.eps}}}.
\item[3.] The two 2-morphisms from $m(m(\id \times (i \circ t)) \times \id): C_1 \times_{C_0} C_1 \ra C_1$ to $m: C_1 \times_{C_0} C_1 \ra C_1$, using on the one hand the associator and the left identity and on the other hand the right identity, are equal.  This axiom is denoted
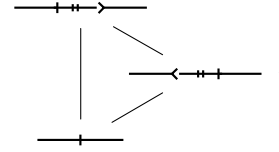
\def\alphanum{\ifcase\xypolynode\or \ing{3stuff-item25-2.eps} \or \ing{3stuff-item25-1.eps} \or 
\ing{3stuff-item25-3.eps} \fi}
\xy/r\sct/: \xypolygon3{~={0}~*{\alphanum}}
\endxy.
\item[4.] The pentagon axiom for the associator is satisfied:
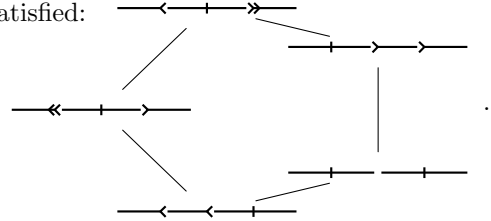
\def\alphanum{\ifcase\xypolynode\or \ing{3stuff-item17-1.eps} \or \ing{3stuff-item17-2.eps} \or 
\ing{3stuff-item17-3.eps} \or \ing{3stuff-item17-4.eps} \or \ing{3stuff-item17-5.eps} \fi}
\hspace{-50pt}
\db{
\xy/r\scd/: \xypolygon5{~:{(0,.7)::}~={36}~*{\alphanum}}
\endxy.
}
Here the double caret end on a morphism indicates that it is composed with the other morphisms after the singly careted morphism.
\end{enumerate}
\end{description}
\end{definition}

\begin{remark} \label{rem-nnf}
This definition can be modified by omitting the condition that $s \times t: C_1 \ra C_0 \times C_0$ be a fibration---we refer to the resulting notion as a not-necessarily fibrant category object in the 2-category $\cC$.  Not-necessarily fibrant category objects in the 2-category $\CAT$ are better known as ``double categories"; they were introduced by Ehresmann~\cite{ehresmann} and studied extensively by Grandis--Par\'e~\cite{grandispare-limits}.
\end{remark}

\begin{remark}
When the ambient 2-category is $\CAT$ or $\SMC$, it is convenient to modify this definition slightly in order to allow $C_0$ and $C_1$ to be large (symmetric monoidal) categories.  Though these are not objects of $\CAT$ or $\SMC$, per se, the definition given functions perfectly well in that context, with ``1-morphism" replaced by ``(symmetric monoidal) functor" and ``2-morphism" replaced by ``(symmetric monoidal) natural transformation".
\end{remark}

\begin{remark} \label{rem-defpullback}
The definition of a category object uses pullbacks in the ambient 2-category $\cC$.  These pullbacks are taken, by definition, in the underlying 1-category of $\cC$, that is the 1-category obtained by forgetting the 2-morphisms of $\cC$.  
Our condition that $s \times t: C_1 \ra C_0 \times C_0$ is a fibration ensures that this approach to pullbacks is sensible.

More generally, given a possibly-weak $n$-category that has an underlying 1-category (that is, that has strictly associative composition of 1-morphisms and strict 1-morphism identities), we define pullbacks in the $n$-category to be pullbacks in the underlying 1-category.
\end{remark}

%\begin{remark}
The axioms for a category object are reminiscent of those occurring in the definition of a bicategory.  Briefly, a bicategory $B$ has a collection $B_0$ of objects, denoted \cb{\ing{0stuff-item0.eps}}, a collection $B_1$ of 1-cells, denoted \cb{\ing{0stuff-item1.eps}}, and a collection $B_2$ of 2-cells, denoted \cb{\ings{0stuff-item2.eps}}, together with source and target maps $s,t: B_1 \ra B_0$ and $s,t: B_2 \ra B_1$ such that $st = ss$ and $tt=ts$.  There is a 1-cell identity map $i_x : B_0 \ra B_1$, denoted \cb{\ing{1stuff-item1.eps}}, and a horizontal composition of 1-cells map $m_x: B_1 \times_{B_0} B_1 \ra B_1$, denoted \cb{\ing{1stuff-item2.eps}}.  There is moreover a 2-cell identity $i_y: B_1 \ra B_2$, \cb{\ings{1stuff-item3.eps}}, and 2-cell composition $m_y: B_2 \times_{B_1} B_2 \ra B_2$, \cb{\ings{1stuff-item4.eps}}.  There are maps $w_r: B_2 \times_{B_0} B_1 \ra B_2$, \cb{\ings{1stuff-item5.eps}} and $w_l: B_1 \times_{B_0} B_2 \ra B_2$, \cb{\ings{1stuff-item6.eps}} that whisker a 2-cell by a 1-cell.  Finally there are three structure maps taking values in invertible 2-cells, namely the left identity $i_l: B_1 \ra B_2$, \cb{\ings{d1stuff-item7.eps}}, right identity $i_r: B_1 \ra B_2$, \cb{\ings{d1stuff-item8.eps}}, and associator $a_x: B_1 \times_{B_0} B_1 \times_{B_0} B_1 \ra B_2$, \cb{\ings{b1stuff-item9.eps}}.  These maps are required to satisfy various equations, including the exchange equation, the pentagon equation, and the left and right and middle triangle equations, denoted respectively: 
\hspace*{-7pt}
\xymatrix@C-17pt{\cb{\ings{2stuff-item8-1.eps}}   \ar@{=}[r] &   \cb{\ings{2stuff-item8-2.eps}}}
\xymatrix@C-17pt{\cb{\ings{b2stuff-item12-1.eps}}   \ar@{=}[r] &   \cb{\ings{b2stuff-item12-2t.eps}}}
\xymatrix@C-17pt{\cb{\ings{b2stuff-item16-1.eps}}   \ar@{=}[r] &   \cb{\ings{b2stuff-item16-2.eps}}}
\xymatrix@C-17pt{\cb{\ings{b2stuff-item17-1.eps}}   \ar@{=}[r] &   \cb{\ings{b2stuff-item17-2.eps}}}
\xymatrix@C-17pt{\cb{\ings{b2stuff-item18-1.eps}}   \ar@{=}[r] &   \cb{\ings{b2stuff-item18-2.eps}}}

\nid A 2-category is a bicategory for which the left identity $i_l$, right identity $i_r$, and associator map $a_x$ all factor through the 2-cell identity map $i_y$.

The reader may wonder why there is no horizontal composition of 2-cells, \cb{\ings{horz2cells.eps}}, in this definition of a bicategory.  Indeed, one could have taken as basic the horizontal composition of 2-cells, and then defined the operations \cb{\ings{1stuff-item5.eps}} and \cb{\ings{1stuff-item6.eps}} to be \cb{\ings{horz2cellsidright.eps}} and \cb{\ings{horz2cellsidleft.eps}} respectively.  Instead we take \cb{\ings{1stuff-item5.eps}} and \cb{\ings{1stuff-item6.eps}} as basic and let \cb{\ings{2stuff-item8-1.eps}} and \cb{\ings{2stuff-item8-2.eps}} play the role of horizontal composition of 2-cells.

A (non-necessarily fibrant) category object $C$ in $\CAT$ has an underlying bicategory, whose objects, 1-cells, and 2-cells are respectively the objects of $C_0$, the objects of $C_1$, and the morphisms of $C_1$ with identity source and target. 
%\end{remark}

The most unusual aspect of the above definition of a category object is the insistence that the 1-morphism $s \times t: C_1 \ra C_0 \times C_0$ be a fibration.  We illustrate the utility of this condition with the bordism category $\Bord_0^2 = (B_0, B_1^2)$ of $0$-, $1$-, and $2$-manifolds.  Here $B_0$ is the category of 0-manifolds with diffeomorphisms, and $B_1^2$ is the category of 1-manifolds with bordisms, as in Section~\ref{sec-localgeo}.  The horizontal composition functor $B_1^2 \times_{B_0} B_1^2 \ra B_1^2$ allows us to glue two 1-manifolds $M$ and $N$, when the target of $M$ is equal to the source of $N$.  However, it often happens that we would like to compose two manifolds $M$ and $N$ when the target of $M$ is merely isomorphic to the source of $N$.  Suppose we are given such an isomorphism $t(M) \xra{\simeq} s(N)$; the object $M$ maps to the domain of the isomorphism $s(M) \times t(M) \xra{\simeq} s(M) \times s(N)$, and using the fibration condition, we can lift this isomorphism to an isomorphism $M \xra{\simeq} M'$, where $M'$ now has source $s(M)$ but target $s(N)$---we can therefore compose the manifold $M'$ with the manifold $N$.  Altogether, the fibration condition allows us to construct a composition functor $B_1^2 \times^h_{B_0} B_1^2 \ra B_1^2$ with domain the weak or homotopy pullback, rather than the strict pullback.  The rationale for the fibration condition becomes even more vivid when we move to 2-category objects in categories: in that case, the fibration conditions are necessary to ensure that there is an associated tricategory.

What we have called a category object in $\CAT$ is very closely related to what Shulman calls a framed bicategory in~\cite{shulman-framed}, and what we have called a category object in $\SMC$ is very closely related to what Shulman calls a fibrant symmetric monoidal double category in~\cite{shulman-constructing}. Shulman provides an elegant and extensive account of fibrant monoidal double categories, and the reader should look at his papers for details about various notions of fibrations of categories and for numerous applications of the idea of category objects in $\CAT$ and in $\SMC$.  The reader familiar with Shulman's work might want to think of Section~\ref{sec-internalbicats} below as developing a theory of ``framed tricategories" or ``fibrant symmetric monoidal $(2\!\times\!1)$-categories".

\subsection{Categories in 3-categories} \label{sec-catin3cat} ($\vcenter{\xymatrix@R=3pt@C=3pt@M=0pt{
\scriptscriptstyle\bullet & \scriptscriptstyle\bullet & \scriptscriptstyle\bullet \\ \scriptscriptstyle\bullet & \scriptscriptstyle\bullet & \scriptscriptstyle\bullet
}}$)
\nopagebreak

\nopagebreak
In Section~\ref{sec-semiweak} we described four notions of ``3-level categories", namely 3-categories, category objects in $\tCAT$, bicategory objects in $\CAT$, and tricategories.  Bicategory objects in $\CAT$, and more generally in any 2-category, will occupy our attention in Section~\ref{sec-internalbicats}.  In a 3-level structure coming from a bicategory object in $\CAT$, the horizontal composition of 1-cells and the vertical composition of 2-cells are both weak operations.  In the 3-level structure arising from a category object in $\tCAT$, only the horizontal composition of 1-cells is weak.

Category objects in $\tCAT$ are examples of category objects in a 3-category.  It is convenient, for instance for the example of tensor categories, to define a notion not of category object in a strict 3-category but of category object in a weaker type of 3-category known as a `Gray 3-category' (also called simply a `Gray category').  Our definition of Gray 3-category below differs from the usual presentation~\cite{gray, gps}, as it is based on a whisker, rather than horizontal-composition, model of 2-categories; but note that the categories of whisker- and horizontal-composition-versions of Gray 3-categories are isomorphic.  (Also note that a (whisker-style) Gray 3-category with only one object is a rigid kind of Kapranov-Voevodsky monoidal 2-category~\cite{kapvoev}.)

\begin{definition}
A Gray 3-category $T$ consists of the following two collections of data, subject to axioms as follows:
\begin{description}

\item[0-data] There are four sets:
\begin{itemize}
\item[] \cb{\ing{0stuff-item0.eps}} --- the set $T_0$ of objects.
\item[] \cb{\ing{0stuff-item1.eps}} --- the set $T_1$ of 1-morphisms.
\item[] \cb{\ing{0stuff-item2.eps}} --- the set $T_2$ of 2-morphisms. \vspace{2pt}
\item[] \cb{\ing{0stuff-item3.eps}} --- the set $T_3$ of 3-morphisms.
\end{itemize}
Moreover there are source and target maps $s,t:T_1 \ra T_0$, $s,t:T_2 \ra T_1$, and $s,t:T_3 \ra T_2$, such that $st = ss$ and $tt = ts$.

\item[1-data] There are thirteen maps of sets in three collections:

\begin{description}

\item[1-morphism target] There are two maps with target $T_1$:
\begin{itemize}
\item[S1-1]: \ing{1stuff-item1.eps} --- $i_x: T_0 \ra T_1$ --- horizontal identity.
\item[S1-2]: \ing{1stuff-item2.eps} --- $m_x: T_1 \times_{T_0} T_1 \ra T_1$ --- horizontal composition.
\end{itemize}
\vspace{4pt}

\item[2-morphism target] There are four maps with target $T_2$:
\begin{itemize}
\item[S1-3]: \cb{\ing{1stuff-item3.eps}} --- $i_y: T_1 \ra T_2$ --- vertical identity. \vspace{2pt}
\item[S1-4]: \cb{\ing{1stuff-item4.eps}} --- $m_y: T_2 \times_{T_1} T_2 \ra T_2$ --- vertical composition. \vspace{2pt}
\item[S1-5]: \cb{\ing{1stuff-item5.eps}} --- $w_r: T_2 \times_{T_0} T_1 \ra T_2$ --- right whisker. \vspace{2pt}
\item[S1-6]: \cb{\ing{1stuff-item6.eps}} --- $w_l: T_1 \times_{T_0} T_2 \ra T_2$ --- left whisker.
\end{itemize}
\vspace{4pt}

\item[3-morphism target] There are seven maps with target $T_3$.  The first of these is a map $sw: T_2 \times_{T_0} T_2 \ra T_3$ indicated by depicting the source and target of the 3-morphism $sw(a)$ in terms of the element $a \in T_2 \times_{T_0} T_2$.  The remaining six maps we indicate directly by drawing a picture of the resulting 3-morphism.
\begin{itemize}
\item[S1-7]: \xymatrix@C-12pt{\cb{\ing{2stuff-item8-1.eps}}   \ar@{-}[r] &   \cb{\ing{2stuff-item8-2.eps}}} 
--- $sw: T_2 \times_{T_0} T_2 \ra T_3$ --- switch. \vspace{4pt}
\item[S1-8]: \cb{\ing{t1stuff-item28.eps}} --- $i_z: T_2 \ra T_3$ --- spatial identity. \vspace{4pt}
\item[S1-9]: \cb{\ing{t1stuff-item29.eps}} --- $m_z: T_3 \times_{T_2} T_3 \ra T_3$ --- spatial composition. \vspace{4pt}
\item[S1-10]: \cb{\ing{t1stuff-item30.eps}} --- $f_b: T_3 \times_{T_1} T_2 \ra T_3$ --- bottom fin. \vspace{4pt}
\item[S1-11]: \cb{\ing{t1stuff-item31.eps}} --- $f_t: T_2 \times_{T_1} T_3 \ra T_3$ --- top fin. \vspace{4pt}
\item[S1-12]: \cb{\ing{t1stuff-item32.eps}} --- $h_r: T_3 \times_{T_0} T_1 \ra T_3$ --- right 3-cell whisker. \vspace{4pt}
\item[S1-13]: \cb{\ing{t1stuff-item33.eps}} --- $h_l: T_1 \times_{T_0} T_3 \ra T_3$ --- left 3-cell whisker.
\end{itemize}
\vspace{4pt}

\item[Inverses]
A 3-morphism $c \in T_3$ is called invertible if there exists a 3-morphism $c^{-1} \in T_3$ such that $m_z(c \times c^{-1}) = i_z(s(c))$ and $m_z(c^{-1} \times c) = i_z(t(c))$.  The 1-datum [S1-7] is required to take values in invertible 3-morphisms.

\end{description}

\item[2-axioms]
The above data are subject to the following thirty-four axioms, together with variant axioms abbreviated in parentheses.  In the first fifteen axioms, the condition is that the indicated 1-morphisms or 2-morphisms are equal.  In the next four axioms, the conditions is that the 3-morphism obtained by composing all the edges of the diagram is the spatial identity.  In the last fifteen axioms, the indicated equation of 3-morphisms is satisfied.  There, composition of 3-morphisms denotes spatial composition; also, the variant axioms are indicated as axial reflections of the drawn axioms.
\begin{description}
\item[1-morphism axioms]\hspace*{\fill}\nopagebreak
\begin{itemize}
\item[S2-1]: \raisebox{\cboxht}{\xymatrix@C-12pt{\ing{2stuff-item15-1.eps}   \ar@{=}[r] &   \hb{\ing{2stuff-item15-2.eps}}}} 
\item[S2-2]: \raisebox{\cboxht}{\xymatrix@C-12pt{\ing{2stuff-item16-1.eps}   \ar@{=}[r] &   \hb{\ing{2stuff-item16-2.eps}}}} 
\item[S2-3]: \raisebox{\cboxht}{\xymatrix@C-12pt{\ing{2stuff-item12-1.eps}   \ar@{=}[r] &   \ing{2stuff-item12-2.eps}}} 
\end{itemize}

\item[2-morphism axioms]\hspace*{\fill}\nopagebreak
\begin{itemize}
\begin{multicols}{2}
\item[S2-4]: \hb{\xymatrix@C-12pt{\cb{\ing{2stuff-item1-1.eps}}   \ar@{=}[r] &   \cb{\ing{2stuff-item1-2.eps}}}} 
\item[S2-5]: \hb{\xymatrix@C-12pt{\cb{\ing{2stuff-item2-1.eps}}   \ar@{=}[r] &   \cb{\ing{2stuff-item1-2.eps}}}} 
\item[S2-6]: \hb{\xymatrix@C-12pt{\cb{\ing{2stuff-item3-1.eps}}   \ar@{=}[r] &   \cb{\ing{2stuff-item3-2.eps}}}} 
\item[S2-7]: \hb{\xymatrix@C-12pt{\cb{\ing{2stuff-item4-1.eps}}   \ar@{=}[r] &   \cb{\ing{2stuff-item4-2.eps}}}} 
\item[S2-8]: \hb{\xymatrix@C-12pt{\cb{\ing{2stuff-item5-1.eps}}   \ar@{=}[r] &   \cb{\ing{2stuff-item5-2.eps}}}} 
\item[S2-9]: \hb{\xymatrix@C-12pt{\cb{\ing{2stuff-item6-1.eps}}   \ar@{=}[r] &   \cb{\ing{2stuff-item6-2.eps}}}} 
\item[S2-10]: \hb{\xymatrix@C-12pt{\cb{\ing{2stuff-item7-1.eps}}   \ar@{=}[r] &   \cb{\ing{2stuff-item7-2.eps}}}} 
\item[S2-11]: \hb{\xymatrix@C-12pt{\cb{\ing{2stuff-item9-1.eps}}   \ar@{=}[r] &   \cb{\ing{2stuff-item9-2.eps}}}} 
\item[S2-12]: \hb{\xymatrix@C-12pt{\cb{\ing{2stuff-item10-1.eps}}   \ar@{=}[r] &   \cb{\ing{2stuff-item10-2.eps}}}} 
\item[S2-13]: \hb{\xymatrix@C-12pt{\cb{\ing{2stuff-item11-1.eps}}   \ar@{=}[r] &   \cb{\ing{2stuff-item11-2.eps}}}} 
\item[S2-14]: \hb{\xymatrix@C-12pt{\cb{\ing{2stuff-item13-1.eps}}   \ar@{=}[r] &   \cb{\ing{2stuff-item13-2.eps}}}} 
\item[S2-15]: \hb{\xymatrix@C-12pt{\cb{\ing{2stuff-item14-1.eps}}   \ar@{=}[r] &   \cb{\ing{2stuff-item14-2.eps}}}} 
\end{multicols}
\end{itemize}

\item[3-morphism axioms]\hspace*{\fill}\nopagebreak
\begin{itemize}

\begin{multicols}{2}

\item[S2-16]:  
\def\alphanum{\ifcase\xypolynode\or \ing{3stuff-item6-4.eps} \or \ing{3stuff-item6-1.eps} \or 
\ing{3stuff-item6-2.eps} \fi}
\def\alphalab{\ifcase\xypolynode\or {=} \or {} \or {=} \fi}
\db
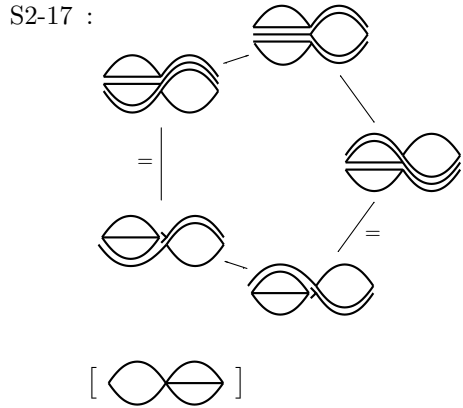
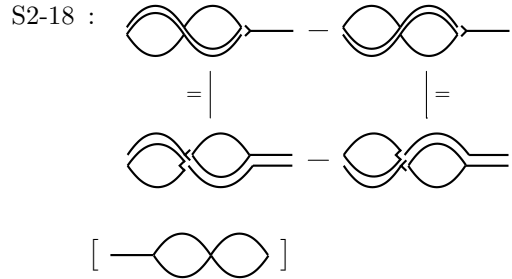
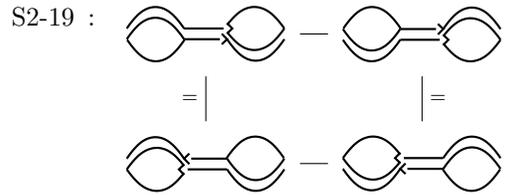
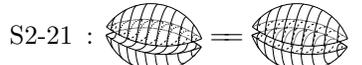
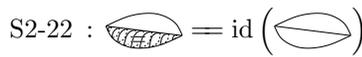
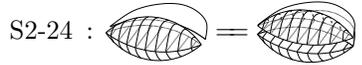
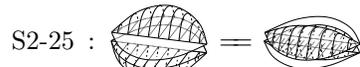
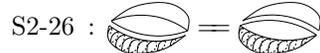
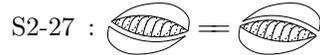
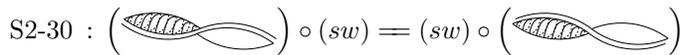
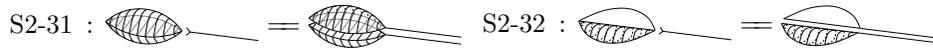
{
\xy/r\sca/: \xypolygon3{~={0}~*{\alphanum}~><{@{-}}~>>{_{\alphalab}}}
\endxy
}
\\
\hspace{20pt} $\big[\hspace{-1.5ex}\bb{\xys{\ing{3stuff-item6a.eps}}}\hspace{-1.5ex}\big]$

\item[S2-17]: 
\def\alphanum{\ifcase\xypolynode\or \ing{3stuff-item8-8t.eps} \or \ing{3stuff-item8-1t.eps} \or 
\ing{3stuff-item8-4t.eps} \or \ing{3stuff-item8-5.eps} \or \ing{3stuff-item8-6.eps} \fi}
\def\alphalab{\ifcase\xypolynode\or {} \or {} \or {=} \or {} \or {=} \fi}
\hspace{\hsqza}
\db{
\xy/r\scc/: \xypolygon5{~={0}~*{\alphanum}~><{@{-}}~>>{_{\alphalab}}}
\endxy
}
\\
\hspace{20pt} $\big[\hspace{-1.5ex}\bb{\xys{\ing{3stuff-item8a.eps}}}\hspace{-1.5ex}\big]$

\item[S2-18]: 
\def\alphanum{\ifcase\xypolynode\or \ing{3stuff-item13-1.eps} \or \ing{3stuff-item13-2.eps} \or 
\ing{3stuff-item13-4.eps} \or \ing{3stuff-item13-5.eps} \fi}
\def\alphalab{\ifcase\xypolynode\or {} \or {=} \or {} \or {=} \fi}
\hspace{-15pt}
\db{
\xy/r\scd/: \xypolygon4{~:{(0,.6)::}~*{\alphanum}~><{@{-}}~>>{_{\alphalab}}}
\endxy
}
\\
\hspace{20pt} $\big[\hspace{-1.5ex}\bb{\xys{\ing{3stuff-item13a.eps}}}\hspace{-1.5ex}\big]$ 

\item[S2-19]: 
\def\alphanum{\ifcase\xypolynode\or \ing{3stuff-item14-1.eps} \or \ing{3stuff-item14-2.eps} \or 
\ing{3stuff-item14-3.eps} \or \ing{3stuff-item14-4.eps} \fi}
\def\alphalab{\ifcase\xypolynode\or {} \or {=} \or {} \or {=} \fi}
\hspace{-15pt}
\db{
\xy/r\scd/: \xypolygon4{~:{(0,.6)::}~*{\alphanum}~><{@{-}}~>>{_{\alphalab}}}
\endxy
}

\end{multicols}

\begin{multicols}{2}

\item[S2-20]: %1
$\cb{\ingt{t2stuff-item27-1.eps}} \xlongequal{\phantom{x}} \cb{\ingt{t2stuff-item27-2.eps}}$ 
%\& [z-flip] 
\vspace{3pt}

\item[S2-21]: %2
$\cb{\ingt{t2stuff-item28-1.eps}} \xlongequal{\phantom{x}} \cb{\ingt{t2stuff-item28-2.eps}}$ 
\vspace{3pt}

\item[S2-22]: %3
$\cb{\scalebox{\hscsqz}{\ingt{t2stuff-item29-1.eps}}}  \xlongequal{\phantom{x}}   \text{id} \left( \cb{\scalebox{\hscsqz}{\ingt{t2stuff-item29-2.eps}}} \right)$ 
%\& [y-flip]
\vspace{3pt}

\item[S2-23]: %4
$\cb{\scalebox{\hscsqz}{\ingt{t2stuff-item30-1.eps}}}  \xlongequal{\phantom{x}}  \cb{\scalebox{\hscsqz}{\ingt{t2stuff-item30-2.eps}}}$ 
%\& [y-flip] 
\vspace{3pt}

\item[S2-24]: %5
$\cb{\ingt{t2stuff-item31-1.eps}} \xlongequal{\phantom{x}} \cb{\ingt{t2stuff-item31-2.eps}}$ 
%\& [y-flip]  
\vspace{3pt}

\item[S2-25]: %6
$\cb{\ingt{t2stuff-item32-1.eps}} \xlongequal{\phantom{x}} \cb{\ingt{t2stuff-item32-2.eps}}$
\vspace{3pt}

\item[S2-26]: %7
$\cb{\scalebox{\hscsqz}{\ingt{t2stuff-item33-1.eps}}}  
\xlongequal{\phantom{x}}  
\cb{\scalebox{\hscsqz}{\ingt{t2stuff-item33-2.eps}}}$ 
%\& [y-flip]
\vspace{3pt}

\item[S2-27]: %8
$\cb{\scalebox{\hscsqz}{\ingt{t2stuff-item34-1.eps}}}
\xlongequal{\phantom{x}}
\cb{\scalebox{\hscsqz}{\ingt{t2stuff-item34-2.eps}}}$
\vspace{3pt}

\item[S2-28]: %9
$\cb{\scalebox{\hscsqz}{\ingt{t2stuff-item35-1.eps}}}  \xlongequal{\phantom{x}}   \text{id} \left( \cb{\scalebox{\hscsqz}{\ingt{t2stuff-item35-2.eps}}} \right)$ 
%\& [x-flip]
\vspace{3pt}

\item[S2-29]: %10 
$\cb{\scalebox{\hscsqz}{\ingt{s2stuff-item29-1.eps}}}
\xlongequal{\phantom{x}}
\cb{\scalebox{\hscsqz}{\ingt{t2stuff-item27-2.eps}}}$ 
%\& [x-flip]
\vspace{3pt}

\end{multicols}

\item[S2-30]: %11
$\left(\cb{\scalebox{\hscsqz}{\ingt{t2stuff-item37-1.eps}}}\right) \circ (sw)   \xlongequal{\phantom{x}}  (sw) \circ \left( \cb{\scalebox{\hscsqz}{\ingt{t2stuff-item37-2.eps}}} \right)$ 
%\& [x-flip]
\vspace{3pt}

\begin{multicols}{2}

\item[S2-31]: %12
$\cb{\scalebox{\hscsqz}{\ingt{t2stuff-item38-1.eps}}} \xlongequal{\phantom{x}} \cb{\scalebox{\hscsqz}{\ingt{t2stuff-item38-2.eps}}}$ 
%\& [x-flip] 
\vspace{3pt}

\item[S2-32]: %13
$\cb{\scalebox{\hscsqz}{\ingt{t2stuff-item39-1.eps}}}
\xlongequal{\phantom{x}}
\cb{\scalebox{\hscsqz}{\ingt{t2stuff-item39-2.eps}}}$ 
%\& [x-flip, y-flip, xy-flip]
\vspace{3pt}

\end{multicols}

\item[S2-33]: %14
$\cb{\scalebox{\hscsqz}{\ingt{s2stuff-item33-1.eps}}} 
\xlongequal{\phantom{x}}
\cb{\scalebox{\hscsqz}{\ingt{s2stuff-item33-2.eps}}}$ 
%[x-flip]
\vspace{3pt}

\item[S2-34]: %15
$\cb{\scalebox{\hscsqz}{\ingt{s2stuff-item34-1.eps}}} 
\xlongequal{\phantom{x}}
\cb{\scalebox{\hscsqz}{\ingt{s2stuff-item34-2.eps}}}$ 
\vspace{3pt}

\item[Reflections]: z-flip of [S2-20]; y-flip of [S2-22], [S2-23], [S2-24], and [S2-26]; x-flip of [S2-28], [S2-29], [S2-30], [S2-31], and [S2-33]; and x-flip, y-flip, and xy-flip of [S2-32].

\end{itemize}

\end{description}
\end{description}
\end{definition}

\nid In axiom [S2-30], ``$sw$'' refers to the switch 3-morphism [S1-7].

\begin{definition}
A 3-category is a Gray 3-category such that the switch map [S1-7] lands in the image of the spatial identity.  In particular, the two composites \cb{\ings{2stuff-item8-1.eps}} and \cb{\ings{2stuff-item8-2.eps}} are equal.
\end{definition}

\begin{example} \label{eg-2catnwk}
There is a 3-category $\tCAT$ whose objects are 2-categories, whose 1-morphisms are strict functors of 2-categories, whose 2-morphisms are strict natural transformations of strict functors, and whose 3-morphisms are modifications of strict natural transformations.  There is a Gray 3-category $\tCATnwk$ whose objects, 1-morphisms, 2-morphisms, and 3-morphisms are respectively 2-categories, strict functors, weak natural transformations, and modifications.  This is the ambient Gray 3-category that will be relevant for tensor categories.  Notice that the collection $\tCATwk$ of 2-categories, weak functors, weak natural transformations, and modifications does not form a Gray 3-category, because the axiom S2-9 will not be satisfied.
\end{example}

We now have the necessary technology to define a category in a Gray 3-category, and in particular then a category in a 3-category.  As the notion ``category in a 2-category" categorifies ``category in a 1-category", so the notion ``category in a (Gray) 3-category" categorifies ``category in a 2-category", as follows.

\begin{definition} \label{def-catinswitch3cat}
A category object $C$ in the Gray 3-category $\cC$ consists of the following four collections of data, subject to the listed axioms.

\begin{description}
\item[0-data] 
There are two objects of $\cC$ as follows:
\begin{itemize}
\item[] $C_0$, denoted \cb{\ing{0stuff-item0.eps}} and called the object of 0-cells.
\item[] $C_1$, denoted \cb{\ing{0stuff-item1.eps}} and called the object of 1-cells.
\item[] Moreover, there are morphisms $s,t: C_1 \ra C_0$ called the source and target.\footnote{One might want to insist that $C_0$ be a 2-groupoid object and that the map $s \times t: C_1 \ra C_0 \times C_0$ be a fibration---see Remark~\ref{rem-nnfcaveat}.}
\end{itemize}
\vspace{4pt}

\item[1-data]
There are two morphisms of $\cC$ as follows:
\begin{itemize}
\item[C1-1]: $i: C_0 \ra C_1$, denoted \cb{\ing{1stuff-item1.eps}} and called the 1-cell identity.
\item[C1-2]: $m: C_1 \times_{C_0} C_1 \ra C_1$, denoted \cb{\ing{1stuff-item2.eps}} and called the horizontal composition.  Here and following, pullbacks are defined as in Remark~\ref{rem-defpullback}.
\end{itemize}
These morphisms are compatible with the source and target morphisms (as in Definition~\ref{def-catobj1cat}).
\vspace{4pt}

\item[2-data]
There are three 2-morphisms of $\cC$ as follows:
\begin{itemize}
\item[C2-1]: A 2-morphism from $m (i \times \id) (s \times \id): C_1 \ra C_1$ to $\id: C_1 \ra C_1$, denoted \cb{\ing{d1stuff-item7.eps}} and called the left identity transformation.
\item[C2-2]: A 2-morphism from $m (\id \times i) (\id \times t): C_1 \ra C_1$ to $\id: C_1 \ra C_1$, denoted \cb{\ing{d1stuff-item8.eps}} and called the right identity transformation.
\item[C2-3]: A 2-morphism from $m (m \times \id): C_1 \times_{C_0} C_1 \times_{C_0} C_1 \ra C_1$ to $m (\id \times m): C_1 \times_{C_0} C_1 \times_{C_0} C_1 \ra C_1$, denoted \cb{\ing{b1stuff-item9.eps}} and called the associator transformation.
\end{itemize}
The left and right identity transformations and the associator transformation are compatible with source and target maps (as in Definition~\ref{def-catobj2cat}).  These 2-morphisms are also required to be invertible, in that there exist three 2-morphisms pointing in the opposite direction, namely \cb{\ing{d1stuff-item9.eps}}, \cb{\ing{d1stuff-item10.eps}}, and \cb{\ing{b1stuff-item9a.eps}}, together with six invertible 3-morphisms from the appropriate composites to identity 2-morphisms, all satisfying the six ``triangle" axioms---note that despite the name, the ``triangle" axiom for an invertible 2-morphism in a Gray 3-category is a bigon.\footnote{See the 1-data portion of Definition~\ref{def-dicat} for a discussion of a bigon triangle axiom in a distinct but related context.}
%\pagebreak

\item[3-data]
There are five invertible 3-morphisms of $\cC$ as follows: %\vspace{-8pt}
\begin{multicols}{2}
\begin{itemize}
\item[C3-1]: 
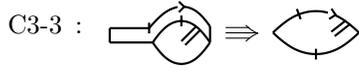
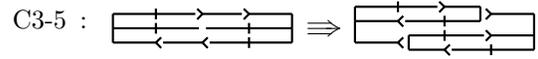
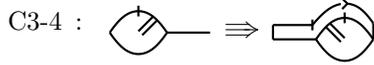
\xymatrix@C-12pt{\cb{\ing{d2stuff-item15-1.eps}}   \ar@3{->}[r] &   \cb{\ing{d2stuff-item15-2.eps}}} 
\item[C3-2]: \xymatrix@C-12pt{\cb{\ing{b2stuff-item16-1.eps}}   \ar@3{->}[r] &   \cb{\ing{b2stuff-item16-2.eps}}} 
\item[C3-3]: \xymatrix@C-12pt{\cb{\ing{b2stuff-item17-1.eps}}   \ar@3{->}[r] &   \cb{\ing{b2stuff-item17-2.eps}}} 
\item[C3-4]: \xymatrix@C-12pt{\cb{\ing{b2stuff-item18-1.eps}}   \ar@3{->}[r] &   \cb{\ing{b2stuff-item18-2.eps}}} 
\item[C3-5]: \xymatrix@C-12pt{\cb{\ing{b2stuff-item12-1.eps}}   \ar@3{->}[r] &   \cb{\ing{b2stuff-item12-2t.eps}}} 
\item[] \hspace*{-48pt} \parbox{3in}{These 3-morphisms are compatible with source and target maps in the sense that applying the source or target map to any of these five 3-morphisms yields an identity 3-morphism.}
\end{itemize}
\end{multicols} \vspace{-6pt}

\item[4-axioms]
The above data are such that the following five diagrams, together with the variant diagrams abbreviated in parentheses, commute: \vspace{-4pt}
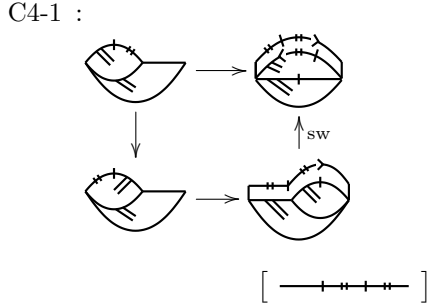
\begin{multicols}{2}
\begin{itemize}

\item[C4-1]:\\
\xymatrix@C-8pt@R-12pt@M-2pt{\cb{\ing{c4stuff-item1-1.eps}} \ar[d] \ar[r] & \cb{\ing{c4stuff-item1-4.eps}}  \\ 
\cb{\ing{c4stuff-item1-2.eps}} \ar[r] & \cb{\ing{c4stuff-item1-3.eps}} \ar[u]_{\text{sw}}
}
\\
\hspace*{\fill} $\big[\hspace{-1.5ex}\xys{\ing{c4stuff-item1a.eps}}\hspace{-1.5ex}\big]$ \hspace*{30pt} \vspace{-2pt}

\item[C4-2]:\\
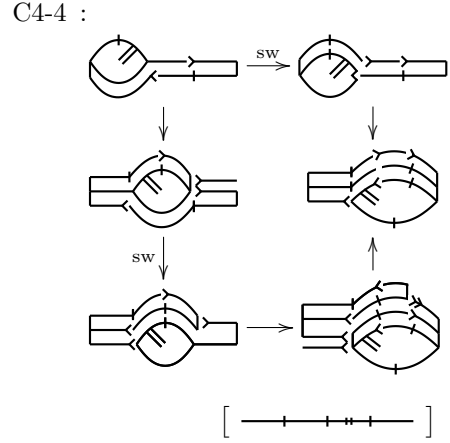
\xymatrix@C-8pt@R-12pt@M-2pt{\cb{\ing{c4stuff-item2-1.eps}} \ar[d] \ar[r]^{\text{sw}} & \cb{\ing{c4stuff-item2-4.eps}} \\ 
\cb{\ing{c4stuff-item2-2.eps}} \ar[r]^{\text{sw}^{-1}} & \cb{\ing{c4stuff-item2-3.eps}} \ar[u]
} \vspace{-2pt}

\item[C4-3]:\\
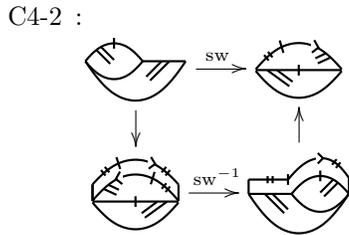
\xymatrix@C-8pt@R-12pt@M-2pt{
\cb{\ing{c4stuff-item3-1.eps}} \ar[r] \ar[d]_{\text{sw}} & \cb{\ing{c4stuff-item3-6.eps}} \ar[d] \\
\cb{\ing{c4stuff-item3-2.eps}} \ar[d] & \cb{\ing{c4stuff-item3-5.eps}}  \\
\cb{\ing{c4stuff-item3-3.eps}} \ar[r] & \cb{\ing{c4stuff-item3-4.eps}} \ar[u]_{\text{sw}}
}
\\
\hspace*{\fill} $\big[\hspace{-1.5ex}\xys{\ing{c4stuff-item3a.eps}}\hspace{-1.5ex}\big]$ \hspace*{30pt}

\item[C4-4]:\\
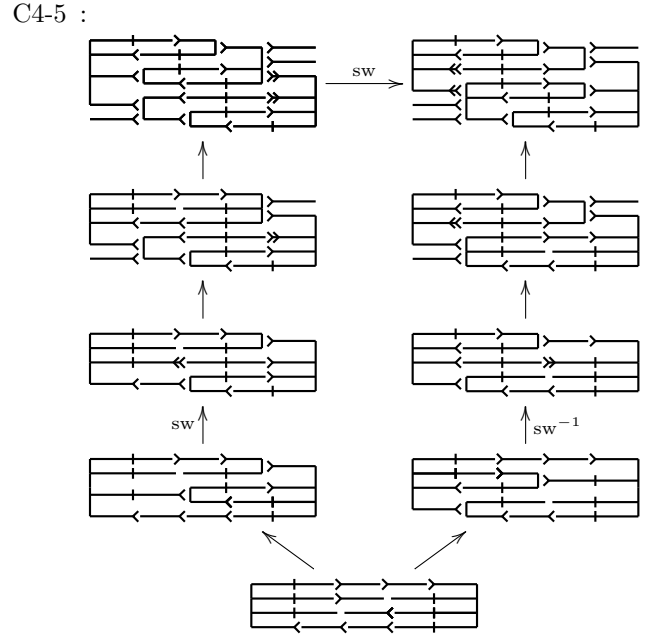
\xymatrix@C-8pt@R-12pt@M-2pt{
\cb{\ing{c4stuff-item4-1.eps}} \ar[r]^{\text{sw}} \ar[d] & \cb{\ing{c4stuff-item4-6.eps}} \ar[d] \\
\cb{\ing{c4stuff-item4-2.eps}} \ar[d]_{\text{sw}} & \cb{\ing{c4stuff-item4-5.eps}} \\
\cb{\ing{c4stuff-item4-3.eps}} \ar[r] & \cb{\ing{c4stuff-item4-4.eps}} \ar[u]
}
\\
\hspace*{\fill} $\big[\hspace{-1.5ex}\xys{\ing{c4stuff-item4a.eps}}\hspace{-1.5ex}\big]$ \hspace*{30pt} \vspace{-4pt}

\item[C4-5]:\\
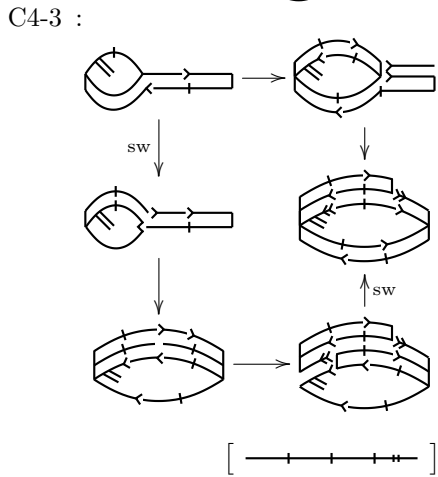
\xymatrix@C-56pt@R-10pt@M-2pt{\cb{\ing{c4stuff-item5-1.eps}} \ar[rr]^{\text{sw}} && \cb{\ing{c4stuff-item5-9.eps}} \\
\cb{\ing{c4stuff-item5-2.eps}} \ar[u]  && \cb{\ing{c4stuff-item5-8.eps}} \ar[u] \\
\cb{\ing{c4stuff-item5-3.eps}} \ar[u]  && \cb{\ing{c4stuff-item5-7.eps}} \ar[u]  \\
\cb{\ing{c4stuff-item5-4.eps}} \ar[u]^{\text{sw}} && \cb{\ing{c4stuff-item5-6.eps}} \ar[u]_{\text{sw}^{-1}}\\
& \cb{\ing{c4stuff-item5-5.eps}} \ar[ur] \ar[ul] &
}

\end{itemize}

\end{multicols}

\nid In these diagrams, unmarked arrows are 3-morphisms given by 3-data, while the arrows marked ``sw" are determined by the switch morphism [S1-7] in the ambient Gray 3-category.

\end{description}

\end{definition}

\begin{remark}
We explain in detail why the morphisms marked as switches in the axioms of the above definition are indeed switches in the ambient Gray 3-category, and along the way explicate the pictographic notation.  We focus on the switch morphism in axiom C4-1; all the other switches are analogous.  The claim is that, given the 0-, 1-, 2-, and 3-data of a category object $C$ in the Gray 3-category $\cC$, the switch morphism in $\cC$ induces a 3-morphism from \cb{\ings{d3stuff-item23-3.eps}} to \cb{\ings{d3stuff-item23-4.eps}}, and therefore from \cb{\ings{c4stuff-item1-3.eps}} to \cb{\ings{c4stuff-item1-4.eps}}.  

Let us first recall the exact meaning of these pictures.  The picture \cb{\ings{d1stuff-item7.eps}} is a 2-morphism of $\cC$ from the 1-morphism $C_1 \xra{s \times \id} C_0 \times C_1 \xra{i \times \id} C_1 \times C_1 \xra{m} C_1$ to the identity morphism $C_1 \xra{\id} C_1$.  The picture \cb{\ings{d2stuff-item16-2.eps}} applies that 2-morphism not to $C_1$ itself, but to the image of the multiplication \cb{\ing{1stuff-item2.eps}}, $m: C_1 \times_{C_0} C_1 \xra{} C_1$.  By \emph{definition} this picture \cb{\ings{d2stuff-item16-2.eps}} is therefore the left whisker \emph{in the ambient Gray 3-category} of the 2-morphism \cb{\ings{d1stuff-item7.eps}} by the 1-morphism $m$---this left whisker operation can be thought of as a form of ``precomposition".  The composite 2-morphism \cb{\ings{d3stuff-item23-4.eps}} can therefore be written in pasting-diagram notation as \vspace{-8pt}
\begin{equation} \nn
\xymatrix@C-8pt@R-16pt{
&&&& C_0 \times C_1 \ar[r]^-{i \times \id} & C_1 \times C_1 \ar[dr]^-{m} & \\
C_1 \ar[r]^-{s \times \id} & C_0 \times C_1 \ar[r]^-{i \times \id} & C_1 \times C_1 \ar[r]^-{m} & C_1 \ar[ur]^-{s \times \id} \ar[rrr]^{\id} 
\ar@/^15pt/ @{{ }{ }}[rrr] |{\Downarrow}
&&& C_1 \\
C_1 \ar[r]^-{s \times \id} \ar@/_26pt/[rrr]^{\id} 
\ar@/_11pt/ @{{ }{ }}[rrr] |{\Downarrow}
& C_0 \times C_1 \ar[r]^-{i \times \id} & C_1 \times C_1 \ar[r]^-{m} & C_1 \ar[rrr]^{\id} &&& C_1
}
\end{equation}
By similar unpacking, the composite 2-morphism \cb{\ings{d3stuff-item23-3.eps}} has the pasting-diagram notation
\begin{equation} \nn
\xymatrix@C-8pt@R-16pt{
& C_0 \times C_1 \ar[r]^-{i \times \id} & C_1 \times C_1 \ar[dr]^-{m} &&&& \\
C_1 \ar[ur]^-{s \times \id} \ar[rrr]^{\id} 
\ar@/^15pt/ @{{ }{ }}[rrr] |{\Downarrow}
&&& C_1 \ar[r]^-{s \times \id} & C_0 \times C_1 \ar[r]^-{i \times \id} & C_1 \times C_1 \ar[r]^-{m} & C_1 \\
C_1\ar[rrr]^{\id} &&& C_1 \ar[r]^-{s \times \id} \ar@/_26pt/[rrr]^{\id} 
\ar@/_11pt/ @{{ }{ }}[rrr] |{\Downarrow}
& C_0 \times C_1 \ar[r]^-{i \times \id} & C_1 \times C_1 \ar[r]^-{m} & C_1
}
\end{equation}
Indeed these two 2-morphisms are related by a switch in the ambient Gray 3-category.
\end{remark}

\begin{remark} \label{rem-nnfcaveat}
The above definition describes a notion of not-necessarily fibrant category object $C = (C_0, C_1)$ in a Gray 3-category $\cC$.  There certainly exist notions of fibrations of 2-categories in the literature (see for instance~\cite[p.35]{gray}), and these could be used to provide a definition of when a 1-morphism is a fibration in a Gray 3-category.  We do not pursue this aspect of the theory here.
\end{remark}

\begin{example} \label{eg-tc}
We expect that tensor categories form a category object $\TC$ in the Gray 3-category $\tCATnwk$.  (Recall from Example~\ref{eg-2catnwk} that $\tCATnwk$ is the Gray 3-category of 2-categories, strict functors, weak natural transformation, and modifications.)  Note that Greenough's work on the monoidal 2-category of bimodules over a fixed tensor category~\cite{greenough} and Schaumann's work on fusion categories and their bimodules~\cite{schaumann} do a substantive part but not all of the work of establishing that $\TC$ forms a category object in $\tCATnwk$.  (Note also that Johnson-Freyd--Scheimbauer construct an ``($\infty,1)$-by-$(\infty,2)$"-category of algebra objects in linear categories~\cite{jfs}, which provides an infinity-category-theoretic version of the category object (of 2-categories) $\TC$.)

The 2-category $\TC_0$ of 0-cells of the category object $\TC$ is the strict 2-category of tensor categories, monoidal functors, and monoidal natural transformations.  The 2-category $\TC_1$ of 1-cells of the category object $\TC$ is the strict 2-category of bimodules between tensor categories, along with functors of bimodule categories, and natural transformations of these functors.  The structure data of the category object $\TC$ consists of the categorifications of the structure data of $\Alg$, the category object in $\CAT$ whose 0-cells are algebras and whose 1-cells are bimodules---for instance the 1-cell identity of $\TC$ is the identity bimodule category, and the horizontal composition is the balanced tensor product of bimodule categories over a tensor category.

There is a model for the balanced tensor product $\mathcal{M} \boxtimes_{\mathcal{A}} \mathcal{N}$ of bimodule categories in which objects are represented by retracts of direct sums of elementary tensors $m \otimes n$ of objects $m \in \mathcal{M}$ and $n \in \mathcal{N}$.  In this model, observe as follows that the left and right identity transformations [C2-1] and [C2-2] are weak, not strict, natural transformations.  For objects $a \in \mathcal{A}$ and $n \in \mathcal{N}$, the $\mathcal{A}$-module action on $\mathcal{N}$ determines an object $a \cdot n \in \mathcal{N}$---this action defines the transformation from the functor $\mathcal{N} \mapsto \mathcal{A} \boxtimes_\mathcal{A} \mathcal{N}$ to the identity functor; for a bimodule functor $\phi: {}_\mathcal{A} \mathcal{N}_\mathcal{B} \ra {}_\mathcal{A} \mathcal{P}_\mathcal{B}$, the objects $a \cdot \phi(n)$ and $\phi(a \cdot n)$ are not equal but merely isomorphic, as accommodated by a weak natural transformation.  It is this consideration that encourages working in the Gray 3-category $\tCATnwk$ rather than the 3-category $\tCAT$.

Notice, however, that in the aforementioned model the category object $\TC$ is somewhat stricter than a generic category object in $\tCATnwk$, in the sense that the associator 2-datum [C2-3] happens to be a strict, rather than weak, natural transformation.  We will refer to this kind of category object as an ``associatively strict category object in $\tCATnwk$".
\end{example}

%It is worth asking how one knows that the list of data and axioms for a category object in a Gray 3-category, given in Definition~\ref{def-catinswitch3cat}, is ``sufficient".  Sufficiency can be formulated as a precise condition, namely the 2-connectivity of the space of ways to evaluate a series of 1-cells of the category object; this space of evaluations is defined in terms of the specified data and axioms for a category object.  This issue of sufficiency and the corresponding connectivity result are discussed in \textcolor{red}{Appendix~A}.

\vspace{-10pt}
\section{Internal bicategories} \label{sec-internalbicats}

Bicategory objects in the 2-category of categories provide a natural framework for studying both the bordism ``3-level category", of $0$-, $1$-, $2$-, and $3$-manifolds, and the ``3-level category" of conformal nets.  Bicategory objects in 2-categories also provide a convenient way to encode the symmetric monoidal structure of a bicategory in categories, by changing the ambient 2-category to symmetric monoidal categories.

Recall that a bicategory has objects, 1-cells, 2-cells, identity 1-cells, composition of 1-cells, identity 2-cells, composition of 2-cells, left and right whiskers of 2-cells by 1-cells, and associator and left and right identity isomorphisms.  A bicategory object $C$ in a 2-category $\cC$ will have, roughly speaking, the following structure.  There will be an object $C_0 \in \cC$ of 0-cells, an object $C_1 \in \cC$ of 1-cells, and an object $C_2 \in \cC$ of 2-cells.  There are morphisms $s,t: C_1 \ra C_0$ of $\cC$, called the source and target of 1-cells, and morphisms $s,t: C_2 \ra C_1$, called the source and target of 2-cells.  There is a morphism $i: C_0 \ra C_1$, called the 1-cell identity, and a morphism $m: C_1 \times_{C_0} C_1 \ra C_1$, called horizontal composition.  Similarly, there is a 2-cell identity morphism $C_1 \ra C_2$, 2-cell composition morphism $C_2 \times_{C_1} C_2 \ra C_2$, and whisker morphisms $C_2 \times_{C_0} C_1 \ra C_2$ and $C_1 \times_{C_0} C_2 \ra C_2$.  The next pieces of structure generalize the notion of the identity transformations and the associator transformation of a bicategory.  

Consider the left identity: it is meant to compare the morphism $m((i \circ s) \times \id): C_1 \ra C_1$ of $\cC$ to the morphism $\id: C_1 \ra C_1$ of $\cC$.  We can either ask for a 2-morphism of $\cC$ between these two 1-morphisms, or we can ask for a morphism $C_1 \ra C_2$ whose composites with the 2-cell source and target $s,t:C_2 \ra C_1$ recover the two morphisms in question.  For example, if the ambient 2-category $\cC$ is the 2-category of categories, then we are asking either for a natural isomorphism between the functors $m((i \circ s) \times \id): C_1 \ra C_1$ and $\id: C_1 \ra C_1$, or for a functor $C_1 \ra C_2$ from 1-cells to 2-cells such that the composites $C_1 \ra C_2 \xra{s,t} C_1$ are the functors $m((i \circ s) \times \id)$ and $\id$ respectively.  Demanding a natural isomorphism leads to a simpler, more rigid structure, while allowing 2-cell transformations produces a more complex, weaker and thus more inclusive structure.

We have a similar choice for the associator transformations: they can either be 2-morphisms of the ambient 2-category (a strict notion) or 2-cell transformations (a weak notion).  Altogether then, there are 4 distinct notions of internal bicategories in 2-categories: strict identities and strict associators, weak identities and strict associators, strict identities and weak associators, and weak identities and weak associators.  We do not have need of the third notion (strict identities and weak associators) and give it no consideration.  The other three notions are called respectively, ``2-category object in a 2-category", ``dicategory object in a 2-category", and ``bicategory object in a 2-category".  These are categorifications of the notions 2-category, dicategory, and bicategory.

In Sections~\ref{sec-2catin2cat},~\ref{sec-dicatin2cat}, and~\ref{sec-bicatin2cat} we give precise definitions of the three variants of internal bicategories.  In Section~\ref{sec-2dibi} we describe how to construct a dicategory object from a 2-category object, and a bicategory object from a dicategory object.  In Section~\ref{sec-catin2catarebicatincat} we observe that categories in (the 3-category of) 2-categories are bicategories in (the 2-category of) categories.  In Section~\ref{app-bicatincataretricat}, we will see that bicategory objects in $\CAT$ have underlying tricategories.

\subsection{2-categories in 2-categories} \label{sec-2catin2cat} $\left(\vcenter{\xymatrix@R=3pt@C=3pt@M=0pt{
\scriptscriptstyle\bullet & \scriptscriptstyle\bullet \\ \scriptscriptstyle\bullet & \scriptscriptstyle\bullet \\ \scriptscriptstyle\bullet & \scriptscriptstyle\bullet
}}\right)$

A 2-category object $C$ in a 2-category $\cC$ has 0-, 1-, and 2-cell objects $C_0, C_1, C_2 \in \cC$, and has identity 1-cells and composition of 1-cells, with comparatively strict identity transformations and associator transformations.  In Section~\ref{sec-catin2cat} we introduced a convenient diagrammatic notion for category objects in 2-categories.  We rely heavily on this notation in the definition of 2-category object.

\begin{definition}
A 2-category object $C$ in the 2-category $\cC$ consists of the following three collections of data, subject to the listed axioms.  
\begin{description}
\item[0-data]
There are three objects of $\cC$ as follows:
\begin{itemize}
\item[] $C_0$, a groupoid object, denoted \cb{\ing{0stuff-item0.eps}} and called the object of 0-cells.
\item[] $C_1$, a groupoid object, denoted \cb{\ing{0stuff-item1.eps}} and called the object of 1-cells.
\item[] $C_2$, denoted \cb{\ing{0stuff-item2.eps}} and called the object of 2-cells.
\end{itemize}
In addition, there are morphisms $s,t: C_1 \ra C_0$ and $s,t: C_2 \ra C_1$, the source and target, such that $st = ss$ and $tt = ts$, and such that $s \times t: C_1 \ra C_0 \times C_0$ and $s \times t: C_2 \ra C_1 \times_{C_0 \times C_0} C_1$ are fibrations.
\item[1-data]
There are six morphisms of $\cC$ as follows:
\begin{itemize}
\item[F1-1]: $i: C_0 \ra C_1$, denoted \cb{\ing{1stuff-item1.eps}} and called the 1-cell identity.
\item[F1-2]: $m: C_1 \times_{C_0} C_1 \ra C_1$, denoted \cb{\ing{1stuff-item2.eps}} and called the horizontal composition.
\item[F1-3]: $i_v: C_1 \ra C_2$, denoted \cb{\ing{1stuff-item3.eps}} and called the 2-cell identity.
\item[F1-4]: $m_v: C_2 \times_{C_1} C_2 \ra C_2$, denoted \cb{\ing{1stuff-item4.eps}} and called the vertical composition.
\item[F1-5]: $w_r: C_2 \times_{C_0} C_1 \ra C_2$, denoted \cb{\ing{1stuff-item5.eps}} and called the right composition or whisker of a 1-cell with a 2-cell.
\item[F1-6]: $w_l: C_1 \times_{C_0} C_2 \ra C_2$, denoted \cb{\ing{1stuff-item6.eps}} and called the left composition or whisker of a 1-cell with a 2-cell.
\end{itemize} \nopagebreak[4]
These morphisms are compatible with source and target maps.

\item[2-data]
There are sixteen 2-isomorphisms of $\cC$ as follows: \vspace{-2pt}
\begin{multicols}{2}
\begin{itemize}
\item[F2-1]: \xymatrix@C-12pt{\cb{\ing{2stuff-item1-1.eps}}   \ar@{=>}[r] &   \cb{\ing{2stuff-item1-2.eps}}} 
\item[F2-2]: \xymatrix@C-12pt{\cb{\ing{2stuff-item2-1.eps}}   \ar@{=>}[r] &   \cb{\ing{2stuff-item2-2.eps}}} 
\item[F2-3]: \xymatrix@C-12pt{\cb{\ing{2stuff-item3-1.eps}}   \ar@{=>}[r] &   \cb{\ing{2stuff-item3-2.eps}}} 
\item[F2-4]: \xymatrix@C-12pt{\cb{\ing{2stuff-item4-1.eps}}   \ar@{=>}[r] &   \cb{\ing{2stuff-item4-2.eps}}} 
\item[F2-5]: \xymatrix@C-12pt{\cb{\ing{2stuff-item5-1.eps}}   \ar@{=>}[r] &   \cb{\ing{2stuff-item5-2.eps}}} 
\item[F2-6]: \xymatrix@C-12pt{\cb{\ing{2stuff-item6-1.eps}}   \ar@{=>}[r] &   \cb{\ing{2stuff-item6-2.eps}}}  
\item[F2-7]: \xymatrix@C-12pt{\cb{\ing{2stuff-item7-1.eps}}   \ar@{=>}[r] &   \cb{\ing{2stuff-item7-2.eps}}} 
\item[F2-8]: \xymatrix@C-12pt{\cb{\ing{2stuff-item8-1.eps}}   \ar@{=>}[r] &   \cb{\ing{2stuff-item8-2.eps}}} 
\item[F2-9]: \xymatrix@C-12pt{\cb{\ing{2stuff-item9-1.eps}}   \ar@{=>}[r] &   \cb{\ing{2stuff-item9-2.eps}}}  
\item[F2-10]: \xymatrix@C-12pt{\cb{\ing{2stuff-item10-1.eps}}   \ar@{=>}[r] &   \cb{\ing{2stuff-item10-2.eps}}} 
\item[F2-11]: \xymatrix@C-12pt{\cb{\ing{2stuff-item11-1.eps}}   \ar@{=>}[r] &   \cb{\ing{2stuff-item11-2.eps}}}  
\item[F2-12]: \xymatrix@C-12pt{\cb{\ing{2stuff-item12-1.eps}}   \ar@{=>}[r] &   \cb{\ing{2stuff-item12-2.eps}}} 
\item[F2-13]: \xymatrix@C-12pt{\cb{\ing{2stuff-item13-1.eps}}   \ar@{=>}[r] &   \cb{\ing{2stuff-item13-2.eps}}}  
\item[F2-14]: \xymatrix@C-12pt{\cb{\ing{2stuff-item14-1.eps}}   \ar@{=>}[r] &   \cb{\ing{2stuff-item14-2.eps}}} 
\item[F2-15]: \xymatrix@C-12pt{\cb{\ing{2stuff-item15-1.eps}}   \ar@{=>}[r] &   \cb{\ing{2stuff-item15-2.eps}}}  
\item[F2-16]: \xymatrix@C-12pt{\cb{\ing{2stuff-item16-1.eps}}   \ar@{=>}[r] &   \cb{\ing{2stuff-item16-2.eps}}}
\end{itemize}
\end{multicols} \vspace{-2pt}
\nid These 2-isomorphisms are compatible with source and target maps in the sense that the sources and targets of [F2-1] through [F2-8] are identity 2-isomorphisms, the sources and targets of [F2-9] through [F2-11] are the 2-isomorphism [F2-12], the source and target of [F2-13] are [F2-15], and the source and target of [F2-14] are [F2-16].

\item[3-axioms]
The above data are such that the following twenty-five diagrams, as well as the variant diagrams abbreviated in parentheses, commute:

\end{description} % In order for axioms to have full page width.

\begin{itemize}
\begin{multicols}{2}
\item[F3-1]:
\db{
\xymatrix@C-12pt{\cb{\ing{3stuff-item1-1.eps}} \ar@/^10pt/ @{-}[r] \ar@/_10pt/@{-}[r] & \cb{\ing{3stuff-item1-2.eps}}}
}
\vspace{10pt}

\item[F3-2]: 
\def\alphanum{\ifcase\xypolynode\or \ing{3stuff-item2-1.eps} \or \ing{3stuff-item2-2.eps} \or 
\ing{3stuff-item2-3.eps} \fi}
\db{
\xy/r\sca/: \xypolygon3{~={0}~*{\alphanum}}
\endxy
}\nopagebreak
\\[-30pt]
\hspace*{\fill} $\big[\hspace{-1.5ex}\bb{\xys{\ing{3stuff-item2a.eps}}}\hspace{-1.5ex}\big]$ \hspace*{30pt}

\item[F3-3]:  
\def\alphanum{\ifcase\xypolynode\or \ing{3stuff-item3-1.eps} \or \ing{3stuff-item3-2.eps} \or 
\ing{3stuff-item3-3.eps} \fi}
\db{
\xy/r\sca/: \xypolygon3{~={0}~*{\alphanum}}
\endxy
}
\vspace{10pt}

\item[F3-4]: 
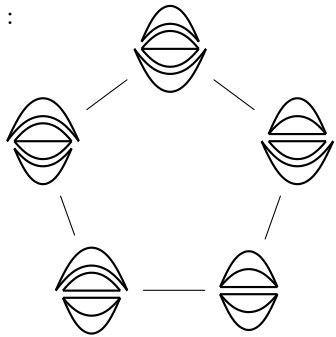
\def\alphanum{\ifcase\xypolynode\or \ing{3stuff-item4-1.eps} \or \ing{3stuff-item4-2.eps} \or 
\ing{3stuff-item4-3.eps} \or \ing{3stuff-item4-4.eps} \or \ing{3stuff-item4-5.eps} \fi}
\hspace{\hsqza}
\db{
\xy/r\scc/: \xypolygon5{~*{\alphanum}}
\endxy
}
\vspace{10pt}

\item[F3-5]: 
\def\alphanum{\ifcase\xypolynode\or \ing{3stuff-item5-1.eps} \or \ing{3stuff-item5-2.eps} \or 
\ing{3stuff-item5-3.eps} \or \ing{3stuff-item5-4.eps} \fi}
\db
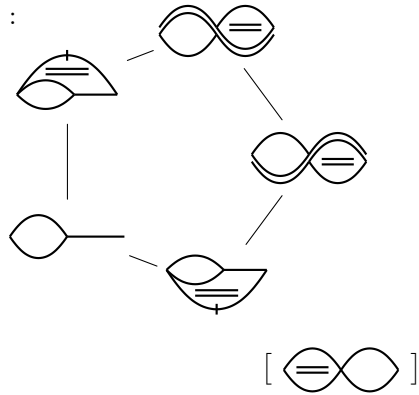
{
\xy/r\scb/: \xypolygon4{~*{\alphanum}}
\endxy
}
\\
\hspace{10pt} $\big[\hspace{-1.5ex}\bb{\xys{\ing{3stuff-item5a.eps}}} 
\hspace{-5pt} \bb{\xys{\ing{3stuff-item5b.eps}}} 
\hspace{-5pt} \bb{\xys{\ing{3stuff-item5c.eps}}}\hspace{-1.5ex}\big]$
\vspace{5pt}

\item[F3-6]: 
\def\alphanum{\ifcase\xypolynode\or \ing{3stuff-item6-1.eps} \or \ing{3stuff-item6-2.eps} \or 
\ing{3stuff-item6-3.eps} \or \ing{3stuff-item6-4.eps} \or \ing{3stuff-item6-5.eps} \fi}
\hspace{\hsqza}
\db
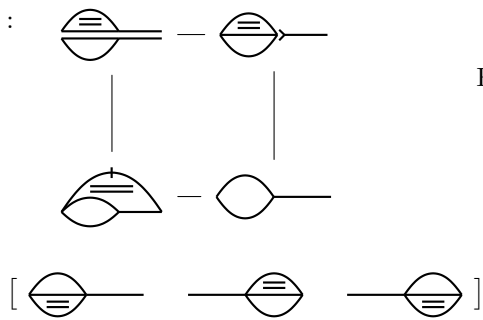
{
\xy/r\scc/: \xypolygon5{~={0}~*{\alphanum}}
\endxy
}
\\[-5pt]
\hspace*{\fill} $\big[\hspace{-1.5ex}\bb{\xys{\ing{3stuff-item6a.eps}}}\hspace{-1.5ex}\big]$ \hspace*{30pt}
\vspace{5pt}

\item[F3-7]: 
\def\alphanum{\ifcase\xypolynode\or \ing{3stuff-item7-1.eps} \or \ing{3stuff-item7-2.eps} \or 
\ing{3stuff-item7-3.eps} \or \ing{3stuff-item7-4.eps} \or \ing{3stuff-item7-5.eps} \or \ing{3stuff-item7-6.eps} \fi}
\hspace{\hsqza}
\db
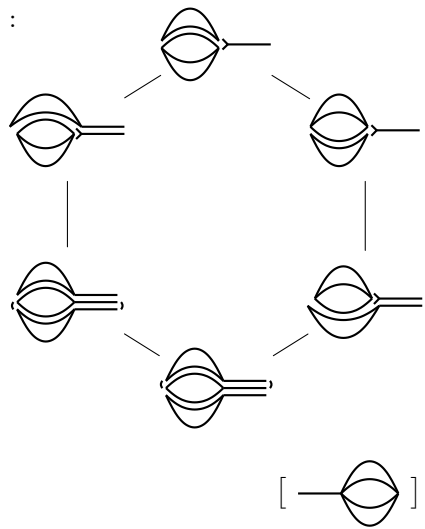
{
\xy/r\sce/: \xypolygon6{~={30}~*{\alphanum}}
\endxy
}
\\[-5pt]
\hspace*{\fill} $\big[\hspace{-1.5ex}\bb{\xys{\ing{3stuff-item7a.eps}}}\hspace{-1.5ex}\big]$ \hspace*{30pt} 
\vspace{5pt}

\end{multicols}

\item[F3-8]: 
\def\alphanum{\ifcase\xypolynode\or \ing{3stuff-item8-1.eps} \or \ing{3stuff-item8-2.eps} \or \ing{3stuff-item8-3.eps} \or \ing{3stuff-item8-4.eps} \or \ing{3stuff-item8-5.eps} \or \ing{3stuff-item8-6.eps} \or \ing{3stuff-item8-7.eps} \or \ing{3stuff-item8-8.eps} \fi}
\hspace{\hsqza}
\db
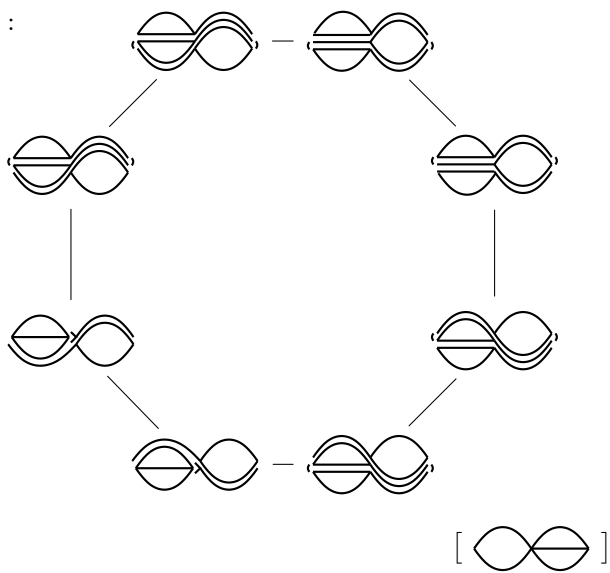
{
\xy/r\scg/:  \xypolygon8{~*{\alphanum}}
\endxy
}
\\[-5pt]
\hspace*{\fill} $\big[\hspace{-1.5ex}\bb{\xys{\ing{3stuff-item8a.eps}}}\hspace{-1.5ex}\big]$ \hspace*{160pt}
\vspace{5pt}

\begin{multicols}{2}

\item[F3-9]: 
\hspace*{-25pt}
\def\alphanum{\ifcase\xypolynode\or \ing{3stuff-item9-1.eps} \or \ing{3stuff-item9-2.eps} \or 
\ing{3stuff-item9-3.eps} \or \ing{3stuff-item9-4.eps} \or \ing{3stuff-item9-5.eps} \fi}
\hspace{\hsqza}
\db
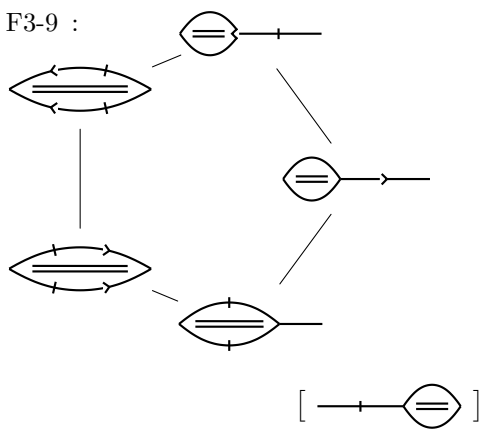
{
\xy/r\scd/: \xypolygon5{~={0}~*{\alphanum}}
\endxy
}
\\[-5pt]
\hspace*{\fill} $\big[\hspace{-1.5ex}\bb{\xys{\ing{3stuff-item9a.eps}}}\hspace{-1.5ex}\big]$ \hspace*{30pt} 
\vspace{5pt}

\item[F3-10]: 
\hspace*{-25pt}
\def\alphanum{\ifcase\xypolynode\or \ing{3stuff-item10-1.eps} \or \ing{3stuff-item10-2.eps} \or 
\ing{3stuff-item10-3.eps} \or \ing{3stuff-item10-4.eps} \or \ing{3stuff-item10-5.eps} \or \ing{3stuff-item10-6.eps} \fi}
\hspace{\hsqza}
\db
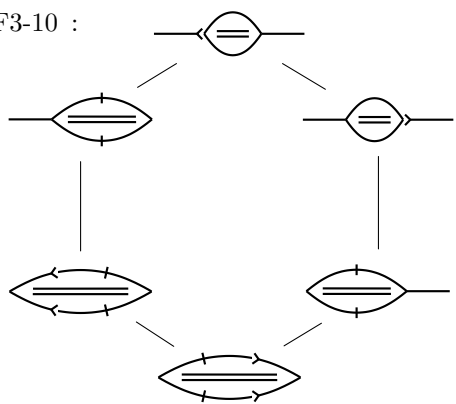
{
\xy/r\sce/: \xypolygon6{~={30}~*{\alphanum}}
\endxy
}
\vspace{10pt}

\item[F3-11]: 
\hspace*{-25pt}
\def\alphanum{\ifcase\xypolynode\or \ing{3stuff-item11-1.eps} \or \ing{3stuff-item11-2.eps} \or 
\ing{3stuff-item11-3.eps} \or \ing{3stuff-item11-4.eps} \or \ing{3stuff-item11-5.eps} \fi}
\hspace{\hsqza}
\db
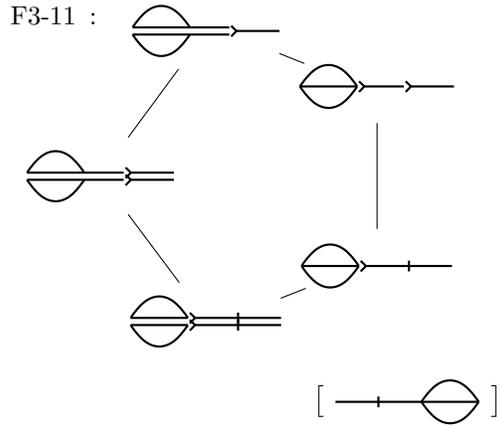
{
\xy/r\scd/: \xypolygon5{~={36}~*{\alphanum}}
\endxy
} \nopagebreak
\\[-5pt] 
\hspace*{\fill} $\big[\hspace{-1.5ex}\bb{\xys{\ing{3stuff-item11a.eps}}}\hspace{-1.5ex}\big]$ \hspace*{30pt} 
\vspace{5pt}

\item[F3-12]: 
\hspace*{-25pt}
\def\alphanum{\ifcase\xypolynode\or \ing{3stuff-item12-1.eps} \or \ing{3stuff-item12-2.eps} \or 
\ing{3stuff-item12-3.eps} \or \ing{3stuff-item12-4.eps} \or \ing{3stuff-item12-5.eps} \or \ing{3stuff-item12-6.eps} \fi}
\hspace{\hsqza}
\db
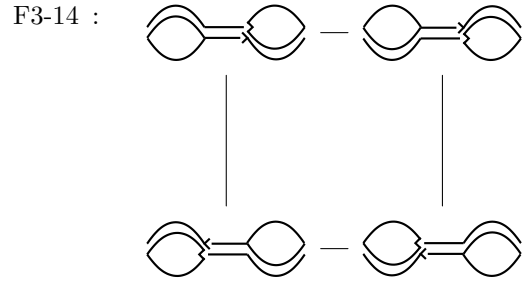
{
\xy/r\sce/: \xypolygon6{~={30}~*{\alphanum}}
\endxy
}
\vspace{10pt}

\item[F3-13]: 
\hspace*{-25pt}
\def\alphanum{\ifcase\xypolynode\or \ing{3stuff-item13-1.eps} \or \ing{3stuff-item13-2.eps} \or 
\ing{3stuff-item13-3.eps} \or \ing{3stuff-item13-4.eps} \or \ing{3stuff-item13-5.eps} \or \ing{3stuff-item13-6.eps} \fi}
\hspace{\hsqza}
\db
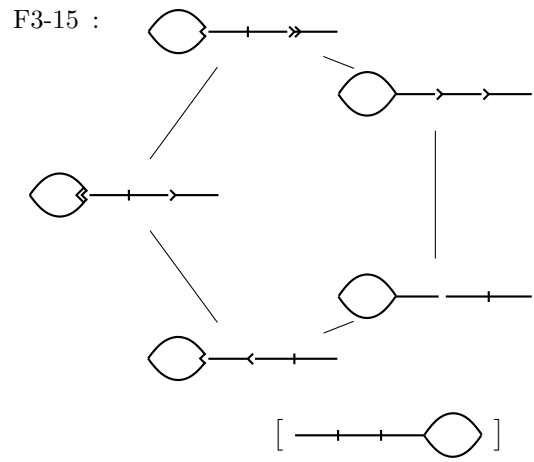
{
\xy/r\sce/: \xypolygon6{~={30}~*{\alphanum}}
\endxy
}
\\[-5pt]
\hspace*{\fill} $\big[\hspace{-1.5ex}\bb{\xys{\ing{3stuff-item13a.eps}}}\hspace{-1.5ex}\big]$ \hspace*{30pt} 
\vspace{5pt}

\item[F3-14]: 
\def\alphanum{\ifcase\xypolynode\or \ing{3stuff-item14-1.eps} \or \ing{3stuff-item14-2.eps} \or 
\ing{3stuff-item14-3.eps} \or \ing{3stuff-item14-4.eps} \fi}
\db
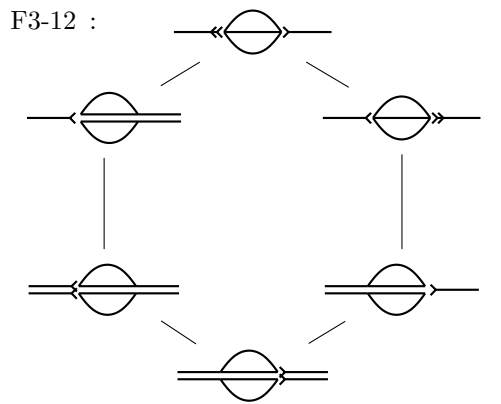
{
\xy/r\scd/: \xypolygon4{~*{\alphanum}}
\endxy
}
\vspace{10pt}

\item[F3-15]: 
\hspace*{15pt}
\def\alphanum{\ifcase\xypolynode\or \ing{3stuff-item15-1.eps} \or \ing{3stuff-item15-2.eps} \or 
\ing{3stuff-item15-3.eps} \or \ing{3stuff-item15-4.eps} \or \ing{3stuff-item15-5.eps} \fi}
\hspace{\hsqzb}
\db
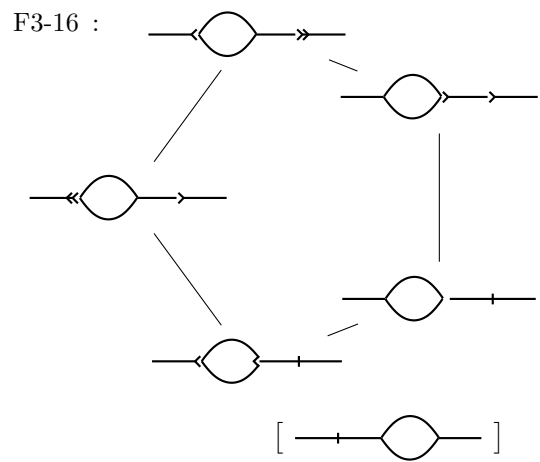
{
\xy/r\sce/: \xypolygon5{~={36}~*{\alphanum}}
\endxy
}
\\[-5pt]
\hspace*{\fill} $\big[\hspace{-1.5ex}\bb{\xys{\ing{3stuff-item15a.eps}}}\hspace{-1.5ex}\big]$ \hspace*{30pt} 
\vspace{5pt}

\item[F3-16]: 
\hspace*{15pt}
\def\alphanum{\ifcase\xypolynode\or \ing{3stuff-item16-1.eps} \or \ing{3stuff-item16-2.eps} \or 
\ing{3stuff-item16-3.eps} \or \ing{3stuff-item16-4.eps} \or \ing{3stuff-item16-5.eps} \fi}
\hspace{\hsqzb}
\db
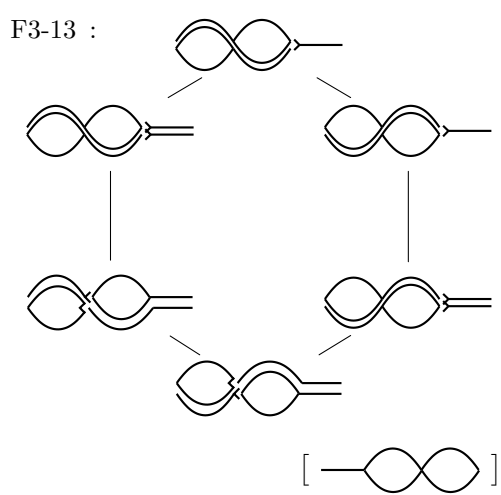
{
\xy/r\sce/: \xypolygon5{~={36}~*{\alphanum}}
\endxy
}
\\[-5pt]
\hspace*{\fill} $\big[\hspace{-1.5ex}\bb{\xys{\ing{3stuff-item16a.eps}}}\hspace{-1.5ex}\big]$ \hspace*{30pt} 
\vspace{5pt}

\item[F3-17]: 
\hspace*{15pt}
\def\alphanum{\ifcase\xypolynode\or \ing{3stuff-item17-1.eps} \or \ing{3stuff-item17-2.eps} \or 
\ing{3stuff-item17-3.eps} \or \ing{3stuff-item17-4.eps} \or \ing{3stuff-item17-5.eps} \fi}
\hspace{\hsqzb}
\db
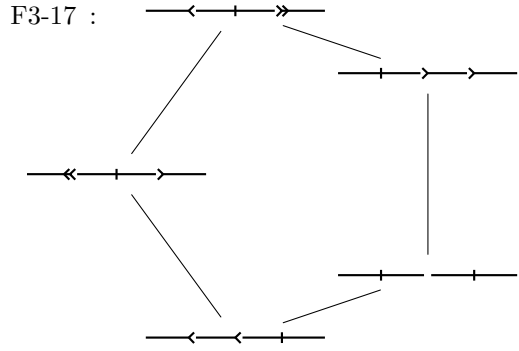
{
\xy/r\sce/: \xypolygon5{~={36}~*{\alphanum}}
\endxy
}
\vspace{10pt}

\item[F3-18]: 
\def\alphanum{\ifcase\xypolynode\or \ing{3stuff-item18-1.eps} \or \ing{3stuff-item18-2.eps} \or 
\ing{3stuff-item18-3.eps} \fi}
\db
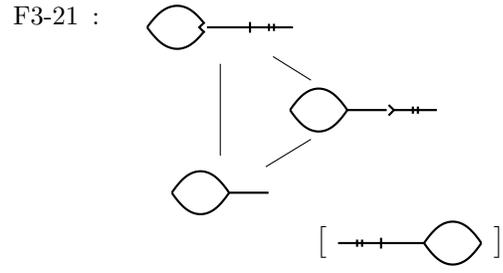
{
\xy/r\sca/: \xypolygon3{~={0}~*{\alphanum}}
\endxy
}
\\[-15pt]
\hspace*{\fill} $\big[\hspace{-1.5ex}\bb{\xys{\ing{3stuff-item18a.eps}}}\hspace{-1.5ex}\big]$ \hspace*{30pt}
\vspace{5pt}

\item[F3-19]:  
\def\alphanum{\ifcase\xypolynode\or \ing{3stuff-item19-1.eps} \or \ing{3stuff-item19-2.eps} \or 
\ing{3stuff-item19-3.eps} \fi}
\db
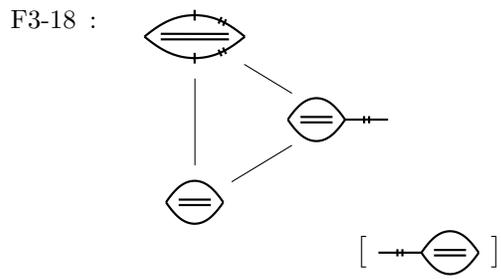
{
\xy/r\sca/: \xypolygon3{~={0}~*{\alphanum}}
\endxy
}
\\[-15pt]
\hspace*{\fill} $\big[\hspace{-1.5ex}\bb{\xys{\ing{3stuff-item19a.eps}}}\hspace{-1.5ex}\big]$ \hspace*{30pt}
\vspace{5pt}

\item[F3-20]:  
\def\alphanum{\ifcase\xypolynode\or \ing{3stuff-item20-1.eps} \or \ing{3stuff-item20-2.eps} \or 
\ing{3stuff-item20-3.eps} \fi}
\db
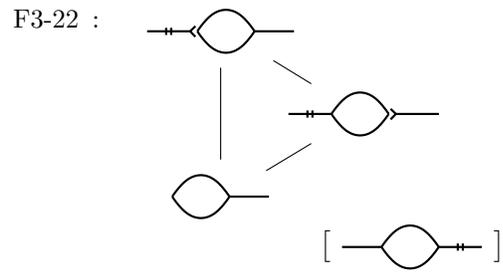
{
\xy/r\sca/: \xypolygon3{~={0}~*{\alphanum}}
\endxy
} \nopagebreak
\\[-15pt]
\hspace*{\fill} $\big[\hspace{-1.5ex}\bb{\xys{\ing{3stuff-item20a.eps}}}\hspace{-1.5ex}\big]$ \hspace*{30pt}
\vspace{5pt}

\item[F3-21]:
\def\alphanum{\ifcase\xypolynode\or \ing{3stuff-item21-1.eps} \or \ing{3stuff-item21-2.eps} \or 
\ing{3stuff-item21-3.eps} \fi}
\db
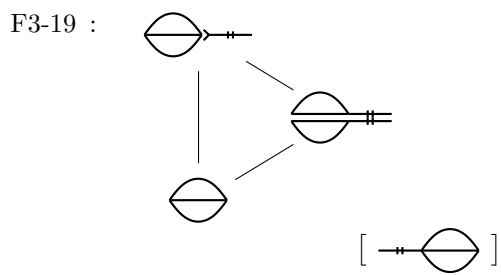
{
\xy/r\sca/: \xypolygon3{~={0}~*{\alphanum}}
\endxy
}
\\[-15pt]
\hspace*{\fill} $\big[\hspace{-1.5ex}\bb{\xys{\ing{3stuff-item21a.eps}}}\hspace{-1.5ex}\big]$ \hspace*{30pt}
\vspace{5pt}

\item[F3-22]: 
\def\alphanum{\ifcase\xypolynode\or \ing{3stuff-item22-1.eps} \or \ing{3stuff-item22-2.eps} \or 
\ing{3stuff-item22-3.eps} \fi}
\db
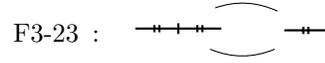
{
\xy/r\sca/: \xypolygon3{~={0}~*{\alphanum}}
\endxy
}
\\[-15pt]
\hspace*{\fill} $\big[\hspace{-1.5ex}\bb{\xys{\ing{3stuff-item22a.eps}}}\hspace{-1.5ex}\big]$ \hspace*{30pt}
\vspace{5pt}

\item[F3-23]: 
\db
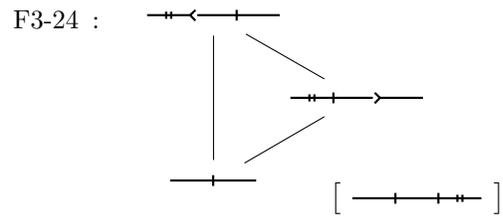
{
\xymatrix@C-12pt{\ing{3stuff-item23-1.eps} \ar@/^10pt/ @{-}[r] \ar@/_10pt/@{-}[r] & \ing{3stuff-item23-2.eps}}
}
\vspace{10pt}

\item[F3-24]: 
\def\alphanum{\ifcase\xypolynode\or \ing{3stuff-item24-1.eps} \or \ing{3stuff-item24-2.eps} \or 
\ing{3stuff-item24-3.eps} \fi}
\db
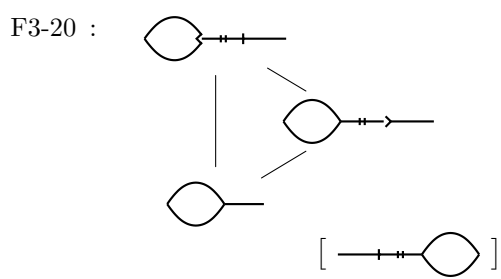
{
\xy/r\sca/: \xypolygon3{~={0}~*{\alphanum}}
\endxy
}
\\[-15pt]
\hspace*{\fill} $\big[\hspace{-1.5ex}\bb{\xys{\ing{3stuff-item24a.eps}}}\hspace{-1.5ex}\big]$ \hspace*{30pt}
\vspace{5pt}

\item[F3-25]: 
\def\alphanum{\ifcase\xypolynode\or \ing{3stuff-item25-1.eps} \or \ing{3stuff-item25-2.eps} \or 
\ing{3stuff-item25-3.eps} \fi}
\db
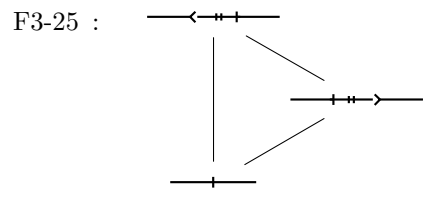
{
\xy/r\sca/: \xypolygon3{~={0}~*{\alphanum}}
\endxy
}

\end{multicols}
\end{itemize}

%\end{description}
\end{definition}

As in the case of category objects in 2-categories, when the ambient 2-category is $\CAT$ or $\SMC$, we modify the definition of 2-category object to allow $C_0$ and $C_1$ to be large categories.  A similar remark applies to the upcoming definitions of dicategory object and bicategory object, and we will not repeat it there.

In a 2-category $M=(M_0, M_1, M_2)$, the pair $(M_0, M_1)$ of objects and morphisms, and the pair $(M_1, M_2)$ of morphisms and 2-morphisms both form categories in their own right.  Similarly, the pairs $(C_0, C_1)$ and $(C_1,C_2)$ extracted from a 2-category object form category objects.  
\begin{prop} \label{prop-subcatobj}
Let $C=(C_0, C_1, C_2)$ be a 2-category object in the 2-category $\cC$.  The pair $(C_0,C_1)$ forms a category object in $\cC$, while the pair $(C_1, C_2)$ forms a not-necessarily fibrant category object in $\cC$---see Remark~\ref{rem-nnf}.

If $\cC$ is $\CAT$, then for fixed objects $a,b \in C_0$, let $C_1(a,b)$ be the subcategory of $C_1$ whose objects have source $a$ and target $b$ and whose morphisms have source the identity on $a$ and target the identity on $b$, and let $C_2(a,b)$ be the subcategory of $C_2$ whose objects have source source $a$ and target target $b$ and whose morphisms have source source the identity on $a$ and target target the identity on $b$.  In this case the pair $(C_1(a,b),C_2(a,b))$ forms a category object in $\CAT$.
\end{prop}
\begin{proof}
The pair of objects $(C_0,C_1)$ may be equipped with the data [F1-1], [F1-2], [F2-15], [F2-16], and [F2-12], subject to the axioms [F3-23], [F3-24], [F3-25], and [F3-17].  These are precisely the requirements for a category object.  Similarly the pair of objects $(C_1,C_2)$ may be equipped with the data [F1-3], [F1-4], [F2-1], [F2-2], and [F2-3], subject to the axioms [F3-1], [F3-2], [F3-3], and [F3-4].  These are precisely the requirements for a category object, except that it may not be the case that $s \times t: C_2 \ra C_1 \times C_1$ is a fibration.  However, when $\cC=\CAT$, if we restrict to the subcategories $C_1(a,b)$ and $C_2(a,b)$, then the condition that $s \times t: C_2 \ra C_1 \times_{C_0 \times C_0} C_1$ be a fibration implies that $s \times t: C_2(a,b) \ra C_1(a,b) \times C_1(a,b)$ is a fibration.
\end{proof}

\subsection{Dicategories in 2-categories} \label{sec-dicatin2cat} $\left(\vcenter{\xymatrix@R=3pt@C=3pt@M=0pt{
\scriptscriptstyle\bullet & \scriptscriptstyle\bullet \\ \scriptscriptstyle\bullet & \scriptscriptstyle\bullet \\ \scriptscriptstyle\bullet & \scriptscriptstyle\bullet
}}\right)$

It often happens that a structure we would like to form a 2-category doesn't quite, and we are forced to work with bicategories instead.  For example, algebras, bimodules, and bimodule maps form a bicategory, not a 2-category, because tensor product is associative and unital only up to isomorphism.  We could hunt down a 2-category equivalent to this bicategory of algebras, but it is wiser to come to terms with the additional complexity of bicategories.

A similar phenomenon happens when working with 2-category objects---that is, natural examples are a bit less strict than we would like.  For instance, the most naive presentation of the bordism category $\Bord_0^3$ of $0$-, $1$-, $2$-, and $3$-manifolds (in which the 1-cell identity [F1-1] sends a 0-manifold $M$ to $M \times [0,1]$) does not form a 2-category object in categories; the identity transformations for 1-manifolds (that is the comparisons between a 1-dimensional bordism and the composite of that bordism with the identity on its source or target) do not exist as \emph{natural} isomorphisms of functors.  Instead, $\Bord_0^3$ is a dicategory object in categories---this structure is nearly the same as a 2-category object, except that the 1-cell identities are given by 2-cells rather than 1-cell morphisms.  Similarly, the 3-level category of conformal nets is not a 2-category object, but in fact a dicategory object.  Indeed, it was our investigation of conformal nets that motivated us to introduce the notion of dicategory object.

For brevity, we refer as much as possible of the definition of a dicategory object to the corresponding structure in a 2-category object.
\begin{definition} \label{def-dicat}
A dicategory object $C$ in the 2-category $\cC$ consists of the following three collections of data, subject to the listed axioms.

\begin{description}

\item[0-data]
The 0-data is exactly as for a 2-category object; in particular, there are three objects $C_0$, $C_1$, and $C_2$ of $\cC$, with $C_0$ and $C_1$ being groupoid objects.

\item[1-data]
There are eight pieces of 1-data, named [D1-1] - [D1-8].
The first six are the 1-data morphisms [F1-1] - [F1-6] for a 2-category object.  There are two additional morphisms of $\cC$ as follows:
\begin{itemize}
\item[D1-7]: $i_l: C_1 \ra C_2$, denoted \cb{\ing{d1stuff-item7.eps}} and called the (upper) left 2-cell identity.
\item[D1-8]: $i_r: C_1 \ra C_2$, denoted \cb{\ing{d1stuff-item8.eps}} and called the (upper) right 2-cell identity.
\end{itemize}
These morphisms are compatible with source and target maps. \\
\indent The morphisms [D1-7] and [D1-8] are required to be invertible, in the following sense.  There exists a morphism of $\cC$, denoted \cb{\ing{d1stuff-item9.eps}} (the lower left 2-cell identity), such that there are invertible 2-morphisms from \cb{\ing{d1stuff-halfid1}} to \cb{\ing{d1stuff-vertid1}} and from \cb{\ing{d1stuff-halfid2}} to \cb{\ing{d1stuff-vertid2}}, such that the two resulting 2-morphisms from \cb{\ing{d1stuff-triangle1}} to \cb{\ing{d1stuff-item7.eps}} are equal, and similarly the two 2-morphisms from \cb{\ing{d1stuff-triangle2}} to \cb{\ing{d1stuff-item9.eps}} are equal --- these are versions of the usual triangle identity for invertible morphisms in a bicategory.  Similarly there exists a morphism \cb{\ing{d1stuff-item10.eps}} (the lower right 2-cell identity) satisfying the corresponding conditions.

\item[2-data]
There are eighteen pieces of 2-data, named [D2-1] - [D2-18].  The first twelve are the twelve 2-isomorphisms [F2-1] - [F2-12] from the definition of a 2-category object; these are the pieces of 2-data not involving the 1-cell identity.  There are six additional 2-isomorphisms of $\cC$ as follows:
\begin{itemize}
\begin{multicols}{2}
\item[D2-13]: \xymatrix@C-12pt{\cb{\ing{d2stuff-item13-1.eps}}   \ar@{=>}[r] &   \cb{\ing{d2stuff-item13-2.eps}}} 
\item[D2-14]: \xymatrix@C-12pt{\cb{\ing{d2stuff-item14-1.eps}}   \ar@{=>}[r] &   \cb{\ing{d2stuff-item14-2.eps}}} 
\item[D2-15]: \xymatrix@C-12pt{\cb{\ing{d2stuff-item15-1.eps}}   \ar@{=>}[r] &   \cb{\ing{d2stuff-item15-2.eps}}} 
\item[D2-16]: \xymatrix@C-12pt{\cb{\ing{d2stuff-item16-1.eps}}   \ar@{=>}[r] &   \cb{\ing{d2stuff-item16-2.eps}}} 
\item[D2-17]: \xymatrix@C-12pt{\cb{\ing{d2stuff-item17-1.eps}}   \ar@{=>}[r] &   \cb{\ing{d2stuff-item17-2.eps}}} 
\item[D2-18]: \xymatrix@C-12pt{\cb{\ing{d2stuff-item18-1.eps}}   \ar@{=>}[r] &   \cb{\ing{d2stuff-item18-2.eps}}} 
\end{multicols}
\end{itemize}
\nid These 2-isomorphisms are compatible with source and target maps in the sense that the sources and targets of [D2-1] through [D2-8] and of [D2-13] through [D2-15] are identity 2-morphisms, the sources and targets of [D2-9] through [D2-11] are the 2-isomorphism [D2-12], the sources of [D2-16] and [D2-18] are [D2-12], the source of [D2-17] is the inverse of [D2-12], and the targets of [D2-16] through [D2-18] are identity 2-isomorphisms.

\item[3-axioms]
The above data are subject to twenty-six axioms, named [D3-1] - [D3-26].  The first seventeen are the seventeen axioms [F3-1] - [F3-17] from the definition of a 2-category object; these are the axioms not involving the 1-cell identity.  There are nine additional axioms, as well as variant axioms abbreviated in parentheses, as follows:

\end{description}

\begin{itemize}
\begin{multicols}{2}
\item[D3-18]:
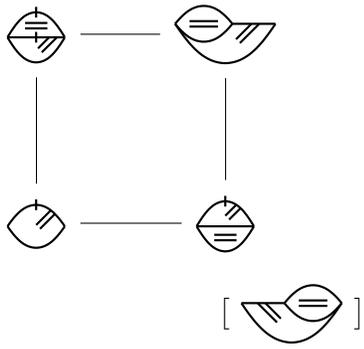
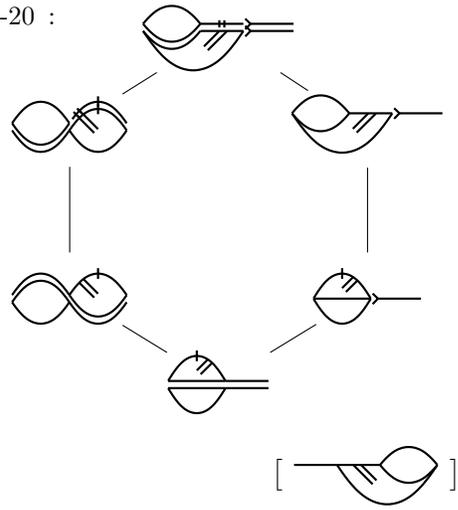
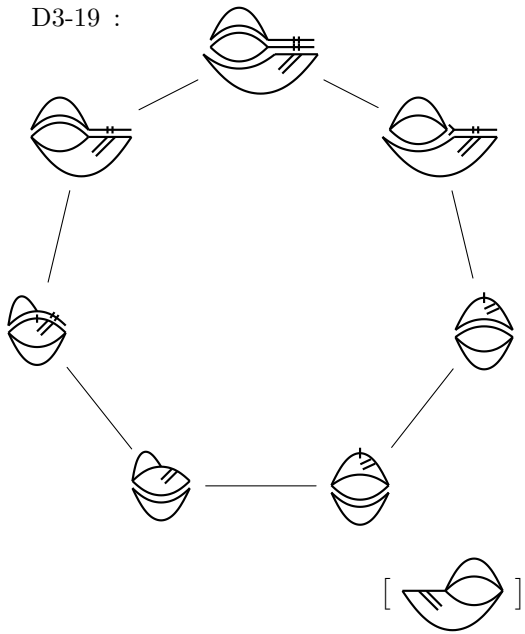
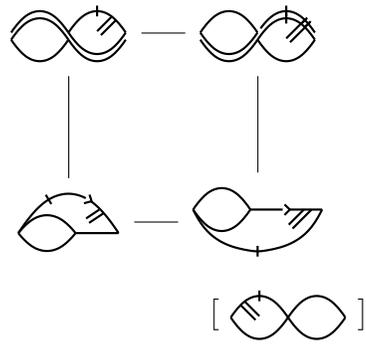
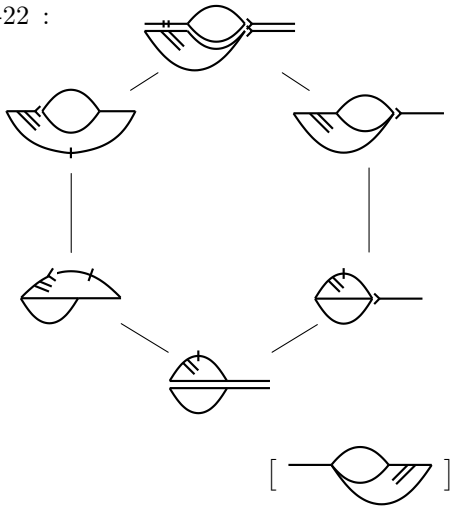
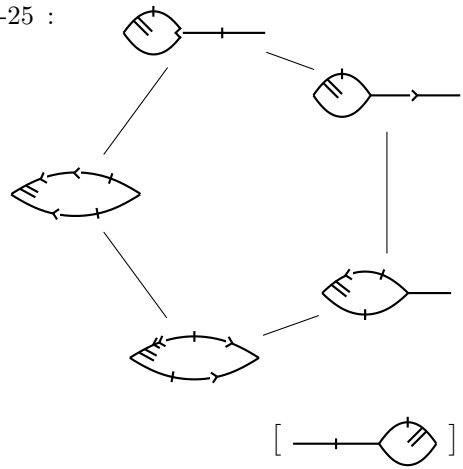
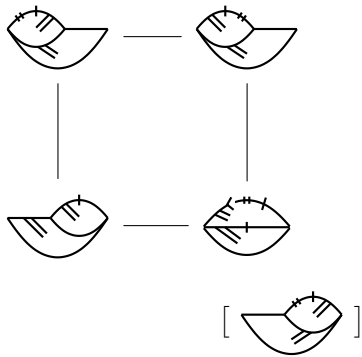
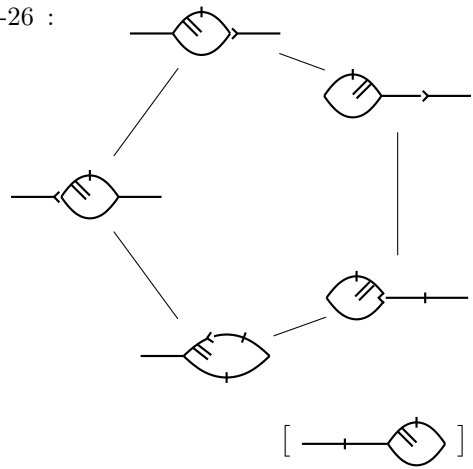
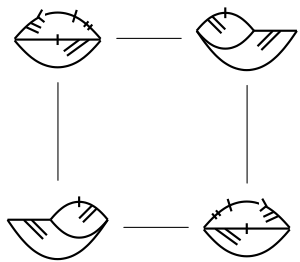
\def\alphanum{\ifcase\xypolynode\or \ing{d3stuff-item18-1.eps} \or \ing{d3stuff-item18-2.eps} \or 
\ing{d3stuff-item18-3.eps} \or \ing{d3stuff-item18-4.eps} \fi}
\db{
\xy/r\scc/: \xypolygon4{~*{\alphanum}}
\endxy
}
\\[-5pt]
\hspace*{\fill} $\big[\hspace{-1.5ex}\bb{\xys{\ing{d3stuff-item18a.eps}}}\hspace{-1.5ex}\big]$ \hspace*{30pt}
\vspace{5pt}

\item[D3-19]:
\def\alphanum{\ifcase\xypolynode\or \ing{d3stuff-item19-1.eps} \or \ing{d3stuff-item19-2.eps} \or \ing{d3stuff-item19-3.eps} \or \ing{d3stuff-item19-4.eps} \or \ing{d3stuff-item19-5.eps} \or \ing{d3stuff-item19-6.eps} \or \ing{d3stuff-item19-7.eps} \fi}
\hspace{\hsqzb}
\db{
\xy/r\scg/:  \xypolygon7{~*{\alphanum}}
\endxy
}
\\[-5pt]
\hspace*{\fill} $\big[\hspace{-1.5ex}\bb{\xys{\ing{d3stuff-item19a.eps}}}\hspace{-1.5ex}\big]$ \hspace*{30pt}
\vspace{5pt}

\item[D3-20]:
\hspace*{-15pt}
\def\alphanum{\ifcase\xypolynode\or \ing{d3stuff-item20-1.eps} \or \ing{d3stuff-item20-2.eps} \or 
\ing{d3stuff-item20-3.eps} \or \ing{d3stuff-item20-4.eps} \or \ing{d3stuff-item20-5.eps} \or \ing{d3stuff-item20-6.eps} \fi}
\hspace{\hsqza}
\db{
\xy/r\sce/: \xypolygon6{~={30}~*{\alphanum}}
\endxy
}
\\[-5pt]
\hspace*{\fill} $\big[\hspace{-1.5ex}\bb{\xys{\ing{d3stuff-item20a.eps}}}\hspace{-1.5ex}\big]$ \hspace*{30pt}
\vspace{5pt}

\item[D3-21]:
\def\alphanum{\ifcase\xypolynode\or \ing{d3stuff-item21-1.eps} \or \ing{d3stuff-item21-2.eps} \or 
\ing{d3stuff-item21-3.eps} \or \ing{d3stuff-item21-4.eps} \fi}
\db{
\xy/r\scc/: \xypolygon4{~*{\alphanum}}
\endxy
}
\\[-5pt]
\hspace*{\fill} $\big[\hspace{-1.5ex}\bb{\xys{\ing{d3stuff-item21a.eps}}}\hspace{-1.5ex}\big]$ \hspace*{30pt}
\vspace{5pt}

\item[D3-22]:
\hspace*{-15pt}
\def\alphanum{\ifcase\xypolynode\or \ing{d3stuff-item22-1.eps} \or \ing{d3stuff-item22-2.eps} \or 
\ing{d3stuff-item22-3.eps} \or \ing{d3stuff-item22-4.eps} \or \ing{d3stuff-item22-5.eps} \or \ing{d3stuff-item22-6.eps} \fi}
\hspace{\hsqza}
\db{
\xy/r\sce/: \xypolygon6{~={30}~*{\alphanum}}
\endxy
} \nopagebreak
\\[-5pt]
\hspace*{\fill} $\big[\hspace{-1.5ex}\bb{\xys{\ing{d3stuff-item22a.eps}}}\hspace{-1.5ex}\big]$ \hspace*{30pt}
\vspace{5pt}

\item[D3-23]:
\def\alphanum{\ifcase\xypolynode\or \ing{d3stuff-item23-1.eps} \or \ing{d3stuff-item23-2.eps} \or 
\ing{d3stuff-item23-3.eps} \or \ing{d3stuff-item23-4.eps} \fi}
\db{
\xy/r\scc/: \xypolygon4{~*{\alphanum}}
\endxy
}
\\[-5pt]
\hspace*{\fill} $\big[\hspace{-1.5ex}\bb{\xys{\ing{d3stuff-item23a.eps}}}\hspace{-1.5ex}\big]$ \hspace*{30pt}
\vspace{5pt}

\item[D3-24]:
\def\alphanum{\ifcase\xypolynode\or \ing{d3stuff-item24-1.eps} \or \ing{d3stuff-item24-2.eps} \or 
\ing{d3stuff-item24-3.eps} \or \ing{d3stuff-item24-4.eps} \fi}
\db{
\xy/r\scc/: \xypolygon4{~*{\alphanum}}
\endxy
}
\vspace{10pt}

\item[D3-25]:
\hspace*{-15pt}
\def\alphanum{\ifcase\xypolynode\or \ing{d3stuff-item25-1.eps} \or \ing{d3stuff-item25-2.eps} \or 
\ing{d3stuff-item25-3.eps} \or \ing{d3stuff-item25-4.eps} \or \ing{d3stuff-item25-5.eps} \fi}
\hspace{\hsqza}
\db{
\xy/r\sce/: \xypolygon5{~={36}~*{\alphanum}}
\endxy
}
\\[-5pt]
\hspace*{\fill} $\big[\hspace{-1.5ex}\bb{\xys{\ing{d3stuff-item25a.eps}}}\hspace{-1.5ex}\big]$ \hspace*{30pt}
\vspace{5pt}

\item[D3-26]:
\hspace*{-15pt}
\def\alphanum{\ifcase\xypolynode\or \ing{d3stuff-item26-1.eps} \or \ing{d3stuff-item26-2.eps} \or 
\ing{d3stuff-item26-3.eps} \or \ing{d3stuff-item26-4.eps} \or \ing{d3stuff-item26-5.eps} \fi}
\hspace{\hsqza}
\db{
\xy/r\sce/: \xypolygon5{~={36}~*{\alphanum}}
\endxy
}
\\[-5pt]
\hspace*{\fill} $\big[\hspace{-1.5ex}\bb{\xys{\ing{d3stuff-item26a.eps}}}\hspace{-1.5ex}\big]$ \hspace*{30pt} 
\end{multicols}
\end{itemize}
\end{definition}

\subsection{Bicategories in 2-categories} \label{sec-bicatin2cat} $\left(\vcenter{\xymatrix@R=3pt@C=3pt@M=0pt{
\scriptscriptstyle\bullet & \scriptscriptstyle\bullet \\ \scriptscriptstyle\bullet & \scriptscriptstyle\bullet \\ \scriptscriptstyle\bullet & \scriptscriptstyle\bullet
}}\right)$

The natural geometric and algebraic examples of symmetric monoidal 3-categories appear to arise as dicategory objects in the 2-category of symmetric monoidal categories---that is, the associator of a composition of 1-cells is a symmetric monoidal natural transformation between the composition functors, rather than a 2-cell object with source and target the respective associations.  As such, bicategory objects are not strictly necessary, for instance, for our investigation of conformal nets.  However, they provide a natural bridge between dicategory objects in $\CAT$ and the notion of tricategory, and therefore warrant discussion.

The structure of a bicategory object is close to that of a dicategory object, and for brevity we refer most of the definition back to the corresponding elements of the earlier definitions.  An abbreviated list of all the data and axioms of a bicategory object is depicted in Table~\ref{table-dataax}, in the Appendix.
\begin{definition}
A bicategory object $C$ in the 2-category $\cC$ consists of the following three collections of data, subject to the listed axioms.
\begin{description}

\item[0-data]
The 0-data is exactly as for a dicategory object; in particular, there are three objects $C_0$, $C_1$, and $C_2$ of $\cC$, with $C_0$ and $C_1$ being groupoid objects.

\item[1-data]
There are nine pieces of 1-data, named [B1-1] - [B1-9].  The first eight are respectively the 1-data morphisms [F1-1] - [F1-6] for a 2-category object and [D1-7] - [D1-8] for a dicategory object.  There is one additional morphism of $\cC$ as follows:
\begin{itemize}
\item[B1-9]: $a: C_1 \times_{C_0} C_1 \times_{C_0} C_1 \ra C_2$, denoted \cb{\ing{b1stuff-item9.eps}} and called the associator 2-cell.  This morphism is compatible with the source and target maps.  Moreover this morphism is required to be invertible, in the sense that there exists a morphism, denoted \cb{\ing{b1stuff-item9a.eps}}, which is a two-sided inverse to \cb{\ing{b1stuff-item9.eps}} up to invertible 2-morphisms satisfying the triangle identities.
\end{itemize}

\item[2-data]
There are eighteen pieces of 2-data, named [B2-1] - [B2-18].  The first eight are the 2-isomorphisms [F2-1] - [F2-8] from the definition of a 2-category object; these are the pieces of 2-data not involving the 1-cell identity or an association of 1-cells.  Further the 2-data [B2-13] - [B2-15] are the three 2-isomorphisms [D2-13] - [D2-15] from the definition of a dicategory object; these involve 1-cell identities but not association of 1-cells.  The remaining dicategory 2-data are replaced by the following seven 2-isomorphisms of $\cC$:
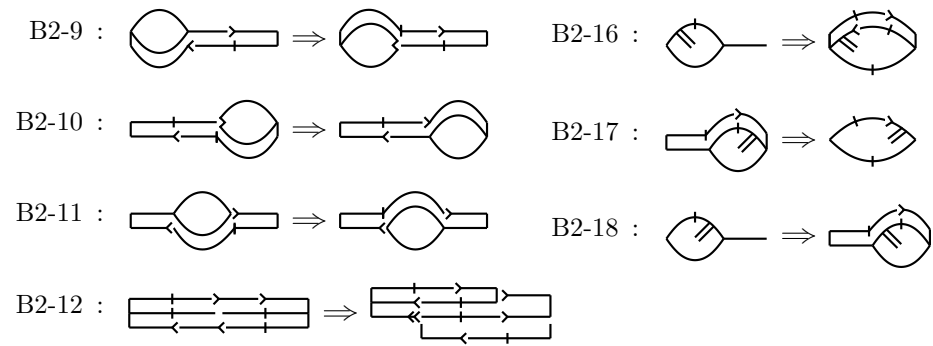
\begin{multicols}{2}
\begin{itemize}
\item[B2-9]: \xymatrix@C-12pt{\cb{\ing{b2stuff-item9-1.eps}}   \ar@{=>}[r] &   \cb{\ing{b2stuff-item9-2.eps}}} 
\item[B2-10]: \xymatrix@C-12pt{\cb{\ing{b2stuff-item10-1.eps}}   \ar@{=>}[r] &   \cb{\ing{b2stuff-item10-2.eps}}} 
\item[B2-11]: \xymatrix@C-12pt{\cb{\ing{b2stuff-item11-1.eps}}   \ar@{=>}[r] &   \cb{\ing{b2stuff-item11-2.eps}}} 
\item[B2-12]: \xymatrix@C-12pt{\cb{\ing{b2stuff-item12-1.eps}}   \ar@{=>}[r] &   \cb{\ing{b2stuff-item12-2.eps}}} 
\item[B2-16]: \xymatrix@C-12pt{\cb{\ing{b2stuff-item16-1.eps}}   \ar@{=>}[r] &   \cb{\ing{b2stuff-item16-2.eps}}} 
\item[B2-17]: \xymatrix@C-12pt{\cb{\ing{b2stuff-item17-1.eps}}   \ar@{=>}[r] &   \cb{\ing{b2stuff-item17-2.eps}}} 
\item[B2-18]: \xymatrix@C-12pt{\cb{\ing{b2stuff-item18-1.eps}}   \ar@{=>}[r] &   \cb{\ing{b2stuff-item18-2.eps}}} 
\end{itemize}
\end{multicols}
\nid The sources and targets of all the 2-data 2-isomorphisms are identity 2-isomorphisms.

\item[3-axioms]
The above data are subject to twenty-six axioms, named [B3-1] - [B3-26].  The eight axioms [B3-1] - [B3-8] are exactly the 2-category axioms [F3-1] - [F3-8], and the two axioms [B3-18] - [B3-19] are exactly the dicategory axioms [D3-18] - [D3-19].  
The six axioms [B3-9] - [B3-14] and the five axioms [B3-20] - [B3-24] are obtained from the axioms [F3-9] - [F3-14] and [D3-20] - [D3-24] by inserting the necessary associator 2-cells; we do not redraw these minor variations.  The remaining five dicategory axioms, namely [F3-15], [F3-16], [F3-17], [D3-25], and [D3-26], are replaced by the following five axioms, together with the variant axioms abbreviated in parentheses:
\end{description}
\pagebreak

\begin{multicols}{2}

\hspace*{-15pt}
B3-15:\nopagebreak\\[5pt]
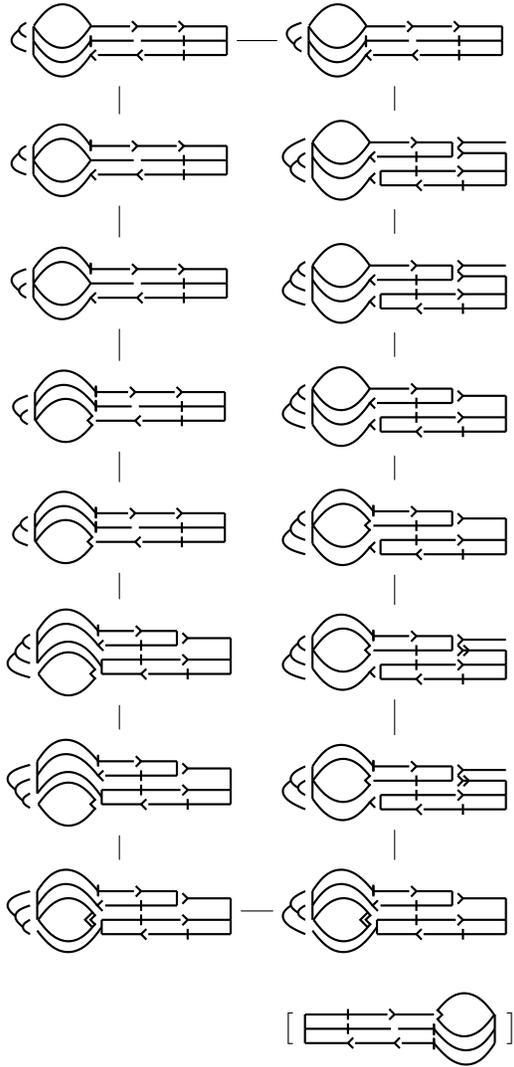
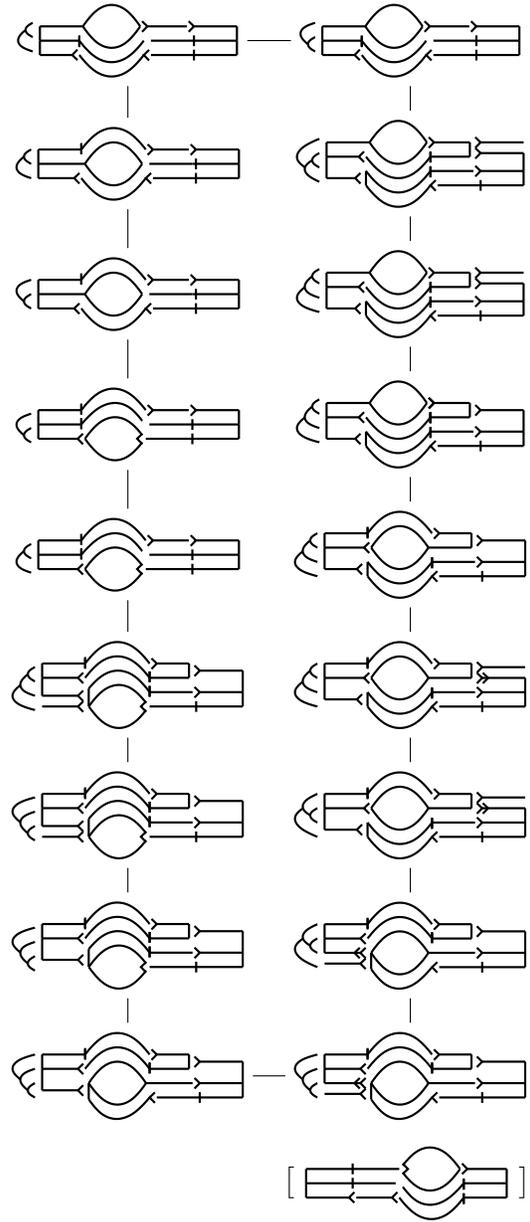
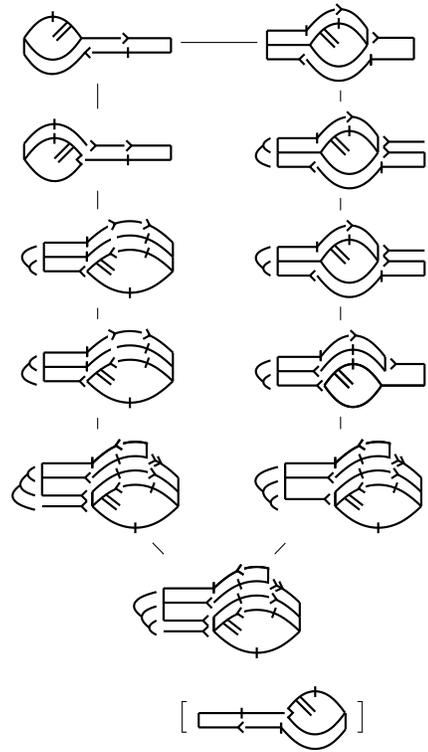
\xymatrix@C-12pt@R-15pt@M-2pt{\cb{\ing{b3stuff-item15-1A.eps}} \ar@{-}[r] \ar@{-}[d] & \cb{\ing{b3stuff-item15-1B.eps}} \ar@{-}[d] \\
\cb{\ing{b3stuff-item15-9B.eps}} \ar@{-}[d] & \cb{\ing{b3stuff-item15-2A.eps}} \ar@{-}[d] \\
\cb{\ing{b3stuff-item15-9A.eps}} \ar@{-}[d] & \cb{\ing{b3stuff-item15-2B.eps}} \ar@{-}[d] \\
\cb{\ing{b3stuff-item15-8B.eps}} \ar@{-}[d] & \cb{\ing{b3stuff-item15-3A.eps}} \ar@{-}[d] \\
\cb{\ing{b3stuff-item15-8A.eps}} \ar@{-}[d] & \cb{\ing{b3stuff-item15-4A.eps}} \ar@{-}[d] \\
\cb{\ing{b3stuff-item15-7B.eps}} \ar@{-}[d] & \cb{\ing{b3stuff-item15-5A.eps}} \ar@{-}[d] \\
\cb{\ing{b3stuff-item15-7A.eps}} \ar@{-}[d] & \cb{\ing{b3stuff-item15-5B.eps}} \ar@{-}[d] \\
\cb{\ing{b3stuff-item15-6B.eps}} \ar@{-}[r] & \cb{\ing{b3stuff-item15-6A.eps}}\\
}\\
\hspace*{\fill} $\big[\hspace{-1.5ex}\bb{\xys{\ing{b3stuff-item15a.eps}}}\hspace{-1.5ex}\big]$ \hspace*{10pt}

B3-16:\nopagebreak\\[5pt]
\hspace*{10pt}
\xymatrix@C-12pt@R-15pt@M-2pt{\cb{\ing{b3stuff-item16-1A.eps}} \ar@{-}[r] \ar@{-}[d] & \cb{\ing{b3stuff-item16-1B.eps}} \ar@{-}[d] \\
\cb{\ing{b3stuff-item16-11B.eps}} \ar@{-}[d] & \cb{\ing{b3stuff-item16-2A.eps}} \ar@{-}[d] \\
\cb{\ing{b3stuff-item16-11A.eps}} \ar@{-}[d] & \cb{\ing{b3stuff-item16-2B.eps}} \ar@{-}[d] \\
\cb{\ing{b3stuff-item16-10B.eps}} \ar@{-}[d] & \cb{\ing{b3stuff-item16-3A.eps}} \ar@{-}[d] \\
\cb{\ing{b3stuff-item16-10A.eps}} \ar@{-}[d] & \cb{\ing{b3stuff-item16-4A.eps}} \ar@{-}[d] \\
\cb{\ing{b3stuff-item16-9B.eps}} \ar@{-}[d] & \cb{\ing{b3stuff-item16-5A.eps}} \ar@{-}[d] \\
\cb{\ing{b3stuff-item16-9A.eps}} \ar@{-}[d] & \cb{\ing{b3stuff-item16-5B.eps}} \ar@{-}[d] \\
\cb{\ing{b3stuff-item16-8A.eps}} \ar@{-}[d] & \cb{\ing{b3stuff-item16-6A.eps}} \ar@{-}[d] \\
\cb{\ing{b3stuff-item16-7A.eps}} \ar@{-}[r] & \cb{\ing{b3stuff-item16-6B.eps}}\\
}\\[-5pt]
\hspace*{\fill} $\big[\hspace{-1.5ex}\bb{\xys{\ing{b3stuff-item16a.eps}}}\hspace{-1.5ex}\big]$ \hspace*{-5pt}

\hspace*{-15pt}
B3-17:\nopagebreak\\[5pt]
\xymatrix@C-12pt@R-15pt@M-2pt{\cb{\ing{b3stuff-item17-1A.eps}} \ar@{-}[r] \ar@{-}[d] & \cb{\ing{b3stuff-item17-13B.eps}} \ar@{-}[d] \\
\cb{\ing{b3stuff-item17-1B.eps}} \ar@{-}[d] & \cb{\ing{b3stuff-item17-13A.eps}} \ar@{-}[d] \\
\cb{\ing{b3stuff-item17-2A.eps}} \ar@{-}[d] & \cb{\ing{b3stuff-item17-12A.eps}} \ar@{-}[d] \\
\cb{\ing{b3stuff-item17-3A.eps}} \ar@{-}[d] & \cb{\ing{b3stuff-item17-11A.eps}} \ar@{-}[d] \\
\cb{\ing{b3stuff-item17-4A.eps}} \ar@{-}[d] & \cb{\ing{b3stuff-item17-10B.eps}} \ar@{-}[d] \\
\cb{\ing{b3stuff-item17-4B.eps}} \ar@{-}[d] & \cb{\ing{b3stuff-item17-10A.eps}} \ar@{-}[d] \\
\cb{\ing{b3stuff-item17-5A.eps}} \ar@{-}[d] & \cb{\ing{b3stuff-item17-9B.eps}} \ar@{-}[d] \\
\cb{\ing{b3stuff-item17-5B.eps}} \ar@{-}[d] & \cb{\ing{b3stuff-item17-9A.eps}} \ar@{-}[d] \\
\cb{\ing{b3stuff-item17-6A.eps}} \ar@{-}[d] & \cb{\ing{b3stuff-item17-8B.eps}} \ar@{-}[d] \\
\cb{\ing{b3stuff-item17-6B.eps}} \ar@{-}[d] & \cb{\ing{b3stuff-item17-8A.eps}} \ar@{-}[d] \\
\cb{\ing{b3stuff-item17-7A.eps}} \ar@{-}[r] & \cb{\ing{b3stuff-item17-7B.eps}}\\
}

\hspace*{25pt}
B3-25:\nopagebreak\\[5pt]
\hspace*{35pt}
\xymatrix@C-12pt@R-20pt@M-2pt{\cb{\ing{b3stuff-item25-1.eps}} \ar@{-}[r] \ar@{-}[d] & \cb{\ing{b3stuff-item25-7.eps}} \ar@{-}[d] \\
\cb{\ing{b3stuff-item25-2.eps}} \ar@{-}[d] & \cb{\ing{b3stuff-item25-6A.eps}} \ar@{-}[d] \\
\cb{\ing{b3stuff-item25-3A.eps}} \ar@{-}[d] & \cb{\ing{b3stuff-item25-5B.eps}} \ar@{-}[d] \\
\cb{\ing{b3stuff-item25-3B.eps}} \ar@{-}[d] & \cb{\ing{b3stuff-item25-5A.eps}} \ar@{-}[d] \\
\cb{\ing{b3stuff-item25-4A.eps}} \ar@{-}[r] & \cb{\ing{b3stuff-item25-4B.eps}}\\
}\\[-5pt]
\hspace*{\fill} $\big[\hspace{-1.5ex}\bb{\xys{\ing{b3stuff-item25a.eps}}}\hspace{-1.5ex}\big]$ \hspace*{30pt}

\hspace*{25pt}
B3-26:\nopagebreak\\[5pt]
\hspace*{35pt}
\xymatrix@C-50pt@R-20pt@M-2pt{\cb{\ing{b3stuff-item26-1.eps}} \ar@{-}[rr] \ar@{-}[d] && \cb{\ing{b3stuff-item26-2.eps}} \ar@{-}[d] \\
\cb{\ing{b3stuff-item26-8.eps}} \ar@{-}[d] && \cb{\ing{b3stuff-item26-3A.eps}} \ar@{-}[d] \\
\cb{\ing{b3stuff-item26-7B.eps}} \ar@{-}[d] && \cb{\ing{b3stuff-item26-3B.eps}} \ar@{-}[d] \\
\cb{\ing{b3stuff-item26-7A.eps}} \ar@{-}[d] && \cb{\ing{b3stuff-item26-4A.eps}} \ar@{-}[d] \\
\cb{\ing{b3stuff-item26-6B.eps}} \ar@{-}[dr] && \cb{\ing{b3stuff-item26-5A.eps}} \ar@{-}[dl]\\
& \cb{\ing{b3stuff-item26-6A.eps}} &
}\\[-5pt]
\hspace*{\fill} $\big[\hspace{-1.5ex}\bb{\xys{\ing{b3stuff-item26a.eps}}}\hspace{-1.5ex}\big]$ \hspace*{30pt}

\end{multicols}

\nid In the above pictures, the trees on the left indicate the order of vertical association.  Note that in these axioms the edges that indicate a change of vertical association may represent multiple applications of the vertical associator 2-morphism [F2-3].  In these cases, which sequence of vertical associators is chosen does not affect the content of the axiom, because of axiom [F3-4].

\end{definition}

\begin{remark}
In the modification of axioms [F3-9] through [F3-14] and [D3-20] through [D3-24] into bicategory axioms, there is one morphism that deserve comment, namely the one that appears on the left in axiom [F3-9] and on the lower left in axiom [F3-10].  This 2-morphism is replaced, in axioms [B3-9] and [B3-10], by a 2-morphism from the vertical composite of an associator 2-cell with an identity to the vertical composite of an identity with an associator 2-cell.  That 2-morphism is \emph{defined} to be the composition of the 2-morphism [F2-2] with the inverse of the 2-morphism [F2-1].
\end{remark}

\begin{remark}
Note that for a bicategory object $C=(C_0,C_1,C_2)$ in the 2-category $\cC$, it is not the case that the pair $(C_0,C_1)$ forms a category object, but it is still the case, as in Proposition~\ref{prop-subcatobj}, that the pair $(C_1,C_2)$ forms a not-necessarily fibrant category object in $\cC$.
\end{remark}

\vspace{12pt}

\subsection{2-categories are dicategories are bicategories} \label{sec-2dibi}

We show that a 2-category object in a 2-category $\cC$ can be given the structure of a dicategory object in $\cC$, and that a dicategory object in $\cC$ can be given the structure of a bicategory object in $\cC$.  These constructions make essential use of the fibrancy conditions in the definitions of 2-category object and dicategory object.  

\begin{prop} \label{prop-2catisdicat}
Given a 2-category object $C$ in a 2-category $\cC$, the triple $(C_0, C_1, C_2)$ of 0-cells, 1-cells, and 2-cells of $C$ also admits the structure of a dicategory object in $\cC$.
\end{prop}
\begin{proof}
\emph{0-data.}
The 0-cells, 1-cells, and 2-cells of the dicategory object $D$ under construction, and the source and target maps between these objects of cells, are the same as those for the given 2-category object $C$.

\emph{1-data.}
The morphisms [D1-1] through [D1-6] for $D$ are the morphisms [F1-1] through [F1-6] for $C$.  To construct the morphisms [D1-7] and [D1-8] we use the lifting property of the fibration $C_2 \xra{s \times t} C_1 \times_{C_0 \times C_0} C_1$.  Specifically, the morphism $C_1 \xra{\id \times \id} C_1 \times_{C_0 \times C_0} C_1$ has a canonical lift across $s \times t$ to the vertical identity morphism $i_v: C_1 \ra C_2$ provided by [F1-3].  That morphism $\id \times \id$ is `homotopic' to the morphism $C_1 \xra{m((i \circ s) \times \id) \times \id} C_1 \times_{C_0 \times C_0} C_1$ by the existence of the 2-isomorphism $\text{[F2-15]} \times \id: m((i \circ s) \times \id) \times \id \dra \id \times \id.$  This homotopy lifts to a homotopy (that is a 2-isomorphism), denoted $\eta_l$, of morphisms $C_1 \ra C_2$ from a lift $i_l$ of $m((i \circ s) \times \id) \times \id$ to the lift $i_v$ of $\id \times \id$.  The datum [D1-7] is defined to be that lift $i_l$.  Again, the homotopy lift is $\eta_l : i_l \dra i_v$, denoted graphically $\eta_l : \cb{\ings{d1stuff-item7.eps}} \dra \cb{\ings{1stuff-item3.eps}}$.  The datum [D1-8] is similar, defined via the homotopy lift $\eta_r : i_r \dra i_v$, that is $\eta_r : \cb{\ings{d1stuff-item8.eps}} \dra \cb{\ings{1stuff-item3.eps}}$, of the 2-isomorphism $\text{[F2-16]} \times \id : m(\id \times (i \circ t)) \times \id \dra \id \times \id$.  The lower left 2-cell identity and the lower right 2-cell identity are produced similarly and they indeed function as (triangle-identity-satisfying) vertical composition inverses for [D1-7] and [D1-8], as required.  % This is much more complicated than that last sentence suggests.  Indeed, producing the right structures requires certain inexplicit boundary conditions to be satisfied, and then checking the triangle axioms uses the uniqueness property of lifting.  See notebook.

\emph{2-data.}
The 2-isomorphisms [D2-1] through [D2-12] for $D$ are the 2-isomorphisms [F2-1] through [F2-12] for $C$.  The remaining 2-data are defined as the following composites:
\begin{itemize}
\item[D2-13]: 
\xymatrix@C+15pt{
\cb{\ing{d2stuff-item13-1.eps}}   
\ar@{=>}[r]^{\text{[F2-13]} \times \eta_l} &
\cb{\ing{2stuff-item2-1.eps}}
\ar@{=>}[r]^{\text{[F2-2]}} &
\cb{\ing{2stuff-item2-2.eps}}
\ar@{<=}[r]^{\text{[F2-1]}} &
\cb{\ing{2stuff-item1-1.eps}}
\ar@{<=}[r]^{\eta_l \times \id} &
\cb{\ing{d2stuff-item13-2.eps}}} \hspace{20pt}
\item[D2-14]: 
\xymatrix@C+15pt{
\cb{\ing{d2stuff-item14-1.eps}}   
\ar@{=>}[r]^{\text{[F2-14]} \times \eta_r} &
\cb{\ing{2stuff-item2-1.eps}}
\ar@{=>}[r]^{\text{[F2-2]}} &
\cb{\ing{2stuff-item2-2.eps}}
\ar@{<=}[r]^{\text{[F2-1]}} &
\cb{\ing{2stuff-item1-1.eps}}
\ar@{<=}[r]^{\eta_r \times \id} &
\cb{\ing{d2stuff-item14-2.eps}}} \hspace{20pt}
\item[D2-15]: 
\xymatrix@C+10pt{
\cb{\ing{d2stuff-item15-1.eps}}  
\ar@{=>}[r]^{\eta_l \circ i} &
\cb{\ing{bi-gps-3-2.eps}}
\ar@{<=}[r]^{\eta_r \circ i} &
 \cb{\ing{d2stuff-item15-2.eps}}} \hspace{20pt}
\item[D2-16]: 
\xymatrix@C+10pt{
\cb{\ing{d2stuff-item16-1.eps}}   
\ar@{=>}[r]^{\eta_l \times \id} &
\cb{\ing{2stuff-item4-1.eps}}
\ar@{=>}[r]^{\text{[F2-4]}} &
\cb{\ing{2stuff-item4-2.eps}}
\ar@{<=}[r]^{\eta_l \circ m} &
\cb{\ing{d2stuff-item16-2.eps}}} \hspace{20pt}
\item[D2-17]: 
\xymatrix@C+10pt{
\cb{\ing{d2stuff-item17-1.eps}}   
\ar@{=>}[r]^{\id \times \eta_r} &   
\cb{\ing{2stuff-item5-1.eps}}
\ar@{=>}[r]^{\text{[F2-5]}} &
\cb{\ing{2stuff-item5-2.eps}}
\ar@{<=}[r]^{\eta_r \circ m} &
\cb{\ing{d2stuff-item17-2.eps}}} \hspace{20pt}
\item[D2-18]: 
\xymatrix@C+10pt{
\cb{\ing{d2stuff-item18-1.eps}}   
\ar@{=>}[r]^{\eta_r \times \id} & 
\cb{\ing{2stuff-item4-1.eps}}
\ar@{=>}[r]^{\text{[F2-4]}} &
\cb{\ing{2stuff-item4-2.eps}}
\ar@{<=}[r]^{\text{[F2-5]}} &
\cb{\ing{2stuff-item5-1.eps}}
\ar@{<=}[r]^{\id \times \eta_l} &
\cb{\ing{d2stuff-item18-2.eps}}} \hspace{20pt}
\end{itemize}

\nid That this 2-isomorphism [D2-15] is compatible with the source map is a consequence of axiom [F3-23].  That [D2-16] and [D2-17] are compatible with the source map is a consequence of axiom [F3-24].  That [D2-18] is compatible with the source map is a consequence of axiom [F3-25].

\emph{3-axioms.}
The axioms [D3-1] through [D3-17] are the axioms [F3-1] through [F3-17].  The remainder of the axioms are directly verified by expanding the definitions of the morphisms involved and observing that the axiom is thereby reduced to known properties of the given 2-category object, as follows.
\begin{itemize}
\item[D3-18:] This axiom reduces, modulo two applications of axiom [F3-1], to axiom [F3-18].
\item[D3-19:] This axiom reduces, modulo an application of axiom [F3-3] and two applications of axiom [F3-2], to axiom [F3-19].
\item[D3-20:] This axiom reduces, modulo an application of axiom [F3-6] and two applications of axiom [F3-5], to axiom [F3-20].
\item[D3-21:] This axiom reduces, modulo an application of axiom [F3-6], to axiom [F3-21].
\item[D3-22:] This axiom reduces, modulo two applications of axiom [F3-5], to axiom [F3-22].
\item[D3-23:] This axiom reduces, modulo an application of axiom [F3-1], to axiom [F3-18].
\item[D3-24:] This axiom reduces to two applications of axiom [F3-18].
\item[D3-25:] This axiom reduces to axiom [F3-9].
\item[D3-26:] This axiom reduces to axioms [F3-9] and [F3-10]. \qedhere
\end{itemize}
\end{proof}
\nid We note that, in the above construction, all the conditions on the 2-category object were used in checking the dicategory object structure.

\begin{prop} \label{prop-dicatisbicat}
Given a dicategory object $D$ in a 2-category $\cC$, the triple $(C_0, C_1, C_2)$ of 0-cells, 1-cells, and 2-cells of $D$ also admits the structure of a bicategory object in $\cC$.
\end{prop}
\begin{proof}
\emph{0-data.}
The 0-cells, 1-cells, and 2-cells of the bicategory object $B$ under construction, and the source and target maps between these objects of cells, are the same as those for the given dicategory object $D$.

\emph{1-data.}
The morphisms [B1-1] through [B1-8] for $B$ are the morphisms [D1-1] through \mbox{[D1-8]} for $D$.  To construct the morphism [B1-9] we use the lifting property of the fibration $C_2 \xra{s \times t} C_1 \times_{C_0 \times C_0} C_1$.  Specifically, the morphism $C_1 \times_{C_0} C_1 \times_{C_0} C_1 \xra{[m \circ (m \times \id)] \times [m \circ (m \times \id)]} C_1 \times_{C_0 \times C_0} C_1$ has a canonical lift across $s \times t$ to the vertical identity morphism on the horizontal composition, that is the composite $C_1 \times_{C_0} C_1 \times_{C_0} C_1 \xra{m \circ (m \times \id)} C_1 \xra{i_v} C_2$.  The morphism $[m \circ (m \times \id)] \times [m \circ (m \times \id)]$ is `homotopic' to the morphism $C_1 \times_{C_0} C_1 \times_{C_0} C_1 \xra{[m \circ (m \times \id)] \times [m \circ (\id \times m)]} C_1 \times_{C_0 \times C_0} C_1$ by the existence of the 2-isomorphism $\id \times \text{[D2-12]}^{-1}: [m \circ (m \times \id)] \times [m \circ (\id \times m)] \dra [m \circ (m \times \id)] \times [m \circ (m \times \id)].$  This homotopy lifts to a homotopy (that is a 2-isomorphism), denoted $\alpha$, of morphisms $C_1 \times_{C_0} C_1 \times_{C_0} C_1 \ra C_2$ from a lift $a$ of $[m \circ (m \times \id)] \times [m \circ (\id \times m)]$ to the lift $i_v \circ (m \circ (m \times \id))$ of $[m \circ (m \times \id)] \times [m \circ (m \times \id)]$.  The datum [B1-9] is defined to be that lift $a$.  The homotopy lift is denoted graphically $\alpha : \cb{\ing{b1stuff-item9.eps}} \dra \cb{\ing{db1stuff-item1.eps}}$.  The vertical composition inverse of the 2-cell $a$ is produced similarly.

\emph{2-data.}
The 2-isomorphisms [B2-1] through [B2-8] and [B2-13] through [B2-15] for $B$ are the 2-isomorphisms [D2-1] through [D2-8] and [D2-13] through [D2-15] for $D$.  The remaining 2-data are defined as the following composites:
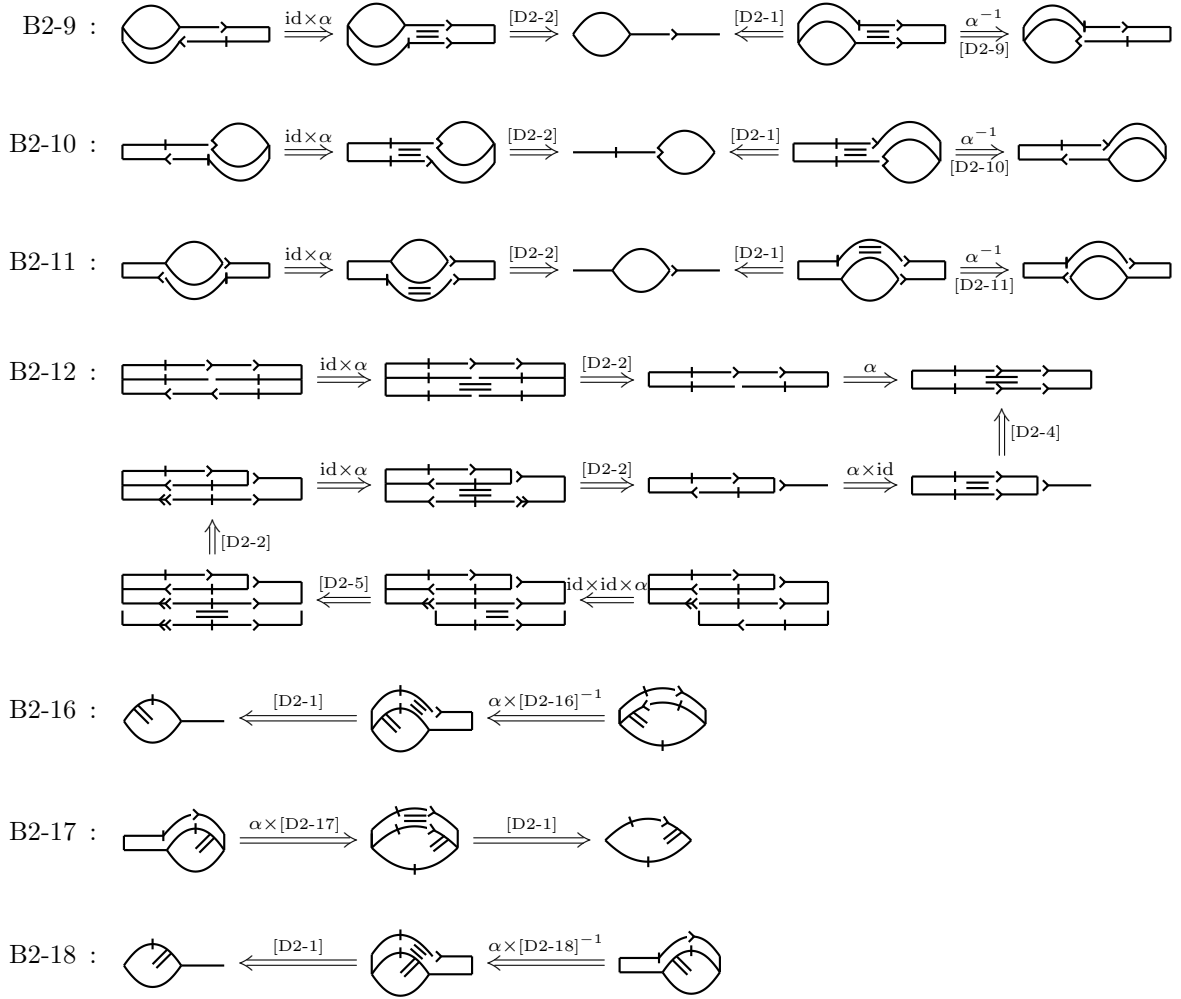
\begin{itemize}
\item[B2-9]: 
\xymatrix@C-6pt{
\cb{\ing{b2stuff-item9-1.eps}}   
\ar@{=>}[r]^{\id \times \alpha} &
\cb{\ing{db2stuff-item9-1.eps}}
\ar@{=>}[r]^{\text{\tiny [D2-2]}} &
\cb{\ing{2stuff-item9-1.eps}}
\ar@{<=}[r]^{\text{\tiny [D2-1]}} &
\cb{\ing{db2stuff-item9-2.eps}}
\ar@{=>}[r]^{\alpha^{-1}}_{\text{\tiny [D2-9]}} &
\cb{\ing{b2stuff-item9-2.eps}}} \hspace{15pt} \vspace{8pt}
\item[B2-10]: 
\xymatrix@C-6pt{
\cb{\ing{b2stuff-item10-1.eps}}   
\ar@{=>}[r]^{\id \times \alpha} &
\cb{\ing{db2stuff-item10-1.eps}}
\ar@{=>}[r]^{\text{\tiny [D2-2]}} &
\cb{\ing{2stuff-item10-1.eps}}
\ar@{<=}[r]^{\text{\tiny [D2-1]}} &
\cb{\ing{db2stuff-item10-2.eps}}
\ar@{=>}[r]^{\alpha^{-1}}_{\text{\tiny [D2-10]}} &
\cb{\ing{b2stuff-item10-2.eps}}} \hspace{15pt} \vspace{8pt}
\item[B2-11]: 
\xymatrix@C-6pt{
\cb{\ing{b2stuff-item11-1.eps}}   
\ar@{=>}[r]^{\id \times \alpha} &
\cb{\ing{db2stuff-item11-1.eps}}
\ar@{=>}[r]^{\text{\tiny [D2-2]}} &
\cb{\ing{2stuff-item11-1.eps}}
\ar@{<=}[r]^{\text{\tiny [D2-1]}} &
\cb{\ing{db2stuff-item11-2.eps}}
\ar@{=>}[r]^{\alpha^{-1}}_{\text{\tiny [D2-11]}} &
\cb{\ing{b2stuff-item11-2.eps}}} \hspace{15pt} \vspace{8pt}
\item[B2-12]: 
\xymatrix@C-4pt@R-12pt{
\cb{\ing{b2stuff-item12-1.eps}}
\ar@{=>}[r]^{\id \times \alpha} &
\cb{\ing{db2stuff-item12-1.eps}}
\ar@{=>}[r]^{\text{\tiny [D2-2]}} &
\cb{\ing{db2stuff-item12-2.eps}}
\ar@{=>}[r]^{\alpha} &
\cb{\ing{db2stuff-item12-3.eps}}
\ar@{<=}[d]^{\text{\tiny [D2-4]}}
\\
\cb{\ing{db2stuff-item12-7.eps}}
\ar@{=>}[r]^{\id \times \alpha} 
\ar@{<=}[d]^{\text{\tiny [D2-2]}}
&
\cb{\ing{db2stuff-item12-6.eps}}
\ar@{=>}[r]^{\text{\tiny [D2-2]}} 
&
\cb{\ing{db2stuff-item12-5.eps}}
\ar@{=>}[r]^{\alpha \times \id} 
&
\cb{\ing{db2stuff-item12-4.eps}}
\\
\cb{\ing{db2stuff-item12-8.eps}}
\ar@{<=}[r]^{\text{\tiny [D2-5]}} &
\cb{\ing{db2stuff-item12-9.eps}}
\ar@{<=}[r]^{\id \times \id \times \alpha} &
\cb{\ing{b2stuff-item12-2.eps}} &
} \hspace{15pt} \vspace{8pt}
\item[B2-16]: 
\xymatrix@C+20pt{
\cb{\ing{b2stuff-item16-1.eps}}  
\ar@{<=}[r]^{\text{\tiny [D2-1]}} &
\cb{\ing{db2stuff-item16-1.eps}}
\ar@{<=}[r]^-{\alpha \times \text{\tiny [D2-16]}^{-1}} &
 \cb{\ing{b2stuff-item16-2.eps}}} \hspace{15pt} \vspace{8pt}
\item[B2-17]: 
\xymatrix@C+20pt{
\cb{\ing{b2stuff-item17-1.eps}}  
\ar@{=>}[r]^{\alpha \times \text{\tiny [D2-17]}} &
\cb{\ing{db2stuff-item17-1.eps}}
\ar@{=>}[r]^{\text{\tiny [D2-1]}} &
 \cb{\ing{b2stuff-item17-2.eps}}} \hspace{15pt} \vspace{8pt}
\item[B2-18]: 
\xymatrix@C+20pt{
\cb{\ing{b2stuff-item18-1.eps}}  
\ar@{<=}[r]^{\text{\tiny [D2-1]}} &
\cb{\ing{db2stuff-item18-1.eps}}
\ar@{<=}[r]^-{\alpha \times \text{\tiny [D2-18]}^{-1}} &
 \cb{\ing{b2stuff-item18-2.eps}}} \hspace{15pt} \vspace{2pt}
\end{itemize}

\nid That the 2-isomorphism [B2-12] is compatible with the target map is a consequence of axiom [D3-17].

\emph{3-axioms.}
The axioms [B3-1] through [B3-8] and [B3-18] and [B3-19] are the corresponding dicategory axioms.  The axioms [B3-9] through [B3-14] and [B3-20] through [B3-24] are elementary consequences of the corresponding dicategory axioms.  The remainder of the axioms are directly verified by expanding the definitions of the morphisms involved and observing that the axiom is thereby reduced to known properties of the given dicategory object, as follows.  Note that in the following we omit mention of uses of the axioms [D3-1] through [D3-4].
\begin{itemize}
\item[B3-15] This axiom reduces, modulo two applications of axiom [D3-5] and an application of axiom [D3-6], to axiom [D3-15].
\item[B3-16] This axiom reduces, modulo four applications of axiom [D3-5], to axiom [D3-16].
\item[B3-17] By systematically connecting all the nodes of this axiom to an identity 2-cell on a left associated composite of five 1-cells, this axiom reduces, rather elaborately, to six instances of axiom [D3-5], two instances of axiom [D3-9], and an instance of axiom [D3-10].
\item[B3-25] This axiom reduces, modulo an application of axiom [D3-5] and an application of axiom [D3-18], to axiom [D3-25].
\item[B3-26] This axiom reduces, modulo two applications of axiom [D3-5], to axiom [D3-26]. \qedhere
\end{itemize}
\end{proof}
\nid Note that in this construction precisely all the conditions on the dicategory object were used in checking the bicategory object structure.

\subsection{Categories in $\tCAT$ are bicategories in $\CAT$} \label{sec-catin2catarebicatincat}

In Example~\ref{eg-2catnwk} in Section~\ref{sec-catin3cat}, we saw that 2-categories, together with strict functors, weak natural transformations, and modifications, form a Gray 3-category, denoted $\tCATnwk$, and in Example~\ref{eg-tc} we noted that tensor categories should form a category object in $\tCATnwk$.  In this section, given a category object in $\tCATnwk$ we construct an associated bicategory object in the 2-category of categories.

\begin{theorem} \label{thm-catin2catarebicatincat}
If $C=(\cz,\cn)=((\cz^0,\cz^1,\cz^2),(\cn^0,\cn^1,\cn^2))$ is a category object in $\tCATnwk$, then the triple $(\bz,\bn,\bt)=((\bz^0,\bz^1),(\bn^0,\bn^1),(\bt^0,\bt^1))$ forms the 0-, 1-, and 2-cells of a not-necessarily fibrant bicategory object $B$ in $\CAT$, where
\begin{align}
(\bz^0,\bz^1) &:= (\cz^0,\cz^1) \nn \\
(\bn^0,\bn^1) &:= (\cn^0,\cn^1) \nn \\
\bt^0 &:= \cz^0 \underset{\cz^1}{\tms} \cn^1 \underset{\cz^1}{\tms} \cz^0 \nn \\
\bt^1 &:= \cn^2 \underset{(\cn^1 \tms \cn^1 \tms \cz^2 \times \cz^2)}{\tms} \Big( \Big( \bt^0 \underset{\cn^0}{\tms} \cn^1 \Big) \underset{(\cn^0 \tms \cn^0 \tms \cz^1 \tms \cz^1)}{\tms} \Big( \cn^1 \underset{\cn^0}{\tms} \bt^0 \Big) \Big) \nn
\end{align} 
\end{theorem}

\nid By ``not-necessarily fibrant bicategory object" we mean not only to drop the fibration condition on the morphisms $B_1 \ra B_0 \times B_0$ and $B_2 \ra B_1 \times_{B_0 \times B_0} B_1$ but also implicitly to drop the condition that $B_0$ and $B_1$ be groupoids.

In more detail the expressions in the theorem are as follows. The $j$-morphisms in the 2-category $C_i$ are denoted $C_i^j$, and the $j$-morphisms in the category $B_i$ are denoted $B_i^j$.  Let $s,t: \cn \ra \cz$ denote the source and target 2-functors, let $S,T: C_i^j \ra C_i^{j-1}$ denote the source and target of morphisms in the 2-categories $C_i$, let $I: C_i^j \ra C_i^{j+1}$ denote the identity of morphisms in the 2-categories $C_i$, and let $M: C_i^j \times_{C_i^{j-1}} C_i^j \ra C_i^j$ denote composition of $j$-morphisms in the 2-category $C_i$.  In the fiber product defining $\bt^0$, the four maps are $I: \cz^0 \ra \cz^1$, $s: \cn^1 \ra \cz^1$, $t: \cn^1 \ra \cz^1$, and $I: \cz^0 \ra \cz^1$.  In the fiber product defining $\bt^1$, the various maps are, from inside out, $T: \bt^0 \ra \cn^0$ and $S: \cn^1 \ra \cn^0$, and $T: \cn^1 \ra \cn^0$ and $S:\bt^0 \ra \cn^0$, then 
\vspace{8pt}
\begin{align}
& (S \tms T \tms s \tms t) \circ M: \bt^0 \tms_{\cn^0} \cn^1 \ra \cn^0 \tms \cn^0 \tms \cz^1 \tms \cz^1 \text{ and} \nn \\
& (S \tms T \tms s \tms t) \circ M: \cn^1 \tms_{\cn^0} \bt^0 \ra \cn^0 \tms \cn^0 \tms \cz^1 \tms \cz^1, \nn \\[-3pt]
%\end{align} 
\intertext{and finally}
%\begin{align}
& S \tms T \tms s \tms t : \cn^2 \ra \cn^1 \tms \cn^1 \tms \cz^2 \times \cz^2, \nn \\
& M \tms M : \Big( \bt^0 \underset{\cn^0}{\tms} \cn^1 \Big) \underset{(\cn^0 \tms \cn^0 \tms \cz^1 \tms \cz^1)}{\tms} \Big( \cn^1 \underset{\cn^0}{\tms} \bt^0 \Big) \ra \cn^1 \times \cn^1, \text{ and} \nn \\
& \! \Big( \bt^0 \underset{\cn^0}{\tms} \cn^1 \Big) \underset{(\cn^0 \tms \cn^0 \tms \cz^1 \tms \cz^1)}{\tms} \Big( \cn^1 \underset{\cn^0}{\tms} \bt^0 \Big) \xra{\text{proj}} \bt^0 \underset{\cn^0}{\tms} \cn^1 \xra{I s M \tms I t M} \cz^2 \tms \cz^2. \nn
\end{align}

\begin{proof}
We describe the data of a bicategory object $B$ on the triple of categories $(\bz,\bn,\bt)$, and then verify the axioms.  The source and target functors $s,t: \bn \ra \bz$ are induced by the source and target 2-functors $s,t: \cn \ra \cz$.  The source, respectively target, functor $\bt \ra \bn$ is induced on objects by the map $S: \cn^1 \ra \cn^0$, respectively $T: \cn^1 \ra \cn^0$, and on morphisms by projection to the second, respectively first, factor of $\cn^1$ in the expression $\left( \bt^0 \tms_{\cn^0} \cn^1 \right) \tms_{(\cn^0 \tms \cn^0 \tms \cz^0 \tms \cz^0)} \left( \cn^1 \tms_{\cn^0} \bt^0 \right)$ inside the definition of $\bt^1$.  By construction, these functors satisfy $st = ss$ and $tt = ts$.

The 1-data of the bicategory object $B$ are as follows.  
\begin{description}
\item[B1-1] The 1-cell identity is the restriction of the 1-cell identity [C1-1] to objects and morphisms.
\item[B1-2] The horizontal composition is the restriction of the horizontal composition [C1-2] to objects and morphisms.
\item[B1-3] The 2-cell identity is induced by the identities $I: \cn^0 \ra \cn^1$ and $I: \cn^1 \ra \cn^2$ of the 2-category $\cn$.
\item[B1-4] The vertical composition is induced by the compositions $M: \cn^1 \times_{\cn^0} \cn^1 \ra \cn^1$ and $M: \cn^2 \times_{\cn^1} \cn^2 \ra \cn^2$ of the 2-category $\cn$.
\item[B1-5,6] The right whisker functor $\bt \tms_{\bz} \bn \ra \bt$ is defined on objects by the horizontal composition 2-functor [C1-2] applied in the first variable to morphisms and in the second variable to identity morphisms, and is defined on morphisms by the same 2-functor applied in the first variable to 2-morphisms and in the second variable to identity 2-morphisms.  The left whisker is similar.
\item[B1-7,8] The left 2-cell identity functor $\bn \ra \bt$ is defined by the left identity transformation [C2-1] applied to objects and morphisms.  The right 2-cell identity is similarly defined by the right identity transformation [C2-2].
\item[B1-9] The associator 2-cell functor is defined by the associator transformation [C2-3] applied to objects and morphisms.
\end{description}
\vspace{8pt}

The 2-data of the bicategory object $B$ are as follows. 
\begin{description}
\item[B2-1,2] Data [B2-1] and [B2-2] are identity isomorphisms, because the morphism and 2-morphism identities in $\cn$ are strict.
\item[B2-3] Datum [B2-3] is an identity isomorphism, because composition of morphisms and 2-morphisms in $\cn$ is associative.
\item[B2-4,5] Data [B2-4] and [B2-5] are identity isomorphisms, because the horizontal composition 2-functor [C1-2] takes identities to identities.
\item[B2-6,7] Data [B2-6] and [B2-7] are identity isomorphisms, because the horizontal composition 2-functor [C1-2] takes composites of morphisms to composites of morphisms, and because the identity 1-morphisms in $\cn$ are strict.
\item[B2-8] Datum [B2-8] is an identity isomorphism, because again the horizontal composition 2-functor [C1-2] takes composites of morphisms to composites of morphisms, and because the identity 1-morphisms in $\cn$ are strict.
\item[B2-9,10,11] Datum [B2-9] is the associator transformation [C2-3] applied to the image of $\bt^0 \times_{\cz^0} \cn^0 \times_{\cz^0} \cn^0 \xra{\text{inc} \times I \times I} \cn^1 \times_{\cz^0} \cn^1 \times_{\cz^0} \times \cn^1$.  The datum [B2-10], respectively [B2-11], is similar with the associator transformation applied after identities in the first two, respectively first and last, factors.
\item[B2-12] Datum [B2-12] is the modification [C3-5].
\item[B2-13,14] Datum [B2-13] is the left identity transformation [C2-1] applied to the image of the inclusion $\bt^0 \ra \cn^1$.  Similarly datum [B2-14] is induced by the right identity transformation [C2-2].
\item[B2-15] Datum [B2-15] is the modification [C3-1].
\item[B2-16] Datum [B2-16] is the modification [C3-2].
\item[B2-17] Datum [B2-17] is the modification [C3-3].
\item[B2-18] Datum [B2-18] is the modification [C3-4].
\end{description}
\vspace{8pt}

The 3-axioms of the bicategory object $B$ are satisfied, as follows.
\begin{description}
\item[B3-1 to B3-8] The first eight axioms [B3-1] through [B3-8] are trivially satisfied, because all the 2-morphisms involved are identities.
\item[B3-9] Unpacked, axiom [B3-9] asserts that the 2-morphism [B2-9] is the identity when applied to an identity 2-cell.  This is the case because the transformation [C2-3] takes identity morphisms to identity morphisms.
\item[B3-10] Unpacked, axiom [B3-10] asserts that the 2-morphism [B2-11] is the identity when applied to an identity 2-cell.  This is the case again because the transformation [C2-3] takes identity morphisms to identity morphisms.
\item[B3-11 to B3-14] Axiom [B3-11] is satisfied because the transformation [C2-3] takes composites of morphisms to composites of morphisms; in this case the composition property is applied with two nontrivial morphisms in the first factor of $\cn$.  Axiom [B3-12] is satisfied for the same reason, applied to nontrivial morphisms in the second factor of $\cn$.  Axiom [B3-13], respectively [B3-14], is satisfied for again the same reason, applied to one nontrivial morphism in each of the first and second, respectively first and third, factors of $\cn$.
\item[B3-15,16] Axiom [B3-15] is the condition that [C3-5] is a modification, applied with a non-identity morphism only in the first factor of $\cn$.  Similarly axiom [B3-16] is the modification condition of [C3-5] for a non-identity morphism in the second factor of $\cn$.
\item[B3-17] Axiom [B3-17] is axiom [C4-5].
\item[B3-18] Axiom [B3-18] asserts that the 2-morphism [B2-14] is the identity when applied to an identity 2-cell.  This is the case because the transformation [C2-1] takes identity morphisms to identity morphisms.
\item[B3-19] Axiom [B3-19] is satisfied because the transformation [C2-1] takes composites of morphisms to composites of morphisms.
\item[B3-20] Axiom [B3-20] is the condition that [C3-4] is a modification, applied with a non-identity morphism only in the first factor of $\cn$.
\item[B3-21] Axiom [B3-21] is the condition that [C3-3] is a modification, applied with a non-identity morphism only in the first factor of $\cn$.
\item[B3-22] Axiom [B3-22] is the condition that [C3-2] is a modification, applied with a non-identity morphism only in the first factor of $\cn$.
\item[B3-23] Axiom [B3-23] is axiom [C4-1].
\item[B3-24] Axiom [B3-24] is axiom [C4-2].
\item[B3-25] Axiom [B3-25] is axiom [C4-3].
\item[B3-26] Axiom [B3-26] is axiom [C4-4]. \qedhere
\end{description}
\end{proof}

Recall from Example~\ref{eg-tc} that we say a category object $C$ in $\tCATnwk$ is associatively strict if the associator transformation [C2-3] is a strict, rather than weak, natural transformation.  Such an associator transformation induces not an associator 2-cell [B1-9] but in fact an associator transformation [F2-12]; however, this transformation may fail to be a 2-isomorphism unless the associator transformation [C2-3] has the further property that it is strictly, not just weakly, invertible.  We refer to a category object in $\tCATnwk$ whose associator transformation [C2-3] is a strictly invertible strict natural transformation as a ``strictly associatively strict" (sas) category object.  For such a category object, the above construction produces a dicategory object rather than a bicategory object:

\begin{cor} \label{cor-ascatin2cataredicatincat}
If $C=(\cz,\cn)$ is a strictly associatively strict category object in $\tCATnwk$, then the triple $(\bz,\bn,\bt)$ defined in Theorem~\ref{thm-catin2catarebicatincat} forms the 0-, 1-, and 2-cells of a not-necessarily fibrant dicategory object in $\CAT$.
\end{cor}

If the category object $C$ is in fact a category object in $\tCAT$, rather than in $\tCATnwk$, then the left and right identity transformations [C2-1] and [C2-2] are strict natural transformation of 2-categories and therefore induce, not only left and right 2-cell identities [D1-7] and [D1-8], but in fact left and right identity 2-transformations [F2-15] and [F2-16], which, though, may fail to be 2-isomorphisms.  We refer to the category object as ``strict" if the identity transformations [C2-1] and [C2-2] and the associator transformation [C2-3] are all strictly invertible.  In this case, the above construction produces a 2-category object rather than a di- or bicategory object:

\begin{cor} \label{cor-catin2catare2catincat}
If $C=(\cz,\cn)$ is a strict category object in $\tCAT$, then the triple $(\bz,\bn,\bt)$ defined in Theorem~\ref{thm-catin2catarebicatincat} forms the 0-, 1-, and 2-cells of a not-necessarily fibrant 2-category object in $\CAT$.
\end{cor}

\section{Tricategories}\label{app-tricat}
In this final section, we investigate the notion of a tricategory, and how it relates to the other notions of ``3-level categories'' that we introduced in the previous sections.  

\subsection{The definition} \label{sec-tricatdef}

For simplicity we restrict attention to small tricategories.  In listing the various structures in the definition, in each case we first give a visual depiction of the item, then describe the item algebraically and verbally. 

\begin{definition} \label{def-tricat}
A tricategory $T$ consists of the following two collections of data, subject to axioms as follows:
\begin{description}

\item[0-data] There are four sets: \vspace{2pt}
\begin{itemize}
\item[] \cb{\ing{0stuff-item0.eps}} --- the set $T_0$ of 0-cells. \vspace{6pt}
\item[] \cb{\ing{0stuff-item1.eps}} --- the set $T_1$ of 1-cells. \vspace{6pt}
\item[] \cb{\ing{0stuff-item2.eps}} --- the set $T_2$ of 2-cells. \vspace{6pt}
\item[] \cb{\ing{0stuff-item3.eps}} --- the set $T_3$ of 3-cells. \vspace{6pt}
\end{itemize}
Moreover there are source and target maps $s,t:T_1 \ra T_0$, $s,t:T_2 \ra T_1$, and $s,t:T_3 \ra T_2$, such that $st = ss$ and $tt = ts$. \vspace{3pt}

\item[1-data] There are thirty-three maps of sets in three collections: \vspace{3pt}

\begin{description}

\item[1-cell target] There are two maps with target $T_1$: \vspace{3pt}
\begin{itemize}
\item[T1-1] \cb{\ing{1stuff-item1.eps}} --- $i_x: T_0 \ra T_1$ --- horizontal identity. \vspace{8pt}
\item[T1-2] \cb{\ing{1stuff-item2.eps}} --- $m_x: T_1 \times_{T_0} T_1 \ra T_1$ --- horizontal composition.
\end{itemize} \vspace{10pt}

\item[2-cell target] There are seven maps with target $T_2$: \vspace{3pt}
\begin{itemize}
\item[T1-3] \cb{\ing{1stuff-item3.eps}} --- $i_y: T_1 \ra T_2$ --- vertical identity. \vspace{8pt}
\item[T1-4] \cb{\ing{1stuff-item4.eps}} --- $m_y: T_2 \times_{T_1} T_2 \ra T_2$ --- vertical composition. \vspace{8pt}
\item[T1-5] \cb{\ing{1stuff-item5.eps}} --- $w_r: T_2 \times_{T_0} T_1 \ra T_2$ --- right whisker. \vspace{8pt}
\item[T1-6] \cb{\ing{1stuff-item6.eps}} --- $w_l: T_1 \times_{T_0} T_2 \ra T_2$ --- left whisker. \vspace{8pt}
\item[T1-7] \cb{\ing{d1stuff-item7.eps}} --- $i_l: T_1 \ra T_2$ --- left identity. \vspace{8pt}
\item[T1-8] \cb{\ing{d1stuff-item8.eps}} --- $i_r: T_1 \ra T_2$ --- right identity. \vspace{8pt}
\item[T1-9] \cb{\ing{b1stuff-item9.eps}} --- $a_x: T_1 \times_{T_0} T_1 \times_{T_0} T_1 \ra T_2$ --- horizontal associator.
\end{itemize} \vspace{10pt}

\item[3-cell target] There are twenty-four maps with target $T_3$.  The first eighteen are maps $\phi: T \ra T_3$, where $T$ is a fiber product involving only $T_0$, $T_1$, and $T_2$; these we indicate by depicting the source and target of the 3-cell $\phi(a)$ in terms of the source $a \in T$ of the map $\phi$.  
The remaining six maps we indicate directly by drawing an image of the target 3-cell. \vspace{5pt}
\begin{itemize}
\item[T1-10]: \xymatrix@C-12pt{\cb{\ing{2stuff-item1-1.eps}}   \ar@{-}[r] &   \cb{\ing{2stuff-item1-2.eps}}} 
--- $i_t: T_2 \ra T_3$ --- top identity. \vspace{4pt}
\item[T1-11]: \xymatrix@C-12pt{\cb{\ing{2stuff-item2-1.eps}}   \ar@{-}[r] &   \cb{\ing{2stuff-item2-2.eps}}} 
--- $i_b: T_2 \ra T_3$ --- bottom identity. \vspace{4pt}
\item[T1-12]: \xymatrix@C-12pt{\cb{\ing{2stuff-item3-1.eps}}   \ar@{-}[r] &   \cb{\ing{2stuff-item3-2.eps}}} 
--- $a_y: T_2 \times_{T_1} T_2 \times_{T_1} T_2 \ra T_3$ --- vertical associator. \vspace{4pt}
\item[T1-13]: \xymatrix@C-12pt{\cb{\ing{2stuff-item4-1.eps}}   \ar@{-}[r] &   \cb{\ing{2stuff-item4-2.eps}}} 
--- $e_{vir}: T_1 \times_{T_0} T_1 \ra T_3$ --- vert id expansion right. \vspace{4pt}
\item[T1-14]: \xymatrix@C-12pt{\cb{\ing{2stuff-item5-1.eps}}   \ar@{-}[r] &   \cb{\ing{2stuff-item5-2.eps}}} 
--- $e_{vil}: T_1 \times_{T_0} T_1 \ra T_3$ --- vert id expansion left. \vspace{4pt}
\item[T1-15]: \xymatrix@C-12pt{\cb{\ing{2stuff-item6-1.eps}}   \ar@{-}[r] &   \cb{\ing{2stuff-item6-2.eps}}} 
--- $dw_{r}: (T_2 \times_{T_1} T_2) \times_{T_0} T_1 \ra T_3$ --- dewhisker right. \vspace{4pt}
\item[T1-16]: \xymatrix@C-12pt{\cb{\ing{2stuff-item7-1.eps}}   \ar@{-}[r] &   \cb{\ing{2stuff-item7-2.eps}}} 
--- $dw_{l}: T_1 \times_{T_0} (T_2 \times_{T_1} T_2) \ra T_3$ --- dewhisker left. \vspace{4pt}
\item[T1-17]: \xymatrix@C-12pt{\cb{\ing{2stuff-item8-1.eps}}   \ar@{-}[r] &   \cb{\ing{2stuff-item8-2.eps}}} 
--- $sw: T_2 \times_{T_0} T_2 \ra T_3$ --- switch. \vspace{4pt}
\item[T1-18]: \xymatrix@C-12pt{\cb{\ing{b2stuff-item9-1.eps}}   \ar@{-}[r] &   \cb{\ing{b2stuff-item9-2.eps}}} 
--- $p_l: T_2 \times_{T_0} T_1 \times_{T_0} T_1 \ra T_3$ --- pass left. \vspace{4pt}
\item[T1-19]: \xymatrix@C-12pt{\cb{\ing{b2stuff-item10-1.eps}}   \ar@{-}[r] &   \cb{\ing{b2stuff-item10-2.eps}}}
--- $p_r: T_1 \times_{T_0} T_1 \times_{T_0} T_2 \ra T_3$ --- pass right. \vspace{4pt}
\item[T1-20]: \xymatrix@C-12pt{\cb{\ing{b2stuff-item11-1.eps}}   \ar@{-}[r] &   \cb{\ing{b2stuff-item11-2.eps}}}
--- $p_m: T_1 \times_{T_0} T_2 \times_{T_0} T_1 \ra T_3$ --- pass middle. \vspace{4pt}
\item[T1-21]: \xymatrix@C-12pt{\cb{\ing{b2stuff-item12-1.eps}}   \ar@{-}[r] &   \cb{\ing{b2stuff-item12-2.eps}}}
--- $pn: T_1 \times_{T_0} T_1 \times_{T_0} T_1 \times_{T_0} T_1 \ra T_3$ \hspace*{\fill} --- pentagonator. \vspace{4pt}
\item[T1-22]: \xymatrix@C-12pt{\cb{\ing{d2stuff-item13-1.eps}}   \ar@{-}[r] &   \cb{\ing{d2stuff-item13-2.eps}}}
--- $pi_l: T_2 \ra T_3$ --- left identity pass. \vspace{4pt}
\item[T1-23]: \xymatrix@C-12pt{\cb{\ing{d2stuff-item14-1.eps}}   \ar@{-}[r] &   \cb{\ing{d2stuff-item14-2.eps}}}
--- $pi_r: T_2 \ra T_3$ --- right identity pass. \vspace{4pt}
\item[T1-24]: \xymatrix@C-12pt{\cb{\ing{d2stuff-item15-1.eps}}   \ar@{-}[r] &   \cb{\ing{d2stuff-item15-2.eps}}}
--- $sp: T_0 \ra T_3$ --- swap. \vspace{4pt}
\item[T1-25]: \xymatrix@C-12pt{\cb{\ing{b2stuff-item16-1.eps}}   \ar@{-}[r] &   \cb{\ing{b2stuff-item16-2.eps}}}
--- $e_{li}: T_1 \times_{T_0} T_1 \ra T_3$ --- left identity expansion. \vspace{4pt}
\item[T1-26]: \xymatrix@C-12pt{\cb{\ing{b2stuff-item17-1.eps}}   \ar@{-}[r] &   \cb{\ing{b2stuff-item17-2.eps}}}
--- $e_{ri}: T_1 \times_{T_0} T_1 \ra T_3$ --- right identity expansion. \vspace{4pt}
\item[T1-27]: \xymatrix@C-12pt{\cb{\ing{b2stuff-item18-1.eps}}   \ar@{-}[r] &   \cb{\ing{b2stuff-item18-2.eps}}}
--- $fp: T_1 \times_{T_0} T_1 \ra T_3$ --- flip. \vspace{12pt}
\item[T1-28] \cb{\ing{t1stuff-item28.eps}} --- $i_z: T_2 \ra T_3$ --- spatial identity. 
\vspace{12pt}
\item[T1-29] \cb{\ing{t1stuff-item29.eps}} --- $m_z: T_3 \times_{T_2} T_3 \ra T_3$ --- spatial composition.
\vspace{12pt}
\item[T1-30] \cb{\ing{t1stuff-item30.eps}} --- $f_b: T_3 \times_{T_1} T_2 \ra T_3$ --- bottom fin.
\vspace{12pt}
\item[T1-31] \cb{\ing{t1stuff-item31.eps}} --- $f_t: T_2 \times_{T_1} T_3 \ra T_3$ --- top fin.
\vspace{12pt}
\item[T1-32] \cb{\ing{t1stuff-item32.eps}} --- $h_r: T_3 \times_{T_0} T_1 \ra T_3$ --- right 3-cell whisker.
\vspace{12pt}
\item[T1-33] \cb{\ing{t1stuff-item33.eps}} --- $h_l: T_1 \times_{T_0} T_3 \ra T_3$ --- left 3-cell whisker.
\end{itemize} \vspace{10pt}

\item[Inverses]
A 3-cell $c \in T_3$ is called invertible if there exists a 3-cell $c^{-1} \in T_3$ such that $m_z(c \times c^{-1}) = i_z(s(c))$ and $m_z(c^{-1} \times c) = i_z(t(c))$.  The 1-data [T1-10] through [T1-27] are required to take values in invertible 3-cells. \\
\hspace*{9pt} A 2-cell $c \in T_2$ is called invertible if there exists a 2-cell $c^{-1} \in T_2$, an invertible 3-cell $C \in T_3$ with source $m_y(c \times c^{-1})$ and target $i_y(s(c))$, and an invertible 3-cell $D \in T_3$ with source $m_y(c^{-1} \times c)$ and target $i_y(t(c))$, such that the following two equations are satisfied:
$m_z(f_b(C \times c) \times i_t(c)) = m_z(m_z(a_y(c \times c^{-1} \times c) \times f_t(c \times D)) \times i_b(c))$,
$m_z(f_b(D \times c^{-1}) \times i_t(c^{-1})) = m_z(m_z(a_y(c^{-1} \times c \times c^{-1}) \times f_t(c^{-1} \times C)) \times i_b(c^{-1}))$.  
These last equations are versions of the usual ``triangle identities", though they are in fact pentagons in this context.  The 1-data [T1-7], [T1-8], and [T1-9] are required to take values in invertible 2-cells.

\end{description} \vspace{3pt}

\item[2-axioms]
The above data are subject to the following forty-one axioms, together with reflection variant axioms abbreviated in the lists [T2-r] and [T2-R].  In the first twenty-six axioms, the condition is that the 3-cell obtained by composing all the edges of the diagram is the spatial identity; when the axiom is listed in a single column, the last edge implicitly loops back to the top picture.  In the last fifteen axioms, the indicated equation of 3-cells is satisfied; there, composition of 3-cells refers to spatial composition.

\end{description}
\end{definition}

%\newpage
\newgeometry{margin=1.5cm}

\begin{tikzpicture}[yscale=.05ex,xscale=.4ex]
\draw (1,0) node {
\begin{tikzpicture}[scale=\tikztextscale]
\draw (0,0) node {\pgftext{T2-1:}};
\end{tikzpicture}
};
\draw (1,-1) node[anchor=north] {
\def\axnum{1}
\begin{tikzpicture}[xscale=\epspicscalex,yscale=\epspicscaleyb]
\axgrid{\axnum};
\foreach \x in {1,...,2}
\draw (\x) node[littlenode-eps,anchor=center]
	{\innertikzeps{\ax{\axnum}-{\x};}};
\end{tikzpicture}
};
\draw (2,0) node {
\begin{tikzpicture}[scale=\tikztextscale]
\draw (0,0) node {\pgftext{T2-2:}};
\end{tikzpicture}
};
\draw (2,-1) node[anchor=north] {
\def\axnum{2}
\begin{tikzpicture}[xscale=\epspicscalex,yscale=\epspicscaleyc]
\axgrid{\axnum};
\foreach \x in {1,...,3}%,0}
\draw (\x) node[littlenode-eps,anchor=center]
	{\innertikzeps{\ax{\axnum}-{\x};}};
\end{tikzpicture}
};
\draw (3,0) node {
\begin{tikzpicture}[scale=\tikztextscale]
\draw (0,0) node {\pgftext{T2-3:}};
\end{tikzpicture}
};
\draw (3,-1) node[anchor=north] {
\def\axnum{3}
\begin{tikzpicture}[xscale=\epspicscalex,yscale=\epspicscaleyc]
\axgrid{\axnum};
\foreach \x in {1,...,3}
\draw (\x) node[littlenode-eps,anchor=center]
	{\innertikzeps{\ax{\axnum}-{\x};}};
\end{tikzpicture}
};
\draw (4,0) node {
\begin{tikzpicture}[scale=\tikztextscale]
\draw (0,0) node {\pgftext{T2-4:}};
\end{tikzpicture}
};
\draw (4,-1) node[anchor=north] {
\def\axnum{4}
\begin{tikzpicture}[xscale=\epspicscalex,yscale=\epspicscaleyd]
\axgrid{\axnum};
\foreach \x in {1,...,5}
\draw (\x) node[littlenode-eps,anchor=center]
	{\innertikzeps{\ax{\axnum}-{\x};}};
\end{tikzpicture}
};
\draw (5,0) node {
\begin{tikzpicture}[scale=\tikztextscale]
\draw (0,0) node {\pgftext{T2-5:}};
\end{tikzpicture}
};
\draw (5,-1) node[anchor=north] {
\def\axnum{5}
\begin{tikzpicture}[xscale=\epspicscalex,yscale=\epspicscaleyc]
\axgrid{\axnum};
\foreach \x in {1,...,4}%,0,-1,-2}
\draw (\x) node[littlenode-eps,anchor=center]
	{\innertikzeps{\ax{\axnum}-{\x};}};
\end{tikzpicture}
};
\draw (6,0) node {
\begin{tikzpicture}[scale=\tikztextscale]
\draw (0,0) node {\pgftext{T2-6:}};
\end{tikzpicture}
};
\draw (6,-1) node[anchor=north] {
\def\axnum{6}
\begin{tikzpicture}[xscale=\epspicscalex,yscale=\epspicscaleyc]
\axgrid{\axnum};
\foreach \x in {1,...,5}%,0}
\draw (\x) node[littlenode-eps,anchor=center]
	{\innertikzeps{\ax{\axnum}-{\x};}};
\end{tikzpicture}
};
\draw (7,0) node {
\begin{tikzpicture}[scale=\tikztextscale]
\draw (0,0) node {\pgftext{T2-7:}};
\end{tikzpicture}
};
\draw (7,-1) node[anchor=north] {
\def\axnum{7}
\begin{tikzpicture}[xscale=\epspicscalex,yscale=\epspicscaleyc]
\axgrid{\axnum};
\foreach \x in {1,...,6}%,0}
\draw (\x) node[littlenode-eps,anchor=center]
	{\innertikzeps{\ax{\axnum}-{\x};}};
\end{tikzpicture}
};
\draw (8,0) node {
\begin{tikzpicture}[scale=\tikztextscale]
\draw (0,0) node {\pgftext{T2-8:}};
\end{tikzpicture}
};
\draw (8,-1) node[anchor=north] {
\def\axnum{8}
\begin{tikzpicture}[xscale=\epspicscalex,yscale=\epspicscaleyc]
\axgrid{\axnum};
\foreach \x in {1,...,8}%,0}
\draw (\x) node[littlenode-eps,anchor=center]
	{\innertikzeps{\ax{\axnum}-{\x};}};
\end{tikzpicture}
};
\draw (1,-50) node {
\begin{tikzpicture}[scale=\tikztextscale]
\draw (0,0) node {\pgftext{T2-9:}};
\end{tikzpicture}
};
\draw (1,-51) node[anchor=north] {
\def\axnum{9}
\begin{tikzpicture}[xscale=\epspicscalex,yscale=\epspicscaleyc]
\axgrid{\axnum};
\foreach \x in {1,...,6}%,0}
\draw (\x) node[littlenode-eps,anchor=center]
	{\innertikztikz{\ax{\axnum}-{\x};}};
\end{tikzpicture}
};
\draw (2,-50) node {
\begin{tikzpicture}[scale=\tikztextscale]
\draw (0,0) node {\pgftext{T2-10:}};
\end{tikzpicture}
};
\draw (2,-51) node[anchor=north] {
\def\axnum{10}
\begin{tikzpicture}[xscale=\epspicscalex,yscale=\epspicscaleyc]
\axgrid{\axnum};
\foreach \x in {1,...,7}
\draw (\x) node[littlenode-eps,anchor=center]
	{\innertikztikz{\ax{\axnum}-{\x};}};
\end{tikzpicture}
};
\draw (3,-50) node {
\begin{tikzpicture}[scale=\tikztextscale]
\draw (0,0) node {\pgftext{T2-11:}};
\end{tikzpicture}
};
\draw (3,-51) node[anchor=north] {
\def\axnum{11}
\begin{tikzpicture}[xscale=\epspicscalex,yscale=\epspicscaleyc]
\axgrid{\axnum};
\foreach \x in {1,...,9}%,0}
\draw (\x) node[littlenode-eps,anchor=center]
	{\innertikztikz{\ax{\axnum}-{\x};}};
\end{tikzpicture}
};
\draw (4,-50) node {
\begin{tikzpicture}[scale=\tikztextscale]
\draw (0,0) node {\pgftext{T2-12:}};
\end{tikzpicture}
};
\draw (4,-51) node[anchor=north] {
\def\axnum{12}
\begin{tikzpicture}[xscale=\epspicscalex,yscale=\epspicscaleyc]
\axgrid{\axnum};
\foreach \x in {1,...,10}
\draw (\x) node[littlenode-eps,anchor=center]
	{\innertikztikz{\ax{\axnum}-{\x};}};
\end{tikzpicture}
};
\draw (5.75,-50) node {
\begin{tikzpicture}[scale=\tikztextscale]
\draw (0,0) node {\pgftext{T2-13:}};
\end{tikzpicture}
};
\draw (5.75,-51) node[anchor=north] {
\def\axnum{13}
\begin{tikzpicture}[xscale=\epspicscaledoubleplusx,yscale=\epspicscaleyc]
\axgrid{\axnum};
\foreach \x in {1,...,14}%, 0}
\draw (\x) node[littlenode-eps,anchor=center]
	{\innertikztikz{\ax{\axnum}-{\x};}};
\end{tikzpicture}
};
\draw (8,-50) node {
\begin{tikzpicture}[scale=\tikztextscale]
\draw (0,0) node {\pgftext{T2-14:}};
\end{tikzpicture}
};
\draw (8,-51) node[anchor=north] {
\def\axnum{14}
\begin{tikzpicture}[xscale=\epspicscaledoubleplusx,yscale=\epspicscaleyc]
\axgrid{\axnum};
\foreach \x in {1,...,12}%, 0}
\draw (\x) node[littlenode-eps,anchor=center]
	{\innertikztikz{\ax{\axnum}-{\x};}};
\end{tikzpicture}
};

\draw (10,0) node {
\begin{tikzpicture}[scale=\tikztextscale]
\draw (0,0) node {\pgftext{T2-r:}};
\end{tikzpicture}
};
\draw (10,-1) node[anchor=north] {
\def\axnum{27}
\begin{tikzpicture}[xscale=\epspicscalex,yscale=\epspicscaleyc]
\axgrid{\axnum};
\foreach \x in {2,51,52,53,6,7,8,15,16,18,19,25,26}
\draw (\x) node[littlenode-eps,anchor=center]
	{\innertikzeps{\ax{\axnum}-{\x};}};
\foreach \x in {9,11,13,20,21,22,23}
\draw (\x) node[littlenode-eps,anchor=center]
	{\innertikztikz{\ax{\axnum}-{\x};}};
\end{tikzpicture}
};
\end{tikzpicture}

\begin{tikzpicture}[yscale=.05ex,xscale=.4ex]
\draw (1,0) node {
\begin{tikzpicture}[scale=\tikztextscale]
\draw (0,0) node {\pgftext{T2-15:}};
\end{tikzpicture}
};
\draw (1,-1) node[anchor=north] {
\def\axnum{15}
\begin{tikzpicture}[xscale=\epspicscaledoublewidex,yscale=\epspicscaleycc]
\axgrid{\axnum};
\foreach \x in {1,...,16}
\draw (\x) node[littlenode-eps,anchor=center]
	{\innertikzeps{\ax{\axnum}-{\x};}};
\end{tikzpicture}
};
\draw (4.15,0) node {
\begin{tikzpicture}[scale=\tikztextscale]
\draw (0,0) node {\pgftext{T2-16:}};
\end{tikzpicture}
};
\draw (4.15,-1) node[anchor=north] {
\def\axnum{16}
\begin{tikzpicture}[xscale=\epspicscaledoublewidex,yscale=\epspicscaleycc]
\axgrid{\axnum};
\foreach \x in {1,...,18}
\draw (\x) node[littlenode-eps,anchor=center]
	{\innertikzeps{\ax{\axnum}-{\x};}};
\end{tikzpicture}
};
\draw (7.5,0) node {
\begin{tikzpicture}[scale=\tikztextscale]
\draw (0,0) node {\pgftext{T2-17:}};
\end{tikzpicture}
};
\draw (7.5,-1) node[anchor=north] {
\def\axnum{17}
\begin{tikzpicture}[xscale=\epspicscaledoublesuperwidex,yscale=\epspicscaleycc]
\axgrid{\axnum};
\foreach \x in {1,...,22}
\draw (\x) node[littlenode-eps,anchor=center]
	{\innertikzeps{\ax{\axnum}-{\x};}};
\end{tikzpicture}
};
\draw (8.75,-63) node {
\begin{tikzpicture}[scale=\tikztextscale]
\draw (0,0) node {\pgftext{T2-18:}};
\end{tikzpicture}
};
\draw (8.75,-64) node[anchor=north] {
\def\axnum{18}
\begin{tikzpicture}[xscale=\epspicscalex,yscale=\epspicscaleyc]
\axgrid{\axnum};
\foreach \x in {1,...,4}
\draw (\x) node[littlenode-eps,anchor=center]
	{\innertikzeps{\ax{\axnum}-{\x};}};
\end{tikzpicture}
};
\draw (7.75,-63) node {
\begin{tikzpicture}[scale=\tikztextscale]
\draw (0,0) node {\pgftext{T2-19:}};
\end{tikzpicture}
};
\draw (7.75,-64) node[anchor=north] {
\def\axnum{19}
\begin{tikzpicture}[xscale=\epspicscalex,yscale=\epspicscaleycc]
\axgrid{\axnum};
\foreach \x in {1,...,7}
\draw (\x) node[littlenode-eps,anchor=center]
	{\innertikzeps{\ax{\axnum}-{\x};}};
\end{tikzpicture}
};
\draw (6.75,-63) node {
\begin{tikzpicture}[scale=\tikztextscale]
\draw (0,0) node {\pgftext{T2-20:}};
\end{tikzpicture}
};
\draw (6.75,-64) node[anchor=north] {
\def\axnum{20}
\begin{tikzpicture}[xscale=\epspicscalex,yscale=\epspicscaleyc]
\axgrid{\axnum};
\foreach \x in {1,...,10}
\draw (\x) node[littlenode-eps,anchor=center]
	{\innertikztikz{\ax{\axnum}-{\x};}};
\end{tikzpicture}
};
\draw (5.75,-63) node {
\begin{tikzpicture}[scale=\tikztextscale]
\draw (0,0) node {\pgftext{T2-21:}};
\end{tikzpicture}
};
\draw (5.75,-64) node[anchor=north] {
\def\axnum{21}
\begin{tikzpicture}[xscale=\epspicscalex,yscale=\epspicscaleyc]
\axgrid{\axnum};
\foreach \x in {1,...,8}
\draw (\x) node[littlenode-eps,anchor=center]
	{\innertikztikz{\ax{\axnum}-{\x};}};
\end{tikzpicture}
};
\draw (4.75,-63) node {
\begin{tikzpicture}[scale=\tikztextscale]
\draw (0,0) node {\pgftext{T2-22:}};
\end{tikzpicture}
};
\draw (4.75,-64) node[anchor=north] {
\def\axnum{22}
\begin{tikzpicture}[xscale=\epspicscalex,yscale=\epspicscaleyc]
\axgrid{\axnum};
\foreach \x in {1,...,10}
\draw (\x) node[littlenode-eps,anchor=center]
	{\innertikztikz{\ax{\axnum}-{\x};}};
\end{tikzpicture}
};
\draw (3.75,-57) node {
\begin{tikzpicture}[scale=\tikztextscale]
\draw (0,0) node {\pgftext{T2-23:}};
\end{tikzpicture}
};
\draw (3.75,-58) node[anchor=north] {
\def\axnum{23}
\begin{tikzpicture}[xscale=\epspicscalex,yscale=\epspicscaleycc]
\axgrid{\axnum};
\foreach \x in {1,...,6}
\draw (\x) node[littlenode-eps,anchor=center]
	{\innertikztikz{\ax{\axnum}-{\x};}};
\end{tikzpicture}
};
\draw (2.75,-57) node {
\begin{tikzpicture}[scale=\tikztextscale]
\draw (0,0) node {\pgftext{T2-24:}};
\end{tikzpicture}
};
\draw (2.75,-58) node[anchor=north] {
\def\axnum{24}
\begin{tikzpicture}[xscale=\epspicscalex,yscale=\epspicscaleycc]
\axgrid{\axnum};
\foreach \x in {1,...,6}
\draw (\x) node[littlenode-eps,anchor=center]
	{\innertikztikz{\ax{\axnum}-{\x};}};
\end{tikzpicture}
};
\draw (1.5,-51) node {
\begin{tikzpicture}[scale=\tikztextscale]
\draw (0,0) node {\pgftext{T2-25:}};
\end{tikzpicture}
};
\draw (1.5,-52) node[anchor=north] {
\def\axnum{25}
\begin{tikzpicture}[xscale=\epspicscalex,yscale=\epspicscaleycc]
\axgrid{\axnum};
\foreach \x in {1,...,10}
\draw (\x) node[littlenode-eps,anchor=center]
	{\innertikzeps{\ax{\axnum}-{\x};}};
\end{tikzpicture}
};
\draw (.25,-51) node {
\begin{tikzpicture}[scale=\tikztextscale]
\draw (0,0) node {\pgftext{T2-26:}};
\end{tikzpicture}
};
\draw (.25,-52) node[anchor=north] {
\def\axnum{26}
\begin{tikzpicture}[xscale=\epspicscalex,yscale=\epspicscaleycc]
\axgrid{\axnum};
\foreach \x in {1,...,11}
\draw (\x) node[littlenode-eps,anchor=center]
	{\innertikzeps{\ax{\axnum}-{\x};}};
\end{tikzpicture}
};
\end{tikzpicture}

\restoregeometry

\begin{itemize}

%%%%%%

\begin{multicols}{2}

\item[T2-27]: %1
$\cb{\ingt{t2stuff-item27-1.eps}} \xlongequal{\phantom{x}} \cb{\ingt{t2stuff-item27-2.eps}}$
\vspace{18pt}

\item[T2-28]: %2
$\cb{\ingt{t2stuff-item28-1.eps}} \xlongequal{\phantom{x}} \cb{\ingt{t2stuff-item28-2.eps}}$
\vspace{18pt}

\item[T2-29]: %3
$\cb{\scalebox{\hscsqz}{\ingt{t2stuff-item29-1.eps}}}  \xlongequal{\phantom{x}}   \text{id} \big( \cb{\scalebox{\hscsqz}{\ingt{t2stuff-item29-2.eps}}} \big)$

\columnbreak

\item[T2-30]: %4
$\left(\cb{\ing{t2stuff-item30-1.eps}}\right) \circ (i_t)   \xlongequal{\phantom{x}}  (i_t) \circ \left( \cb{\ing{t2stuff-item30-2.eps}} \right)$
\vspace{10pt}

\item[T2-31]: %5
$\cb{\ing{t2stuff-item31-1.eps}}   \xlongequal{\phantom{x}}   \cb{\ing{t2stuff-item31-2.eps}}$
\vspace{10pt}

\item[T2-32]: %6
$\cb{\ing{t2stuff-item32-1.eps}}   \xlongequal{\phantom{x}}   \cb{\ing{t2stuff-item32-2.eps}}$

\end{multicols}

\item[T2-33]: %7
$\left(\cb{\ing{t2stuff-item33-1.eps}}\right) \circ (a_y)   \xlongequal{\phantom{x}}  (a_y) \circ \left( \cb{\ing{t2stuff-item33-2.eps}} \right)$ 
\vspace{8pt}

\item[T2-34]: %8
$\left(\cb{\ing{t2stuff-item34-1.eps}}\right) \circ (a_y)   \xlongequal{\phantom{x}}  (a_y) \circ \left( \cb{\ing{t2stuff-item34-2.eps}} \right)$
\vspace{8pt}

\item[T2-35]: %9
$\cb{\ing{t2stuff-item35-1.eps}}  \xlongequal{\phantom{x}}   \text{id} \left(\cb{\ing{t2stuff-item35-2.eps}}\right)$
\vspace{8pt}

\item[T2-36]: %10
$\left(\cb{\ing{t2stuff-item36-1.eps}}\right) \circ (pi_r)   \xlongequal{\phantom{x}}  (pi_r) \circ \left( \cb{\ing{t2stuff-item36-2.eps}} \right)$
\vspace{8pt}

\item[T2-37]: %11
$\left(\cb{\ing{t2stuff-item37-1.eps}}\right) \circ (sw)   \xlongequal{\phantom{x}}  (sw) \circ \left( \cb{\ing{t2stuff-item37-2.eps}} \right)$
\vspace{8pt}

\item[T2-38]: %12
$\cb{\ing{t2stuff-item38-1.eps}}   \xlongequal{\phantom{x}}   \cb{\ing{t2stuff-item38-2.eps}}$
\vspace{8pt}

\item[T2-39]: %13
$\left(\cb{\ing{t2stuff-item39-1.eps}}\right) \circ (dw_r)   \xlongequal{\phantom{x}}  (dw_r) \circ \left( \cb{\ing{t2stuff-item39-2.eps}} \right)$
\vspace{8pt}

\item[T2-40]: %14
$\left(\cb{\ing{t2stuff-item40-1.eps}}\right) \circ (p_l)   \xlongequal{\phantom{x}}  (p_l) \circ \left( \cb{\ing{t2stuff-item40-2.eps}} \right)$
\vspace{8pt}

\item[T2-41]: %15
$\left(\cb{\ing{t2stuff-item41-1.eps}}\right) \circ (p_m)   \xlongequal{\phantom{x}}  (p_m) \circ \left( \cb{\ing{t2stuff-item41-2.eps}} \right)$
\vspace{8pt}

\item[T2-R]: z-flip of [T2-27]; y-flip of [T2-29], [T2-30], [T2-31], and [T2-33]; x-flip of [T2-35], [T2-36], [T2-37], [T2-38], and [T2-40]; and x-flip, y-flip, and xy-flip of [T2-39].

\end{itemize}

\nid In these axioms, ``$i_t$'' refers to the top identity [T1-10], ``$a_y$" to the vertical associator [T1-12], ``$pi_r$" to the right identity pass [T1-22], ``$sw$" to the switch [T1-17], ``$dw_r$" to the dewhisker right [T1-15], ``$p_l$" to the pass left [T1-18], and ``$p_m$" to the pass middle [T1-20].

%\end{description}

%\end{definition}

\begin{remark}
Our pictograms for encoding the data and axioms of a tricategory are reminiscent of a 3-dimensional version of Kapranov and Voevodsky's 2-dimensional hieroglyphs for monoidal 2-categories~\cite{kapvoev}, and indeed these notions are related, as follows.  The paper~\cite{kapvoev} describes the notion of a monoidal (strict) 2-category, but it omits two pieces of data (the relationships of the identity morphism $\id_{a \otimes b}$ on a tensor of two objects to the tensor $\id_a \otimes b$ and to the tensor $a \otimes \id_b$) and thirteen axioms.  We refer to the notion that includes these pieces of data and axioms as a ``complete KV monoidal 2-category".  We observe that a complete KV monoidal 2-category is an especially rigid kind of tricategory with one object.  (The additional pieces of data for a complete KV monoidal 2-category ensure the existence of data [T1-13] and [T1-14] in the associated tricategory, and the additional axioms are necessary to ensure that the following tricategory axioms hold: [T2-5], its three reflections, [T2-6], its reflection, [T2-9], its reflection, [T2-10], [T2-18], its reflection, [T2-35], and its reflection.)

The above definition is also closely related to Batanin's globular operad approach to weak $n$-categories~\cite{batanin,leinster}, and to Gurski's algebraic models for tricategories~\cite{gurski}---Gurski moreover provides an extensive investigation of coherence properties for definitions of tricategories.
\end{remark}

\subsection{Bicategories in $\CAT$ are tricategories} \label{app-bicatincataretricat}

In Section~\ref{sec-catin2cat} we noted that a category object in the 2-category of categories has an underlying bicategory.  In this section, we prove that every bicategory object in the 2-category of categories has an underlying tricategory.  There are a number of analogous results in the literature associating globular higher categories to cubical-type higher categories---in particular, Shulman~\cite{shulman-constructing} constructs a monoidal bicategory from a monoidal double category, Garner--Gurski~\cite{garnergurski} construct a Gordon--Powers--Street-style tricategory~\cite{gps} from a multi-object monoidal double category (a `locally cubical bicategory'), and Grandis--Par\'e~\cite{grandispare-cubicalglobular} construct a GPS-style tricategory from a weak 3-cubical category.

\begin{theorem} \label{thm-bicatincataretricat}
Given a bicategory object $C=(C_0,C_1,C_2)$ in the 2-category $\CAT$ of categories, the quadruple $\{\obj C_0, \obj C_1, \obj C_2, (\obj C_1) \times_{C_1} (\mor C_2) \times_{C_1} (\obj C_1)\}$ forms the 0-, 1-, 2-, and 3-cells $\{T_0, T_1, T_2, T_3\}$ of a tricategory.  Here $(\obj C_1) \times_{C_1} (\mor C_2) \times_{C_1} (\obj C_1)$ is the set of morphisms $\phi$ of the category $C_2$ whose source $s(\phi)$ and target $t(\phi)$ are identity morphisms in $C_1$.
\end{theorem} 
\begin{proof}
We construct the data of a tricategory and then verify the axioms.  The 1-data items [T1-1] through [T1-9] of a tricategory are given by the value of the functors [B1-1] through [B1-9] on objects.  The 1-data items [T1-10] through [T1-27] are given by the value of the natural transformations [B2-1] through [B2-18] on objects.  The remaining data are as follows:
\begin{description}
\item[T1-28] The 1-datum item [T1-28], $i_z: T_2 \ra T_3$, is the identity of the category $C_2$.  
\item[T1-29] The 1-datum item [T1-29], $m_z: T_3 \times_{T_2} T_3 \ra T_3$, is the composition in the category $C_2$.  
\item[T1-30] The 1-datum item [T1-30], $f_b: T_3 \times_{T_1} T_2 \ra T_3$, is given by $f_b(\gamma \times a) = m_v(\gamma \times \id_a)$, where $m_v$ is the vertical composition functor [B1-4] and $\id_a$ is the identity morphism on the object $a$ in the category $C_2$.  
\item[T1-31] The 1-datum [T1-31] is similarly given by $f_t(a \times \gamma) = m_v(\id_a \times \gamma)$.  
\item[T1-32] The 1-datum [T1-32], $h_r: T_3 \times_{T_0} T_1 \ra T_3$, is given by $h_r(\gamma \times u) = w_r(\gamma \times \id_u)$, where $w_r$ is the right whisker functor [B1-5] and $\id_u$ is the identity morphism on the object $u$ in the category $C_1$.  
\item[T1-33] The 1-datum [T1-33] is similarly given by $h_l(u \times \gamma) = w_l(\id_u \times \gamma)$.  
\end{description}

The axioms [T2-1] through [T2-26] follow immediately from the axioms [B3-1] through [B3-26].  The remaining axioms are checked as follows:
\begin{description}
\item[T2-27] Axiom [T2-27] is satisfied because $C_2$ is a category---in particular the identity is strict. %%
\item[T2-28] Axiom [T2-28] is satisfied because $C_2$ is a category---in particular the composition is strictly associative. %%
\item[T2-29] Axiom [T2-29] is satisfied because [B1-4] is a functor---in particular it takes the identity to the identity. %%
\item[T2-30] Axiom [T2-30] is satisfied because [B2-1] and [B2-2] are natural transformations, using the fact that [B1-3] is a functor, therefore takes the identity to the identity. %%
\item[T2-31] Axiom [T2-31] is satisfied because [B1-4], i.e. $m_v$, is a functor---more specifically for any object $a \in \obj C_2$, the composites $C_2 \times_{C_1} \{a\} \ra C_2 \times_{C_1} C_2 \xra{m_v} C_2$ and $\{a\} \times_{C_1} C_2 \ra C_2 \times_{C_1} C_2 \xra{m_v} C_2$ are functors; in particular they take composition to composition. %%
\item[T2-32] Axiom [T2-32] is satisfied because [B1-4] is a functor---we have $m_z(f_t(s(\delta) \times \gamma) \times f_b(\delta \times t(\gamma))) = m_z(m_v(\id_{s(\delta)} \times \gamma) \times m_v(\delta \times \id_{t(\gamma)})) =
m_v(m_z(\id_{s(\delta)} \times \delta) \times m_z(\gamma \times \id_{t(\gamma)})) = m_v(\delta \times \gamma) = m_v(m_z(\delta \times \id_{t(\delta)}) \times m_z(\id_{s(\gamma)} \times \gamma)) = m_z(m_v(\delta \times \id_{s(\gamma)}) \times m_v(\id_{t(\delta)} \times \gamma)) = m_z(f_b(\delta \times s(\gamma)) \times f_t(t(\delta) \times \gamma))$. %%
\item[T2-33] Axiom [T2-33] is satisfied because [B2-3] is a natural transformation, using the fact that [B1-4] is a functor, therefore takes the identity to the identity. %%
\item[T2-34] Axiom [T2-34] is satisfied because [B2-3] is a natural transformation. %%
\item[T2-35] Axiom [T2-35] is satisfied because [B1-5] and [B1-6] are functors---in particular they take the identity to the identity. %%
\item[T2-36] Axiom [T2-36] is satisfied because [B2-13] and [B2-14] are natural transformations, using the fact that [B1-1], [B1-7], and [B1-8] are functors, therefore take the identity to the identity. %%
\item[T2-37] Axiom [T2-37] is satisfied because [B2-8] is a natural transformation, using the fact that [B1-5] and [B1-6] are functors, therefore take the identity to the identity. %%
\item[T2-38] Axiom [T2-38] is satisfied because [B1-5], i.e. $w_r$, and [B1-6], i.e. $w_l$, are functors---more specifically for any object $u \in \obj C_1$, the composites $C_2 \times_{C_0} \{u\} \ra C_2 \times_{C_0} C_1 \xra{w_r} C_2$ and $\{u\} \times_{C_0} C_2 \ra C_1 \times_{C_0} C_2 \ra C_2$ are functors; in particular they take composition to composition. %%
\item[T2-39] Axiom [T2-39] is satisfied because [B2-6] and [B2-7] are natural transformations, using the fact that [B1-5] and [B1-6] are functors, therefore take the identity to the identity. %%
\item[T2-40] Axiom [T2-40] is satisfied because [B2-9] and [B2-10] are natural transformations, using the fact that [B1-2] and [B1-9] are functors, therefore take the identity to the identity. %%
\item[T2-41] Axiom [T2-41] is satisfied because [B2-11] is a natural transformation, using the fact that [B1-9] is a functor, therefore takes the identity to the identity. \qedhere %%
\end{description}
\end{proof}

We note that the above construction of a tricategory does not depend on the fibrancy condition on a bicategory object.  However, besides this item, the construction and verification uses precisely the subset of the functoriality and naturality properties of bicategory objects that can be expressed using the 0-data of the underlying tricategory.

\appendix
\section*{Appendix. Table of data and axioms for internal bicategories} \label{app-table}

In Table~\ref{table-dataax}, we list a single pictorial abbreviation for each piece of data and each axiom in the definition of an internal bicategory; the meaning of these pictograms is described in Section~\ref{sec-internalbicats}. 

\begin{figure}[ht]
\begin{tikzpicture}[xscale=\tikztablescalex,yscale=\tikztablescaley]
\draw (0,0) node {0-data};
\draw (-.3,-1) node[anchor=north west] {
\def\tablenum{1}
\begin{tikzpicture}[scale=\tikztablescaleinner]
\tablegrid{\tablenum};
\foreach \x in {1,...,3}
\draw (\x) node[tablenode,anchor=west]
	{\innertikztikz{\table{\tablenum}-{\x};}};
\end{tikzpicture}
};
\draw (.95,0) node {1-data};
\draw (.65,-1) node[anchor=north west] {
\def\tablenum{2}
\begin{tikzpicture}[scale=\tikztablescaleinner]
\tablegrid{\tablenum};
\foreach \x in {1,...,9}
\draw (\x) node[tablenode,anchor=west]
	{\innertikztikz{\table{\tablenum}-{\x};}};
\end{tikzpicture}
};
\draw (2.05,0) node {2-data};
\draw (1.75,-1) node[anchor=north west] {
\def\tablenum{3}
\begin{tikzpicture}[xscale=\tikztablescaleinnerx,yscale=\tikztablescaleinnery]
\tablegrid{\tablenum};
\foreach \x in {1,...,2}
\draw (\x) node[tablenode-eps,anchor=west]
	{\innertikzeps{\table{\tablenum}-{\x};}};
\foreach \x in {3,...,18}
\draw (\x) node[tablenode,anchor=west]
	{\innertikztikz{\table{\tablenum}-{\x};}};
\end{tikzpicture}
};
\draw (4.3,0) node {3-axioms};
\draw (4,-1) node[anchor=north west] {
\def\tablenum{4}
\begin{tikzpicture}[xscale=\tikztablescaleinnerx,yscale=\tikztablescaleinnery]
\tablegrid{\tablenum};
\foreach \x in {1,2,3,6,7,8,9,12,13}
\draw (\x) node[tablenode-eps,anchor=west]
	{\innertikzeps{\table{\tablenum}-{\x};}};
\foreach \x in {4,5,10,11}
\draw (\x) node[tablenode,anchor=west]
	{\innertikztikz{\table{\tablenum}-{\x};}};
\foreach \x in {14,...,46}
\draw (\x) node[tablenode,anchor=west]
	{\innertikztikz{\table{\tablenum}-{\x};}};
\end{tikzpicture}
};
\end{tikzpicture}
\vspace*{-20pt}
\renewcommand{\figurename}{Table}
\caption{Abbreviated definition of an internal bicategory.} \label{table-dataax}
\end{figure}

We briefly describe the recursive algorithm by which is it possible to determine how many pieces of data and axioms there must be, and therefore to check that our list is complete.  A pictogram of the kind drawn in Table~\ref{table-dataax} may be transformed into another by performing one of the following four operations: taking a 0-cell and replacing it by a 1-cell; taking a 1-cell that is not part of a 2-cell and replacing it by an identity 1-cell; taking a non-identity 1-cell and replacing it by a 2-cell; taking a 2-cell and replacing it by an identity 2-cell.  Each of these operations will be considered to increase the complexity of the pictogram by 1 unit, and the inverse of such an operation will be considered to decrease the complexity by 1 unit.  We declare a single non-identity 2-cell to have complexity 0; every pictogram obtainable from such a 2-cell by the listed operations (and their inverses) is thereby assigned a unique \emph{complexity}.

Each pictogram corresponds to a piece of data or an axiom, for the internal bicategory $\cC=(\cC_0,\cC_1,\cC_2)$, that occurs either in the object of 0-cells $\cC_0$, or in the object of 1-cells $\cC_1$, or in the object of 2-cells $\cC_2$.  The \emph{weight} of a pictogram is equal to its complexity if the pictogram represents a piece of data or an axiom occurring in $\cC_2$; the weight is 1 more than the complexity if the pictogram represents a piece of data occurring in $\cC_1$; and the weight is 2 more than the complexity if the pictogram represents a piece of data occurring in $\cC_0$.  The set of weight 3 pictograms is exactly the set of complexity 3 pictograms, and this is the set of ``3-axioms" in the table.  The set of weight~2 pictograms is exactly the set of complexity 2 pictograms, and this set is the ``2-data" in the table.  The set of weight 1 pictograms is the set of complexity 1 pictograms, together with the identity 1-cell pictogram and the composite of two 1-cells pictogram (each of which is complexity 0 but weight 1); this set is the ``1-data" in the table.  The set of weight 0 pictograms contains the 2-cell, the 1-cell, and the 0-cell, and is the ``0-data" in the table.

Let $\fC(k)$ denote the number of pictograms of complexity $k$, and let $\fN(k)$ denote the number of pictograms of complexity $k-2$ whose horizontal width is exactly 1.  Note that $\fN(1) = 1$, $\fN(2) = 2$, and $\fN(k)$ is the $k$-th Fibonacci number for $k>2$.  By considering the ways of obtaining a pictogram of complexity $k$ and width $w$ from a pictogram of complexity less than $k$ and width $w-1$, we observe that $\fC(k) = \sum_{i=1}^{k+2} \fN(i)\fC(k-i)$ for $k > -2$.  From this we calculate that $\fC(3) = 46$, and indeed there are 46 3-axioms listed in Table~\ref{table-dataax}, which are in turn organized into 26 axiom groups in Section~\ref{sec-bicatin2cat}.  Proceeding further, note that the notion of internal bicategory in a 3-category would have 118 axioms.  These techniques predict, more generally, the number of axioms for a weak $k$-category internal to a strict $n$-category.  For instance, a $1$-category internal to a strict $n$-category will obey $\fF(n+3)$ axioms for $n>0$, where $\fF(i)$ is the $i$-th Fibonacci number.  Inductive numerology similarly implies that a tricategory internal to a 1-category will satisfy 74 axioms (and indeed Definition~\ref{def-tricat} contains 74 axioms organized into 41 groups), or that a tricategory internal to a 2-category will satisfy 231 axioms, or, finally and heaven forbid, that a tricategory internal to a 3-category will satisfy 725 axioms.

\subsection*{Acknowledgments}  We thank Arthur Bartels for the rewarding collaboration that motivated this paper, and Peter Teichner, Stephan Stolz, and Mike Shulman for inspiring and informative conversations.  The first author was partially supported by a Miller Research Fellowship, and the second author was partially supported by ERC Horizon 2020 grant 674978.

\bibliography{ib}

\bibliographystyle{plain}

\end{document}